\documentclass[reqno,11pt]{amsart}
\title{
Embedding calculus and  Vassiliev spectral sequence}

\author{Syunji Moriya
}

\address{University of Aizu, Tsuruga, Ikki-machi, Aizu-Wakamatsu City, Fukushima, 965-8580  Japan}
\email{s-moriya@u-aizu.ac.jp}

\usepackage{setspace}
\usepackage{amsmath,amsthm,amssymb,mathrsfs}
\usepackage{graphics}
\usepackage[matrix]{xy}
\usepackage{here}
\xyoption{all}
\usepackage{enumerate}

\theoremstyle{definition}
\newtheorem{defi}{Definition}[section]

\newtheorem{exa}[defi]{Example}
\newtheorem{rem}[defi]{Remark}
\theoremstyle{plain}
\newtheorem{prop}[defi]{Proposition}

\newtheorem{lem}[defi]{Lemma}

\newtheorem{thm}[defi]{Theorem}

\newtheorem{cor}[defi]{Corollary}

\newtheorem{conj}[defi]{Conjecture}

\newcommand{\eps}{\epsilon}

\newcommand{\CC}{\mathcal{C}}

\newcommand{\TT}{\mathcal{T}} 

\newcommand{\PK}{\mathcal{PK}} 

\newcommand{\TPhi}{\widetilde{\Phi}}

\newcommand{\MM}{\mathcal{M}}
\newcommand{\CH}{\mathrm{CH}}
\newcommand{\TOP}{\mathrm{Top}}

\newcommand{\Map}{\mathrm{Map}}

\newcommand{\SP}{\mathcal{SP}}

\newcommand{\CG}{\mathrm{Top}}

\newcommand{\tot}{\operatorname{Tot}}

\newcommand{\GG}{\mathsf{G}}
\newcommand{\TGG}{\tilde{\mathsf{G}}}

\newcommand{\RR}{\mathbb{R}}

\newcommand{\Sphere}{\mathbb{S}}

\newcommand{\FF}{\mathbb{F}}
\newcommand{\ZZ}{\mathbb{Z}}
\newcommand{\QQ}{\mathbb{Q}}

\newcommand{\LL}{\mathsf{L}}

\newcommand{\CECHF}{\check{\mathsf{C}}}

\newcommand{\kk}{\mathsf{k}}

\newcommand{\PP}{\mathrm{P}} 

\newcommand{\NN}{\mathcal{N}} 

\newcommand{\GP}{\widetilde{\mathrm{P}}} 

\newcommand{\PPP}{\mathrm{P}'}

\newcommand{\KKK}{\mathsf{K}} 

\newcommand{\FFF}{\mathscr{F}}

\newcommand{\WW}{\mathcal{W}}
\newcommand{\VV}{\mathcal{V}}

\newcommand{\SSS}{\mathcal{S}}

\allowdisplaybreaks[1]

\begin{document}
\begin{abstract}
Vassiliev spectral sequence and  Sinha spectral sequence are both related to cohomology of the space of long knots $\RR\to \RR^3$. Although they have different origins, the Vassiliev $E_1$-page and the Sinha $E_2$-page are isomorphic (up to a degree shift). In this paper, we prove that they have isomorphic $E_\infty$-pages  if the coefficient ring is a field. Together with degeneracy of the Sinha sequence, this implies that the Vassiliev sequence degenerates at $E_1$-page over $\QQ$ including the non-diagonal part. 
Our result also implies that for any coefficient field, the space of finite type $n$ knot invariants is isomorphic to the space of weight systems of weight $\leq n$ if and only if the parts of the Sinha sequence of  bidegree $(-i,i)$  degenerate at $E_2$ for $i\leq 2n$. For the construction of the isomorphism, we use  a variant of Thom space model which was introduced in the author's previous paper and captures embedding calculus of the knot space in terms of fat diagonals. As a byproduct of the construction, we  give a partial computation on differentials of  unstable versions of the Vassiliev sequence which converge to finite dimensional approximations of the knot space.
\end{abstract}
\maketitle
\tableofcontents
\section{Introduction}
A long knot in $\RR^d$ is a smooth embedding $\RR\to\RR^d$ which coincides with a fixed linear embedding outside  a compact set. 
Let $K_d$ be the space of long knots in $\RR^d$ with $C^\infty$ topology. This paper 
concerns two spectral sequences for cohomology of  $K_3$. One is Vassiliev spectral sequence introduced in \cite{vassiliev}, using the Alexander duality and simplicial resolutions of  finite dimensional approximations of the  space of singular long knots.  The diagonal part  of this sequence is related to the finite type knot invariants.
The other is  Sinha spectral sequence, which originates in Goodwillie-Weiss embedding calculus, a powerful framework to deal with the homotopy type of embedding spaces in terms of configuration spaces (see \cite{goodwillie,GK, GW, weiss}). Sinha \cite{sinha, sinha1} gave a cosimplicial model encoding  the embedding calculus of the knot space and defined his  sequence as Bousfield-Kan type cohomology spectral sequence associated to the cosimplicial model. The cosimplicial model is closely related to the little disks operad,  which has promoted study of computation of the Sinha sequence. \\
\indent While the definitions of these two sequences are very different, Turchin \cite{turchin} proved that the Vassiliev $E_1$-page and Sinha $E_2$-page are isomorphic as bigraded modules up to degree shift.   So  it is natural to expect some relation between these sequences for the later pages. 
In this paper, we prove the following.
\begin{thm}\label{TE_infinity} Let $\kk$ be a field. 
\begin{enumerate}
\item The $E_\infty$ pages of the Vassiliev spectral sequence ${}^V\!  E_r$ and the Sinha spectral  sequence ${}^S\! E_r$  for the space $K_3$ of long knots in $\RR^3$, over $\kk$, are isomorphic as  (abstract) bigraded $\kk$-vector spaces with the standard degree shift  $(-p,q)\leftrightarrow (q-3p, 2p)$ given in \cite{turchin}.
\item The Vassiliev spectral sequence   over $\kk$ degenerates at the first page if and only if the Sinha spectral sequence  over $\kk$ degenerates at the second page. 
\end{enumerate}
\end{thm}
See Definitions \ref{Dsinha}, \ref{DVassiliev} and \ref{Dunreduced_Vassiliev}. We remark that  the  part of the Sinha sequence of a bidegree which is not of the form  $(q-3p,2p)$ for any integers $p,q$, is zero from $E_1$.
Since the sequences are finite dimensional in each bidegree, second part of the theorem follows from first part via the result of \cite{turchin}.

Since degeneracy of the Sinha sequence is already known, see \cite{LTV, moriya, tsopmene},  we can solve affirmatively a remaining part of an old conjecture of Vassiliev which states that his sequence would degenerate at $E_1$-page: 
\begin{cor}[\cite{kontsevich1} for the diagonal part]
The Vassiliev spectral sequence ${}^V\! E_r$ degenerates at $E_1$-page for $\kk=\QQ$.\hfill\qedsymbol
\end{cor}
 Vassiliev proposed the conjecture, motivated by analogy between the discriminant of the knot space and affine subspace  arrangements (see \cite{vassiliev2}). The diagonal part (degree 0 part) has been paid much attention since it is related to the finite type invariants. For the diagonal part and for $\kk=\QQ$, Kontsevich \cite{kontsevich1} proved the conjecture using the technique of configuration space integrals. In an unpublished version of the paper \cite{kontsevich1}, he seemed to outline a proof of degeneracy for the non-diagonal part and for the case of higher codimension, see \cite{vassiliev1,vassiliev2} (the Vassiliev and Sinha sequences are also defined for $K_d$ with $d\geq 4$).  Along this line, Cattaneo--Cotta-Ramusino--Longoni \cite{CCL} constructed non-trivial cocycles of the knot space $K_d$ for $d\geq 4$,  but a complete proof of the conjecture did not appear. Then, Lambrechts-Turchin-Voli\'c \cite{LTV} proved the conjecture for the case of higher codimension.  Their method is as follows: The {\em space of long knots modulo immersions} $\bar K_d$ is defined as the homotopy fiber of the inclusion from $K_d$ to the space of immersions which coincide with the fixed embedding outside  a compact set.  Sinha's cosimplicial model for $\bar K_d$ is related to the little $d$-disks operad (see \cite{sinha1}). Using the (relative) formality of the  operads \cite{kontsevich1, LV},  one can prove that Sinha's cohomological spectral sequence $\bar E_r$ for the space of long knots  modulo immersions degenerates at $E_2$-page. By \cite{turchin}, the map of spectral sequence $^SE_r\to \bar E_r$ induced by the canonical map $\bar K_d\to K_d$ gives a monomorphism on the $E_2$-page. Therefore,  $^SE_r$ also degenerates at $E_2$-page. Since both of $^SE_r$ and $^VE_r$ converge to $H^*(K_d)$  for $d\geq 4$, Vassiliev's degeneracy is obtained. In the case of $d=3$, the degeneracy of Sinha's sequence (for knots modulo immersion) and the injectivity of the map $^SE_r\to \bar E_r$ still hold, but the convergence to $H^*(K_3)$ is unknown and the relation between the $E_\infty$-pages was unclear, which is uncovered by Theorem \ref{TE_infinity}.  \\
\indent Recently, Marino-Salvatore \cite{MS} proposed a computer-aided proof of non-degeneracy of the Sinha sequence for $\kk=\FF_2$ using a combinatorial model of the cosimplicial space which is derived from the Fox-Neuwirth cell structure (see  \cite{marino}). With this and Theorem \ref{TE_infinity} (2), we can disprove the conjecture of Vassiliev for the coefficient. This is very interesting in analogy with the affine subspace arrangements which   homotopically split into a wedge sum of filtered quotients, and whose associated spectral sequence  degenerates at any coefficient.\\
\indent We shall look at application to the finite type invariants.  A knot invariant $\pi_0(K_3)\to \kk$ over a commutative ring $\kk$  is called {\em of (finite) type $n$} if a natural extension of the invariant to singular knots vanishes on the singular knots  with at least $n+1$ double points (see e.g. \cite{bar-natan, kosanovic} for the precise definition). This notion  was introduced in \cite{vassiliev}, see also Goussarov \cite{gusarov} and Birman-Lin \cite{BL}. While the completeness of the finite type invariants is still open,  they include many examples of known knot invariants. The relation between embedding calculus and finite type invariants has been extensively studied, see e.g. Voli\'c \cite{volic}, Budney-Conant-Koytcheff-Sinha \cite{BCKS} and Kosanovi\'c \cite{kosanovic}. The Vassiliev sequence is related to finite type invariants as follows. Let $\VV_n$ be the space of finite type $n$ invariants over $\kk$. The space $^VE_\infty^{-n,n}$ is isomorphic to $\VV_n/\VV_{n-1}$ via the Alexander duality (see \cite{BL}, where $\kk=\QQ$ is assumed, but the arguments there work for any coefficient).  A  feature of the finite type invariants over $\QQ$ is that they correspond to combinatorial objects called weight systems.  The {\em space $W_n$ of weight systems of weight $n$} is defined as the dual of the $\kk$-module generated by chord diagrams with $n$ chords modulo the $1T$ and $4T$ relations. The space $W_n$ is isomorphic to $^VE_1^{-n,n}$ so we have an inclusion $\VV_n/\VV_{n-1}\subset W_n$ (see \cite{bar-natan}). The correspondence in $\kk=\QQ$ is a consequence of the above degeneracy of the diagonal part of the Vassiliev sequence. \\
\indent For the diagonal part, we have the slightly finer version of Theorem \ref{TE_infinity} (2) as follows.
\begin{thm}\label{Cdegree0}
If $\kk$ is a principal ideal domain, we have $\VV_n/\VV_{n-1}= W_n$ (under the inclusion of the  previous paragraph) if and only if the diagonal part $\oplus_{i\leq 2n}{}^S\!E^{-i,i}_r$ of the Sinha spectral sequence (for $K_3$) up to the bidegree $(-2n,2n)$ degenerates at $E_2$-page.  
\end{thm}

Combining this with an action of Grothendieck-Teichm\"uller group on $p$-completed little disks operads constructed by  Boavida de Brito-Horel \cite{BH,BH1} we have  the following (see section \ref{Sproof} for the proof). 
\begin{cor}\label{Cinvariants}
 Let $p$ be a prime. For $\kk=\ZZ_{(p)}\ or\ \FF_p$, we have $\VV_n/\VV_{n-1}= W_n$ if $n\leq p+1$, where $\ZZ_{(p)}$ and $\FF_p$ are the ring of $p$-local integers and the field of $p$ elements, respectively.
\end{cor}
Our results are cohomological. On the other hand, there are {\em homotopical} studies on the finite type invariants and embedding calculus. Kosanovi\'c \cite{kosanovic} proved that the {\em additive} finite type invariants correspond to the indecomposable Feynman diagrams if the homotopy embedding calculus spectral sequence degenerates at $E_2$ (among other things), and using this, Boavida de Brito-Horel \cite{BH} proved
a result similar to Corollary \ref{Cinvariants} for the additive  ones  (see section \ref{Sproof} for the comparison). The result of Marino-Salvatore suggests possibility of non-degeneracy of the diagonal part of the Sinha sequence. 
Since Theorem \ref{Cdegree0} contains `only if' statement whose homotopical counterpart is unknown, it might be  interesting to examine the relation between the finite type invariants and weight systems in positive characteristics.\\
\indent We shall explain the outline of the proof of Theorem \ref{TE_infinity}.  In \cite{moriya3}, the author constructed a Thom space model which is stably equivalent to the Sinha cosimplicial model, (or more precisely,  punctured knot model) by applying Atiyah duality to the configuration spaces to transfer the embedding calculus structure to the Thom spaces over fat diagonals of products of disks and use this model to compute some differentials of the Sinha sequence of codimension one. The Vassiliev sequence is defined via the Alexander duality, so both of the Thom space model and the sequence sit at the opposite side of the knot space in the similar dualities and 
actually, they look alike. 
A resolution of the Thom space model  consists of  subspaces labeled with graphs. These subspaces are Thom spaces of the diagonals which equate the components of the products corresponding to the vertices  in a common connected component of the labeling graph.  Vassiliev's simplicial resolution (or resolvent) consists of pairs of a graph and a singular knot which may have multiple point(s), and can be considered as  an element of  a diagonal subspace of the product $(\RR^3)^{\RR}$. Therefore, by evaluating the singular ones in the resolvent at multiple points, we have a natural map from the resolvent to the geometric realization of the resolution of the Thom space model, which is expected to induce an isomorphism of the spectral sequences. The shift of bidegree will occur by change of a  filtration. This idea of proof is very simple, but we encounter some technical problems, which makes the paper lengthy (see subsection \ref{SSintuition} for further outline of the construction).\\
\indent The construction of the isomorphism in Theorem \ref{TE_infinity} is mostly given at space-level so we expect a natural compatibility among this isomorphism and the maps between the genuine cohomology of the  knot space $K_3$ and the spectral sequences. We state such a  possible relation as Conjecture \ref{Conj_compati}. This conjecture is interesting since  it implies that the additive finite type tower  and embedding calculus $\pi_0$ tower have isomorphic inverse limits.\\ 
\indent The convergence of  the Vassiliev and Sinha sequences is still largely unknown in dimension $3$. For example, in general, it is unclear whether the elements of the $E_\infty$-page give non-trivial elements of the genuine cohomology (except for the diagonal part, and see Sakai \cite{sakai1} for a related problem). This is related to another conjecture of Vassiliev concerning unstable versions of his sequence which are directly obtained from  finite dimensional approximations of the space of (singular) knots. As a byproduct of the construction of the isomorphism, we will give a partial computation of the differentials of the unstable sequences, see  subsection \ref{SSunstable}. We expect that  further analysis of the construction will lead to better understanding of the relation between the sequences and the knot space. \\
\indent The outline of this paper is as follows. In section \ref{Spreliminary}, we recall the punctured knot model which approximates the knot space in terms of embedding calculus. We also recall  definition of the Vassiliev spectral sequence. It is defined using the resolvent of (finite dimensional approximations of) the space of  singular knots. We also give a variant of the resolvent which is more suitable for our purpose. Most of the contents of this section are not essentially new. In section \ref{Stranslation}, we replace the punctured knot model in the embedding calculus with the Thom space model which is a functor consisting of Thom spaces of fat diagonals. This model is similar to the models given in \cite{moriya2,moriya3} but it is different from the one in \cite{moriya3}, in that we take tangent vectors into account here, and also different from the one in \cite{moriya2}, in that  we consider tangent vectors $v$ with $0<|v|<1$ instead of the unit vectors  and that we involve the boundary $|v|=0, 1$ in the duality here. We need point-set level compatibility of structure maps so we take care about parameters such as the width of a regular neighborhood. We establish a chain-level equivalence between the punctured knot model and Thom space model via a configuration space model, and prove that a truncated version of the Sinha spectral sequence is isomorphic to the sequence associated to the Thom space model. The outline of the constructions in the section is similar to \cite{moriya3} but the proofs are somewhat more involved so we give detailed proofs for the most part. In sections \ref{SU_1}, \ref{SU_2} and \ref{SU_3}, we connect the variant of the resolvent and a resolution of the Thom space model by zigzag of maps which induce isomorphisms of spectral sequences. For an outline of these three sections, see subsection \ref{SSintuition}. Most of the technical constructions and relatively long proofs of this paper are given in sections \ref{Stranslation}, \ref{SU_1}, and \ref{SU_2}. In section \ref{SU_3}, we make minor adjustments.    In the last section \ref{Sproof}, we give proofs of the results stated in Introduction, state a conjecture, and give a partial result on the unstable spectral sequence.  \\

\noindent {\bf Acknowledgement:} The author is  grateful to Keiichi Sakai for giving valuable comments on an earlier version of this paper. 
\subsection{Notation and terminology}\label{SSNT}

\begin{enumerate}
\item $\CG$ denotes the category of unpointed topological spaces and continuous maps and $\CG_*$ the corresponding pointed category. 
For a space $X\in \CG$, we let $X^*\in \CG_*$ denote the one-point compactification of $X$.
\item For an unpointed (resp. pointed) topological space $X$, $C_*(X)$ (resp. $\bar C_*(X)$) denotes the  singular chain complex (resp. the reduced  singular chain complex) with coefficients in a fixed commutative ring $\kk$. We denote by $H_*(X)$ (resp. $\bar H_*(X)$) the corresponding ordinary homology group.
\item Let $\CH_{\kk}$ be the category of chain complexes and chain maps over  $\kk$.

\item We let $sc:[-\infty,\infty]\to [-1,1]$ denote the homeomorphism given by $sc(t)=\frac{2}{\pi}\arctan t$. Let $d_{sc}$ denote the distance on $[-\infty,\infty]$ given  by $d_{sc}(t_1,t_2)=|sc(t_1)-sc(t_2)|$.
\item For a finite set $S$, let $|S|$ be its cardinality.
\item Let $K$ be a (topological) poset, which is always regarded as a category. We denote  by $|K|$ the geometric realization of the nerve of $K$ which has the canonical topology (taking the topology of $K$ into account if $K$ is topological).  We often express an element $u$ of $|K|$ like $u=\tau_0x_0+\cdots +\tau_mx_m$ with some integer $m\geq 0$. In this kind of expression, we always assume that $x_i\in K$,  $\tau_i\in\RR$ and $\tau_i\geq 0$ for $0\leq i\leq m$,  and $\tau_0+\cdots +\tau_m=1$ and $\exists\  x_i\to x_{i+1}$ for $0\leq i\leq m-1$. For an element $u\in |K|$ with the expression $u=\tau_0x_0+\cdots +\tau_mx_m$ satisfying $\tau_m\not=0$ (in addition to the previous assumptions), we call the element $x_m$ the {\em last $K$-element of}  $u$ and denote the element by $\LL(u)$. Clearly, $\LL(u)$ always exists and  depends only on $u$.

\item We denote by $=_1$, $<_1$,$\leq_1, \dots$ etc., the relations between the first coordinates of elements of $\RR^3$. For example,  for two elements $x=(x_1,x_2, x_3), y=(y_1,y_2, y_3)\in \RR^3$ and a number $t\in \RR$, $x<_1y$ and $x=_1y$ mean $x_1<y_1$ and $x_1=y_1$, respectively, and  $x<_1 t$ means $x_1< t$.
\item Let $Y\in \CG$ and $X$  a quotient space of $Y$. To denote an element of $X$, we use the same symbol as its representative in $Y$ as long as it does not cause confusion. 
\end{enumerate}
\section{Preliminary}\label{Spreliminary}

\subsection{Punctured knot model $\PK$ and cosimplicial model $\SSS^\bullet$ }\label{SPK}
For technical reasons, we mainly deal with the punctured knot model defined below instead of the  cosimplicial model.  Throughout the paper, $n$ denotes a positive integer.
\begin{defi}\label{Dpartition0}
A {\em partition $P$ of } $[n+1]=\{0,1,\dots, n+1\}$ is a set of subsets of $[n+1]$ satisfying the following conditions.
\begin{enumerate}
\item $\cup_{\alpha \in P}\alpha =[n+1]$.
\item Each element of  $P$ is non-empty.
\item If $\alpha, \beta \in P$, either of $\alpha=\beta$ or $\alpha\cap \beta=\emptyset$ holds.
\item For each element  $\alpha \in P$, if numbers $i,j,k$ satisfy $i<j<k$ and $i,k\in \alpha$, the number $j$ also belongs to $\alpha$. 
\item $ |P|\geq 2$, in other words the set consisting of the single element $[n+1]$ is not a partition. 
\end{enumerate} 
We call an element of  $P$ a {\em piece of} $P$. We regard a partition as a totally ordered set via the order induced by $[n+1]$.  Let $P^\circ\subset P$ denote the subset of pieces which are neither the minimum nor the maximum. For $i\in [n+1]$, we write $i\in P^\circ$ if there is a piece $\alpha\in P^\circ$ satisfying $i\in \alpha$. A partition $Q$ is said to be a {\em subdivision of} $P$ if $Q\not=P$ and each piece of $Q$ is contained in some piece of $P$. We let $\PP_n$ denote the category (or poset) of partitions of $[n+1]$. Its objects are the partitions of $[n+1]$. 
A unique non-identity morphism $P\to Q$ exists if and only if $Q$ is a subdivision of $P$. 
\end{defi}
\begin{exa}\label{Epartition}
The following sets are examples of objects of $\PP_4$:
\[
P=\{\{0\},\{1, 2\},\{3,4,5\}\}, \qquad Q=\{\{0\},\{1,2\}, \{3\}, \{4,5\}\}.
\]
  We have $\{0\}<\{1,2\}<\{3,4,5\}$ for the pieces of $P$. For the piece $\alpha=\{3,4,5\}$ of $P$, $\max\alpha=5$ and $\min\alpha=3$. We see that $|P|=3$, $|Q|=4$, $P^\circ=\{\{1,2\}\}$, $\min P=\{0\}$, $\max P=\{3,4,5\}$, and  $Q$ is a subdivision of $P$.  
\end{exa}

\begin{defi}\label{DPK}
 We define a functor $\PK:\PP_n\to \CG$ as follows.
For a partition $P=\{\alpha_0<\cdots <\alpha_{p+1}\}$,  $S_P\subset \{1,\dots, n+1\}$ denotes the set of minimum elements in each of $\alpha_1,\dots, \alpha_{p+1}$ (so $| S_P|=p+1$).  Let $D^3\subset \RR^3$ be the $3$-dimensional unit closed disk, and $Int(D^3)$ its interior. The space $\PK(P)$ is the space of embeddings 
\[
f:[0,1]-\bigcup_{i\in S_P}\left(\frac{4i-1}{4(n+2)}, \frac{4i+1}{4(n+2)}\right) \quad \longrightarrow \quad D^3
\]   such that
 $f(0)=(-1,0,0)$, $f(1)=(1,0,0)$,  $f'(0)=f'(1)=(1/2,0,0)$ and $f(t)\in Int (D^3)$ for $0<t<1$.
For a subdivision $Q$ of $P$, the map $\PK(P)\to \PK(Q)$ is the  restriction induced by the inclusion $S_P\subset S_Q$.
\end{defi}
The functor $\PK$ is called the {\em punctured knot model}. Actually, there is an obvious natural transformation from the punctured knot model in \cite{sinha} for $M=[0,1]^3$ to $\PK$ which induces a weak homotopy equivalence on each term. 
\begin{defi}\label{Dsinha}
\begin{enumerate}
\item Let $\Delta_n$ be the category whose objects are $[k]=\{0,1,\dots, k\}$ ($0\leq k\leq n$) and whose morphisms are the weakly order preserving maps. 
\item Let $\SSS^\bullet$ denote Sinha's cosimplicial model for $K_3$ given in Corollary 4.22 of \cite{sinha} where we set $M=[0,1]^3$. 
The space $\SSS^p$ is   a version of Fulton-MacPherson compactification of the ordered configuration space of $p$ points with  unit tangent vectors in $\RR^3$. We denote by $\SSS^{\leq n}$ the restriction of $\SSS^\bullet$ to $\Delta_n$. We call the Bousfield-Kan cohomology spectral sequence associated to $\SSS^\bullet$ the {\em Sinha (spectral) sequence} and denote it by ${}^S\! E^{**}_r$.
\item We define a functor $\mathcal{F}:\PP_n\to \Delta_n$ by $P\mapsto [| S_P|-1]$ and $(P\to Q)\mapsto ([| S_P|-1]\cong S_P\subset S_Q\cong [|S_Q|-1])$, where $\cong$ denotes the order-preserving bijection. 

\end{enumerate}
\end{defi}

\begin{thm}[\cite{sinha, sinha1}]\label{TSinha}
 There is a zigzag of natural transformations which give a weak homotopy equivalence at each term, between $\mathcal{F}^*\SSS^{\leq n}$ and $\PK$. \hfill \qedsymbol
\end{thm}

\subsection{The functors $\mathcal{F}$, $\CECHF$ and  total complex}
\begin{defi}\label{Dnormal} Let $\MM$ be one of the categories $\TOP_*$, $\CH_\kk$.
\begin{enumerate}
\item We identify the poset $\PP_n$ with the poset $Q_n$ of non-empty subsets of $\{1,\dots, n+1\}$ by $P\mapsto S_P$ (see the previous subsection). The functor $\mathcal{F}$ is regarded as a functor $Q_n\to \Delta_n$. 
\item For a functor $X: \Delta_n^{op}\to \CH_\kk$, the {\em normalization} $NX$ is the double complex given by
\[
N_pX=\left\{
\begin{array}{cc}
X_p/(\sum_{0\leq i\leq p-1}s_i(X_{p-1}))&\text{ if } p\leq n, \vspace{2mm}\\
0 & \text{ if } p>n,
\end{array}\right.
\]
where $s_i$ denotes the degeneracy map, with one of the differentials given by the signed sum of the face maps and the other by the original differential of $X$
\item For the functor $X=C^*(\SSS^{\leq n})$ obtained by taking the singular cochain of the restricted cosimplicial model $\SSS^{\leq n}$ in the termwise manner,  we call the spectral sequence associated to the normalization $N(X)$ with the filtration by simplicial degree, the {\em $n$-truncated   Sinha spectral sequence}. There is a canonical map from the $n$-truncated sequence to the full  Sinha sequence given in Definition \ref{Dsinha}. 

\item
Let $Fun(C,D)$ be the category of functors $C\to D$ and natural transformations between them. We denote  by $\CECHF: Fun({Q_n}^{op},\MM)\to Fun(\Delta_n^{op}, \MM)$  the left Kan extension along $\mathcal{F}:{Q_n}^{op}\to \Delta_n^{op}$, i.e. the left adjoint of the pullback by $\mathcal{F}$.  Concretely speaking,  it associates to a functor $Y:{Q_n}^{op}\to \MM$, a functor $\CECHF(Y):\Delta_n^{op}\to \MM$ which sends $[k]$ to $\bigsqcup_{f}Y(f([k]))$ where $f:[k]\to \{1,\dots, n+1\}$ runs through the weakly order-preserving maps.  When $\MM=\CH_\kk$,  the normalization  $N\CECHF(Y)$ consists of the summands labeled by monomorphisms $f$. 
\end{enumerate}
\end{defi}
The following lemma follows from the facts that $\mathcal{F}$ is left cofinal (see \cite{sinha}) and that the total complex of normalization is equivalent to the  homotopy colimit (see \cite{hirschhorn} for homotopy colimits).
\begin{lem}\label{Lcofinal}
Let $X:\Delta^{op}_n\to \CH_\kk$ be a functor. The map $\tot(N\CECHF(\mathcal{F}^*X))\to \tot(NX)$ between the total complexes,  induced by the unit of the adjoint pair $(\CECHF, \mathcal{F}^*)$, is a quasi-isomorphism. Here, $\tot$ denotes the total complex of a double complex.\hfill\qedsymbol
\end{lem}
The complex $N_k\CECHF(\mathcal{F}^*X)$ is $\bigoplus_f X([k])$, where $f$ runs through the order-preserving monomorphisms $[k]\to \{1,\dots, n+1\}$ and the map $N_k\CECHF(\mathcal{F}^*X)\to N_kX$ is identified with the map forgetting the label $f$, followed by the quotient map.
\subsection{Vassiliev spectral sequence}
In this section, we recall definition of the Vassiliev spectral sequence. For the whole theory of Vassiliev, see \cite{vassiliev, vassiliev1, BL}. We also introduce a variant of the resolvent (simplicial resolution) of the space of singular knots. 
\begin{defi}
\begin{enumerate}
\item A {\em classed configuration} is a pair $(A,B)$ where $A$ is a  possibly empty finite family of pairwise disjoint finite subsets of $\RR$ each of which has cardinality $\geq 2$, and $B$ is a possibly empty finite subset of $\RR$.  The {\em empty classed configuration} is such that $A$ and $B$ are empty. So  a non-empty classed configuration is such that at least one of $A$ and $B$ is non-empty. 
\item Let $F=(A,B)$ be a classed configuration. A {\em geometric point} of $F$ is an element of the subset
\[
(\cup_{a\in A}a)\cup B\quad \subset \quad \RR.
\] We call the cardinality of this subset the {\em cardinality} of $F$ and denote  by $|F|$ so the cardinality counts a geometric point only once even if it belongs to both of $B$ and an element of $A$. The {\em complexity} $c(F)$  is given by 
\[
c(F)= |\cup_{a\in A} a|-|A|+|B|
\]
\item Two classed configurations $F_1=(A_1,B_1), F_2=(A_2,B_2)$ are said to be {\em equivalent} if there is an orientation-preserving diffeomorphism $g : \RR \to \RR$ satisfying the following condition: The image of each element of $A_1$ by $g$ is equal to some element of $A_2$ and the map $A_1\ni a_1\mapsto g(a_1)\in A_2$ is a bijection $A_1\cong A_2$, and $g$ also induces a bijection $B_1\cong B_2$. 
\end{enumerate}
\end{defi}
\begin{rem}
A non-empty classed configuration $(A,B)$ is an $(A',b)$-configuration in the sense of \cite{vassiliev} for $A'=\{ |a|\mid a\in A\},\ b= |B|$. 
\end{rem}
To define the Vassiliev spectral sequence, we need  finite dimensional approximations of the knot space. Let $K_3'$ be the space of embeddings $f:\RR\to \RR^3$ satisfying $\langle f'(t), v_0\rangle\geq 1$ for all $t$ outside a compact set. Here $v_0=(1,1,1)/\,\sqrt[]{3}$ and $\langle-,-\rangle$ is the standard inner product on $\RR^3$. The inclusion $K_3\to K'_3$ is a weak homotopy equivalence so in this section, we also call an element of $K'_3$ a long knot.
\begin{defi}
\begin{enumerate}
\item We say a smooth map $f:\RR \to \RR^3$ {\em respects} a classed configuration $F=(A,B)$ if for each $a\in A$, $f$ maps all the points of $a$ to one point, which may differ for different elements of $A$, and $f$ satisfies $f'(t)=0$ for each $t\in B$. All maps respect the empty classed configuration.
\item Let $\tilde \Gamma_n$ be the space of  maps $f:\RR\to \RR^3$ of the form $f(t)=(P_1(t), P_2(t), P_3(t))$ where $P_i(t)$ is a polynomial of the following form
\[
t^{2n+1}+a_1t^{2n-1}+a_2t^{2n-2}+\cdots a_{2n}
\]
with $a_j\in \RR$, for $i=1,2,3$. The space  $\tilde \Gamma_n$ is naturally regarded as a $6n$-dimensional affine space. 
\item Let $V\subset \tilde \Gamma_N$ be an affine subspace and $F$ a classed configuration. We denote by $\chi(V,F)\subset V$  the subspace of maps  respecting $F$.
\end{enumerate}
\end{defi}
As $t$ tends to $\pm \infty$, the normalized tangent vector of $f(t)$ does to $(1,1,1)/\,\sqrt[]{3}$ for $f\in \tilde \Gamma_n$, so if $f$ is an embedding, we can regard it as a long knot.  The affine space $\tilde \Gamma_n$ is different from the corresponding space $\tilde \Gamma^d$ in \cite{vassiliev} but they give isomorphic (stable) spectral sequences, see Remark \ref{Rchoice_affine}.\\
\indent For  an odd number $w\geq 3$, we define an affine embedding $I:\tilde \Gamma_n\to \tilde \Gamma_N$, where $N=\frac{2wn+w-1}{2}$, by $I(f)(t)=f(t^w+t)$. The space $\tilde \Gamma_n$ contains maps with infinitely many singular points (see \cite{vassiliev}) so we perturb it in $\tilde \Gamma_N$ so that the result only contains  maps with finitely many singularities.
\begin{lem}[\cite{vassiliev}]\label{Lgeneral_posi}
Suppose $N\geq 12n+2$.
For almost any choice of a $6n$-dimensional affine  subspace $V$ of $\tilde \Gamma_N$, for any classed configuration $F$, we have the following statements. 
\begin{enumerate}
\item For almost any classed configuration $F'$ equivalent to $F$, the set $\chi(V,F')$ is   an affine subspace of codimension $3c(F)$ in $V$.
\item Suppose $c(F)=k\leq 2n$. Then in the set of all configurations equivalent to $F$, the  (semi-algebraic) subset of those configurations $F'$ for which  $\chi(V,F')=\emptyset$ is of codimension $\geq 6n-3k+1$, while the set of $F'$ such that the codimension of $\chi(V,F')$ in $V$ equals $3k-i,\ i\geq 1$ is a subset of codimension $\geq  i(6n-3k+i+1)$. In particular, when $k\leq (6n+1)/5$, the codimension of the set $\chi(V, F)$  is exactly equal to $3k$ (for any configuration $F$ with $c(F)=k$). 
\item Suppose $c(F)=k>2n$. Then in the set of configurations equivalent to $F$, the set of all $F'$ such that $\dim \chi(V,F')=l\geq 0$ is of codimension  $\geq (l+1)(3k-6n+l)$. In particular, the set of all $F'$ such that $\chi(V,F')\not=\emptyset$ is of codimension $\geq 3(k-2n)$ and is empty when $k>6n$.
 
\end{enumerate}
Here, we understand that  a subset of codimension larger than the dimension of the whole space is   empty. 
\end{lem}
\begin{proof}
This is standard argument about transversality so the proof is omitted in \cite{vassiliev}. We record an outline of the proof for the reader's convenience. A variant of Vandermonde determinant $\det (VD(t_1,\dots,t_N))$ where
\[VD(t_1,\dots, t_N)=
\begin{pmatrix}
R(t_1)\\
\vdots \\
R(t_N)\\
R'(t_1)\\
\vdots \\
R'(t_N)
\end{pmatrix},\qquad R(t)=(1,t,\dots, t^{2N-1}),
\]
is non-zero if $t_1,\dots,t_N$ are pairwise distinct, since the evaluation map
$\tilde \Gamma_N\ni f\mapsto (f(t_i), f'(t_i))_i\in \RR^{6N}$ is bijective.\\
\indent Put $c(F)=k$. Let $C_F$ be the space of classed configurations equivalent to $F$. We define a map $\delta:\tilde \Gamma_N\times C_F\to \RR^{3k}$ as follows. Let $F'=(A,B)\in C_F$. Write $A=\{\{t_{i1},\dots, t_{i,a_i}\}\mid 1\leq i\leq l\},\ B=\{p_1,\dots, p_b\}$ (so $\sum_ia_i-l+b=k$). We set
\[
\delta(f,F')=\bigl((f(t_{i1})-f(t_{i2}), \dots, f(t_{i1})-f(t_{i,a_i}))_{1\leq i\leq l}, f'(p_1),\dots, f'(p_b)\bigr)
\]
Clearly, $\chi(V, F')=\{f\in V\mid \delta(f,F')=0\}$.
Let $\Gamma_N^0$ be the space of 3-tuples of polynomials at most of degree $2N-1$, regarded as a vector space in the standard manner, and we set
\[
S=\{(q,u_1,\dots, u_{6n})\in \tilde\Gamma_N\times (\Gamma_N^0)^{6n}\mid u_1,\dots, u_{6n} \ \text{are linearly independent}\}.
\] 
We define $\tilde \delta:S\times C_F\times \RR^{6n}\to \RR^{3k}$ by $\tilde \delta(q, u_1,\dots, u_{6n};F'; x_1,\dots, x_{6n})=\delta(q+x_1u_1+\cdots +x_{6n}u_{6n}, F')$ and we take its adjoint $\tilde \delta':S\times C_F\to \{\text{affine maps}:\RR^{6n}\to\RR^{3k}\}\cong Mat(3k,6n+1)$. Here, the sum and scalar multiple are the standard ones for the polynomials and $Mat(3k,6n+1)$ is the space of $ 3k\times (6n+1) $ matrices. This map is identified as $\tilde \delta'(q,u_i, F')=(\delta(q,F'),\delta(u_i,F'))_{1\leq i\leq l}$ (for an element of $\Gamma_N^0$, $\delta$ is defined by the same formula as above). Let $u_{ijm}$ denote the coefficient of $t^m$ of the $j$-th component of $u_i$. We have
\[
\frac{\partial}{\partial u_{ijm}}\tilde \delta'={}^t((t_{i1}^m-t_{i2}^m,\dots, t_{i1}^m-t_{i,a_i}^m)_{1\leq i\leq l}\ ;\ mp_1^{m-1},\dots , mp_b^{m-1})
\]
placed from the $((j-1)k+1,i)$-th to $(jk,i)$-the entries of the matrix with zeros in the other entries. \\
\indent Suppose $c(F)=k\leq 2n$. The $2N$ vectors $\frac{\partial}{\partial u_{ijm}}\tilde \delta'$ with $m=0,\dots, 2N-1$  are linear combinations of the row vectors of $VD(t_1,\dots t_N)$ whose coefficient matrix is of maximum rank, and we also have $|F'|<2k\leq 4n\leq N$. Therefore,  these vectors span $\RR^k$, which implies that $\tilde \delta'$ is a submersion. For $s=(q,u_i)\in S$, let $\langle s\rangle=\{q+x_1u_1+\cdots +x_{6n}u_{6n}\mid x_i\in \RR\}$. For $F'\in C_F$,  $\chi(\langle s\rangle, F')\not=\emptyset$ and  the codimension of $\chi(\langle s \rangle ,F')$ in $ \langle s \rangle $ is $3k$ if and only if  $rank \ \tilde\delta'(s,F')=3k$. The subset $Mat^{3k-i}\subset Mat(3k,6n+1)$ of matrices of rank $3k-i$ has codimension $i(6n-3k+i+1)$. By the parametric transversality theorem, for almost all $s$, $\tilde \delta'_s:C_F\to Mat(3k,6n+1)$ is transversal to $Mat^{3k-i}$, and  the set of configurations $F'$ such that the codimension of  $\chi( \langle s \rangle ,F')$ in  $\langle s \rangle$ is $3k-i$, identified with $(\tilde\delta'_s)^{-1}(Mat^{3k-i})$, has codimension $i(6n-3k+i+1)$ unless it is empty. If $ \langle s \rangle = \langle s' \rangle $, the transversality of $\tilde \delta'_s$ and $\tilde \delta'_{s'}$ are equivalent so we have obtained the claim of part 2.\\
\indent For $2n\leq k\leq 6n+1$, the space $\chi( \langle s \rangle ,F')$ is non-empty and $\dim \chi( \langle s \rangle ,F')= l$ if and only if $rank\ \tilde \delta'_s(F')=6n-l$. By the same argument as above, we see the claim of the lemma holds since  $|F'|\leq 12n+2\leq N$. The dimension of $C_F$ is at most $2k$ so if $k =6n+1$, we have $\chi( \langle s \rangle ,F')=\emptyset$.\\
\indent For $k > 6n+1$, we can take a configuration $F_1$ of complexity $6n+1$ satisfying $\chi( \langle s \rangle ,F')\subset \chi( \langle s \rangle ,F_1)$, so  we have $\chi( \langle s \rangle ,F')=\emptyset$.
\end{proof}
{\bf Convention.} Throughout the rest of the paper, we fix $V$ satisfying the conditions of Lemma \ref{Lgeneral_posi} and write $V=\Gamma_n$. (By choosing $\Gamma_n$ sufficiently close to $I(\tilde\Gamma_n)$, we can make $\Gamma_n$ include a knot which has the same isotopy type as a knot in $I(\tilde \Gamma_n)$.) After Definition \ref{Dpsi_0}, we will add a condition on the  choice of $\Gamma_n$ which is satisfied if $\Gamma_n$ is close to $I(\tilde \Gamma_n)$. 
\begin{defi}
Let $AG_i$ denotes the space of $i$ dimensional affine subspaces of $\tilde \Gamma_N$ with the standard quotient topology of an open subset of the space of $6N\times (i+1)$ matrices.
\end{defi}
The following lemma will be used in  section \ref{SU_1}.
\begin{lem}\label{Leval}
Let $T$ be a positive number. There exists a neighborhood $\WW_T$ of $I(\tilde \Gamma_n)$ in $AG_{6n}$ such that  
for any $n$ distinct points $t_1,\dots, t_n\in [-T,T]$ and any $V\in \WW_T$, the map 
\[
ev_{t_1,\dots, t_n} :V\to \RR^{6n}\ \  \text{given by}\ \  V\ni f\mapsto (f(t_i),f'(t_i))_{1\leq i\leq n}\in \RR^{6n}
\]
is bijective. In particular, if $F$ is a classed configuration all of whose geometric points belong to $[-T,T]$ and whose cardinality is $\leq n$, the map $\WW_T\ni V\mapsto \chi(V,F)\in AG_{6n-3c(F)}$ is well-defined and  continuous.
\end{lem}
\begin{proof}
Let $V\in AG_{6n}$.  Let $P_0,P_1,\dots, P_{6n}\in V$ be elements such that $P_1-P_0,\dots, P_{6n}-P_0$ are linearly independent. We denote by $v_i\in \RR^{6N}$ the vector made by arranging the coefficients of monomials in $P_i-P_0$. The matrix representing (the linear part of) the affine map $ev_{t_1,\dots, t_n}$ is given by the product
\[
A=\left(
\begin{array}{c}
R(t_1)\\
\vdots \\
R(t_n)\\
R'(t_1)\\
\vdots \\
R'(t_n)
\end{array}
\right)^{\oplus 3}
(v_1,\dots, v_{6n}),
\]
where $R(t)=(1,t,\dots, t^{2N-1})$. It is clear that $\det A$ is divided by $(t_i-t_j)^{12}$ ($1\leq i<j \leq n$). We set
\[
D(v_1,\dots, v_{6n})=\frac{\det A}{\prod_{i<j}(t_i-t_j)^{12}}.
\]
When $V=I(\tilde \Gamma_{n})$, we have 
\[
D(v_1,\dots, v_{6n})=k\prod_{i=1}^n(wt_i^{w-1}+1)\prod_{1\leq j<l\leq n}\left(
\frac{t_j^w-t_l^w}{t_j-t_l}+1\right)^{12}
\] for a choice of $P_0,\dots, P_{6n}$, where $k$ is a non-zero constant, where $w$ is the number in the definition of the map $I$. It follows that  $|D(v_1,\dots, v_{6n})|\geq |k|>0$ for $V=I(\tilde \Gamma_n)$ and for any (not necessarily pairwise distinct) $t_1,\dots, t_n$ since $w$ is odd. As $|D(v_1,\dots, v_{6n})|$ is continuous on  $v_1,\dots, v_{6n}, t_1,\dots, t_n$, we have proved the former part of the claim.
For the latter part, if $F$ satisfies the condition of the claim, by the property of evaluation map $ev_{t_1,\dots, t_n}$, $\chi(V,F)$ is non-empty and its codimension  in $V$ is $3c(F)$ for $V\in \WW_T$ (even if $V$ and $F$ do not satisfy the condition of Lemma \ref{Lgeneral_posi}). Therefore,  the spaces $V\in \WW_T$ and $\chi(\tilde \Gamma_N,F)$ intersect transversally in $\tilde \Gamma_N$, which ensures well-definedness and continuity of the map. 
\end{proof}
Let $Sing\subset \Gamma_n$ be the subset of singular maps, i.e. maps $f$ which have at least one pair $t_1<t_2$ with $f(t_1)=f(t_2)$ or one point $t$ with $f'(t)=0$. The complement $\Gamma_n-Sing$ is regarded as a subspace of the knot space $K_3'$. If we take $n$ sufficiently large, any finitely many isotopy types of knots are realized in the space $\Gamma_n-Sing$. By the Alexander duality, we have $\bar H^*(\Gamma_n-Sing)\cong \bar H_{6n-*-1}(Sing^*)$.\\
\indent We shall define the resolvent  of $Sing$, which admits a nice filtration.

\begin{defi}\label{Dgraph}
\begin{enumerate}
\item A {\em graph} $G$ in this paper consists of  a  vertex set $V(G)$  and an edge set $E(G)$. Each edge has at most two endopoints, which are vertices. If $v$ is an endpoint of an edge $e$, we say $v$ is {\em incident} with $e$ or $e$ is {\em incident} with $v$. An edge having only one endpoint is called a {\em loop}. A vertex is {\em discrete} if it is not incident with any edge (including a loop). By $e\in G$, we mean $e\in E(G)$.
\item Let $\GG^{g+}$ be the set of graphs $G$ satisfying the following conditions:
\begin{itemize}
\item The vertex set  $V(G)$ is a (possibly empty) finite subset of $\RR$ and 
the edge set $E(G)$ is a  subset of $\{(t,t')\in V(G)^2\mid t\leq t'\}$.  We understand that an edge $(t,t')$ is incident with $t$ and $t'$ (and it is a loop if $t=t'$). 
\item No vertex is discrete. 
\end{itemize}
For a graph $G\in \GG^{g+}$, let $G_1\in \GG^{g+}$ be the subgraph of $G$ consisting of the non-loop edges and their endpoints. The graph $G$  is called a {\em generating graph} of a classed configuration $F=(A,B)$ if $\pi_0(G_1)$ equals $A$ and the set of vertices incident with  loops of $G$  equals $B$.  We also say $F$ is the {\em (underlying  classed) configuration of} $G$.  We define the {\em cardinality} $|G|$ and  {\em complexity} $c(G)$ of $G$ as $|F|\ (=|V(G)|)$ and $c(F)$,  respectively. We say a map $f$ respects $G$ if it respects the underlying classed configuration of $G$.
\item Let $\GG^g\subset \GG^{g+}$ be the subset of graphs $G$ with $V(G)\not=\emptyset$.
\item Fix a polynomial map $\Psi:\RR^2\to \RR^M$ such that the convex hull of $\{\Psi(t,t')\mid (t,t')\in E(G)\}$ is an $( |E(G)|-1)$-dimensional  simplex for any generating graph with $c(G)\leq 6n$ and for any two distinct graphs with complexity $\leq 6n$, the corresponding simplices have no common inner points, where $M$ is a sufficiently large integer. We define a subspace $\Sigma^V\subset \Gamma_n\times \RR^M$ by declaring $(f,x)\in \Sigma^V$ if and only if $f$ respects $G$ and $x$ belongs to the convex hull of  $\{\Psi(t,t') \mid (t,t')\in E(G)\}$ for some $G\in \GG^g$. 
\end{enumerate}
\end{defi}
\begin{rem}
In \cite{vassiliev}, an element of $\GG^g$ is called a generating family  and a loop is denoted by $*$.
\end{rem}
\begin{prop}[\cite{vassiliev}]
The projection to $\Gamma_n$ induces a homotopy equivalence $(\Sigma^V)^*\to Sing^*$.\hfill \qedsymbol
\end{prop}
The space $\Sigma^V$ is Vassiliev's resolvent in \cite{vassiliev} (up to difference of the affine space $\tilde \Gamma_n$). 
We will give an equivalent definition of slightly different appearance, which is perhaps more familiar for homotopy theorists. 
\begin{defi}\label{Dgenerating}
We regard $\GG^{g+}$ as a topological poset as follows. 
\begin{enumerate}
\item  $\exists \ G_1\to G_2\iff E(G_1)\subset E(G_2)$.
\item For $G\in\GG^{g+}$ and $i\geq 1$, let $U_i(G)\subset \GG^{g+}$ be the subset of  graphs $H$ such that for any  $e\in E(H)$, there is an edge  $e'\in E(G)$ satisfying   $|e'-e|\leq 1/i$, where $|-|$ denotes the standard Euclidean norm of   $\RR^2$.  We topologize $\GG^{g+}$ so that $\{U_i(G)\}_{i\geq 1}$ is an open basis of any $G\in \GG^{g+}$.
\end{enumerate} 
We regard $\GG^g$ as the topological  subposet of $\GG^{g+}$. Let $|\GG^g|$ be the realization of the nerve of $\GG^g$ (see subsection \ref{SSNT}). Let
\[
\Sigma^0_{nc}\subset \Gamma_n\times |\GG^g|
\] 
be the subspace of  elements   $(f,u)$ such that   $f$ respects  $\LL(u)$, where $\LL(u)$ is the last $\GG^g$-element of $u\in |\GG^g|$. We define a pointed space $\Sigma^0$ by the one-point compactification $\Sigma^0=(\Sigma^0_{nc})^*$.
\end{defi}
For example, for three points $t_1<t_2<t_3$ in $\RR$, when $t_1$, $t_2$ are fixed and $t_3$ tends to $t_2$, the graph $G$ with $E(G)=\{(t_1, t_2), (t_1, t_3)\}$ tends to the graph with the unique edge $(t_1,t_2)$ in  $\GG^g$.
The following is obvious.
\begin{lem}\label{Lhomeo}
 $\Sigma_{nc}^0$ is homeomorphic to $\Sigma^V $.
\end{lem}
\begin{proof}
We use the map $\Psi:\RR^2\to \RR^M$ fixed in Definition \ref{Dgraph}.
We shall define a continuous map $v:\GG^g\to \RR^M$ such that for each graph $G\in \GG^g$, $v(G)$ belongs to the interior of the simplex spanned by $\{\Psi(e)\mid e\in E(G)\}$, that is, $v(G)$ can be expressed as $v(G)=\sum_{e\in E(G)}s_e\Psi(e)$ with some numbers $s_e>0$ satisfying $\sum_{e\in E(G)}s_e=1$. We define $v$ by induction on $p=|E(G)|$.  When $p=1$, we set $v(G)=\Psi(e)$, where $e$ is the unique edge of $G$. When $p=2$, we set $v(G)=\frac{\Psi(e_1)+\Psi(e_2)}{2}$, where $E(G)=\{e_1,e_2\}$. Suppose we have defined $v$ for the graphs with $\leq p-1$ edges. We set
\[
v(G)=\frac{\sum_{i=1}^p d_i \, v(\partial_iG)}{\sum_{i=1}^pd_i}
\]
where $E(G)=\{e_1,\dots, e_p\}$ and $\partial_iG$ is the graph obtained by removing $e_i$ from $G$, and $d_i=\prod_{j<k, \, j,k\not=i}|e_j-e_k|$. We easily see that $v(G)$ is continuous and in the interior of the simplex corresponding to $G$. The simplex corresponding to $G$ with $p=|E(G)|$ is divided into the simplices spanned by $\{v(G_1'),\dots, v(G_{p-1}'), v(G)\}$, where $\{G_1',\dots, G_{p-1}'\}$ runs through the sequences of graphs with $G_1'\subset \cdots \subset G_{p-1}'\subset G$ and $|E(G_i')|=i$.   Therefore,  the map $\Sigma^0_{nc}\to \Sigma^V$ given by 
\[
(f,\tau_0G_0+\cdots +\tau_mG_m)\mapsto (f,\tau_0v(G_0)+\cdots +\tau_mv(G_m))
\]
is a homeomorphism. 
\end{proof}
Let $\FFF_k=\FFF_k(\Sigma^0)\subset \Sigma^0$ be the subspace consisting of the basepoint and elements  $(f,u)$   which satisfy $c(\LL(u))\leq k$.   
In view of Lemma \ref{Lgeneral_posi} (3), this defines a bounded filtration on $\Sigma^0$
\[
\FFF_0\subset \FFF_1\subset \cdots \subset \FFF_k\subset \cdots \subset \FFF_{6n}=\Sigma^0,
\]
 which we call the {\em complexity filtration} on $\Sigma^0$. This filtration precisely corresponds to the filtration in \cite{vassiliev} through the homeomorphism in the proof of  Lemma \ref{Lhomeo} (up to the choice of $\tilde \Gamma_n$). \\
\begin{defi}[\cite{vassiliev}]\label{DVassiliev}
We consider the spectral sequence associated to the filtered complex $(\bar C_*(\Sigma^0),\{\bar C_*(\FFF_k)\})$. After relabeling the part of  bidegree $(p,q)$ with the bidegree $(-p,6n-q-1)$, we denote the spectral sequence by $\bar E^{p,q}_r(n)$ and call it the {\em unstable  Vassiliev spectral sequence}. For $n'>n$, we can pick  isomorphisms $\eta_r : \bar E^{p,*}_r(n')\cong \bar E^{p,*}_r(n)$ defined in the range of $-p+r\leq 6n/5$, such that $\eta_r$ is compatible with the $d_r$-differentials as far as it is defined,  and that the map induced by $\eta_r$ on the subquotient $\bar E^{p,*}_{r+1}(n)$ coincides with $\eta_{r+1}$. This is because the domains of differentials $d_{r_1}$ with $r_1<r$ coming into $(p,*)$ is in the stable range of complexity $\leq 6n/5$, see Lemma \ref{Lgeneral_posi} and Remark \ref{Rchoice_affine}.
We fix a sequence of integers and isomorphisms 
\[
n_1<n_2<\cdots <n_i<\cdots,\qquad \eta^i_r: \bar E_r^{p,*}(n_i)\to \bar E_r^{p,*}(n_{i-1})
\]
defined in $-p+r\leq 6n_{i-1}/5$ and satisfying the above compatibility. We define a spectral sequence by the inverse limit 
\[
^V\!\bar E_r=\lim_{\leftarrow}\{\bar E_r(n_i), \eta^i_r\}_i\ ,
\]
which we call   the {\em (stable) Vassiliev spectral sequence}. If we say simply, Vassiliev (spectral) sequence, it means the stable sequence  according to most of the literatures.
 
\end{defi}
Since the stable range (the range of complexity $\leq 6n/5$)  becomes larger as $n$ becomes larger, for each $(p,q,r)$, there is an integer $i_0=i_0(p,q,r)$ such that $^V\!\bar  E^{p q}_r\cong \bar E^{p q}_r(n_i)$ for $i\geq i_0$.  Different choices of the isomorphisms $\eta^i_r$ give the isomorphic stable sequences.

\begin{rem}\label{Rchoice_affine}
\begin{enumerate}
\item We can see that different choices of $\Gamma_n$ also give the isomorphic stable sequences  as follows. The subspace $\chi(\tilde\Gamma_N, F)\subset \tilde \Gamma_N$ has codimension $3c(F)$ for any configuration $F$ with $|F|\leq N$ so if we define the resolvent $\Sigma^V(N)$ by replacing $\Gamma_n$ with $\tilde  \Gamma_N$,  for the filtration defined by complexity, we have a homeomorphism $\FFF_{k}(\Sigma^V(N))\cong \FFF_k(\Sigma^V)\times \RR^{6N-6n}$ for $k\leq 6n/5$ (see \cite{vassiliev1}). This induces an isomorphism between the corresponding unstable sequences (with possibly different $n$) in the range of $-p+r\leq 6n/5$.
\item Our choice of the space of polynomial maps $\tilde \Gamma_N$ is different from that of \cite{vassiliev}, where components of maps are monic polynomials with zero constant term but these choices gives isomorphic Vassiliev spectral sequence  since these spaces are embedded in a common affine space such as the space of all maps consisting of monic polynomials of degree $2N+1$, to which we can apply the argument as in the previous part.
\item The boundness of the filtration $\FFF_k$ on $\Sigma^0$ is necessary for the convergence to $\bar H_*(\Sigma^0)$, which is used in the identification of the finite type invariants and the diagonal part of the $E_\infty$-page of the Vassiliev sequence, see \cite{BL}. 
\end{enumerate}
\end{rem}
We shall define an unreduced version of the Vassiliev sequence.
\begin{defi}\label{Dunreduced_Vassiliev}
  Let $ 
 \tilde \Sigma_{nc} \subset \Gamma_n\times |\GG^{g+}| $
be the subspace of elements
$(f,u)$ such that  $f$ respects  $\LL(u)$.
We set $\tilde \Sigma=(\tilde \Sigma_{nc})^*$.
The space $\Sigma^0$ in Definition \ref{Dgenerating} is naturally regarded as a subspace of $\tilde \Sigma$ by the inclusion $\GG^g\subset \GG^{g+}$.
We set
\[
\Sigma=\tilde \Sigma/ \Sigma^0.
\]
Let $\{\FFF_k=\FFF_k(\Sigma)\}_k$ be the {\em complexity filtration on} $\Sigma$ defined by $\FFF_k=\{*\}\cup \{(f,u) \mid c(\LL(u))\leq k\}$. Let $ E_r^{p,q}(n)$ be the spectral sequence associated to $(\bar C_*(\Sigma), \{\bar C_*(\FFF_k(\Sigma))\})$ after the change of bidegree $(p,q)\leftrightarrow (-p,6n-q)$. 
The space $\Sigma$ is homeomorphic to the mapping cone of the map $\iota:\Sigma^0\to \Gamma_n^*$ defined by the inclusion $\chi(\LL(u))\subset \Gamma_n$ after the projection. A homeomorphism is given by
\[
Cone(\iota)\ni (t; f, \tau_0G_0+\cdots +\tau_mG_m)\mapsto  (f, t\emptyset+(1-t)(\tau_0G_0+\cdots +\tau_mG_m)) \in \Sigma,
\]
where $\emptyset$ denotes the graph with the empty vertex and edge sets.
The composition of this map and the natural collapsing map $Cone(\iota)\to S^1\wedge \Sigma^0$ preserves the filtration so gives a map of spectral sequences $ E^{p,q}_r(n)\to \bar E^{p,q}_r(n)$ which is an isomorphism for $(p,q)\not=(0,0)$ and equals  the map $\kk\to 0$ for $(p,q)=(0,0)$. By stabilizing the sequence $E_r(n)$ similarly to $\bar E_r(n)$, we obtain a spectral sequence 
\[
 ^V\!E_r=\lim_{\leftarrow} E_r(n).
\] 
We also call  $^V\!E_r$ (resp. $E_r(n)$) the {\em Vassliev spectral sequence} (resp. {\em unstable Vassiliev sequence}) since it only differs with ${}^V\bar E_r$ (resp. $\bar E_r(n)$) in the bidegree $(0,0)$ for any $r$. We distinguish them by the presence of `bar' if necessary.
\end{defi}
\indent We can prove the  theorems in Introduction using either of $\Sigma^0$ or $\Sigma$ similarly. We will use $\Sigma$ hereafter since its correspondent naturally appears in the space level duality. The arguments in this paper are independent of a choice of the isomorphisms $E_r(n')\cong E_r(n)$ used to define the stable sequence. \\ 
\indent We define a variant of $\Sigma$.
\begin{defi}\label{Dbar_Sigma}
\begin{enumerate}
\item  For  $P\in \PP_n$, let $\TGG(P)$ be the set of graphs with  the vertex set $V(G)=P$ and an edge set $E(G)\subset \{(\alpha,\beta)\mid \alpha,\beta\in P,\ \alpha\leq \beta\}$.  So, the vertices of $G$ are the pieces of $P$.
$\GG(P)\subset \TGG(P)$ denotes the subset of graphs with an edge set $E(G)\subset \{(\alpha,\beta)\mid \alpha, \beta \in P^\circ\}$. In other words, $\GG(P)$ is the subset of graphs where the minimum and maximum pieces are discrete. Let $\emptyset_P\in \GG(P)$ denote the graph with  the empty edge set. We give $\TGG(P)$ a structure of poset as follows : There is a unique map $G\to H$ if and only if $E(G)\subset E(H)$.  We also regard $\GG(P)$ as the subposet of $\TGG(P)$. We endow $\TGG(P)$ and $\GG(P)$ with the discrete topology.
\item For a map $P\to Q$ of partitions,  $\delta_{PQ}:Q\to P$ denotes the map of sets sending  $\alpha\in Q$ to the piece of $P$ containing $\alpha$. This map induces a map $\delta_{PQ}:\TGG(Q)\to \TGG(P)$. For a graph $G\in \TGG(Q)$, the  graph $\delta_{PQ}(G)$ has an edge $(\alpha,\beta)$ if and only if there are pieces $\alpha_1, \beta_1$ of $Q$ satisfying $\alpha_1\subset \alpha$, $\beta_1\subset \beta$ and $(\alpha_1,\beta_1)\in E(G)$.
\item Let $P\in \PP_n$ and $G\in \GG(P)$. Let $G_1\subset G$ be the subgraph such that the edge set consists of all the non-loop edges of $G$ and the vertex set consists of their endpoints. Let $b$ be the number of loops in $G$. We define the {\em complexity} $c(G)$ by $c(G)=|V(G_1)|-|\pi_0(G_1)|+b$.  By this definition,  discrete vertices do not affect the complexity. 
\item Let $\PPP_n\subset [-\infty,\infty]^n$ be the subset consisting of elements $(t_1,\dots, t_n)$ satisfying $t_1\leq t_2\leq \cdots \leq t_n$. We call an element of $\PPP_n$ a {\em geometric partition}. For each $A=(t_1,\dots, t_n)\in \PP'_n$, there is a unique partition $P$ such that $t_i=t_j$ if and only if $i$ and $j$ belong to a common piece of $P$, where we agree that $t_0=-\infty$ and $t_{n+1}=\infty$. 
We call  $P$ the {\em (underlying) abstract partition of} $A$. The points $t_1,\dots, t_n$ are called the {\em geometric points of} $A$.
\item A {\em partitional graph} is a pair $J=(A,G)$ consisting of $A=(t_1,\dots, t_n)\in \PPP_n$ and $G\in \GG(P)$, where $P$ is the underlying partition of $A$. We  call $A$, $P$ and $G$ the {\em (underlying) geometric partition,  abstract partition, and  abstract graph of $J$}, respectively. We also call the geometric points of $A$ the {\em geometric points of $J$}. For $\alpha\in P$, set $t_\alpha=t_{i}$ with $i=\min\alpha$, where $t_0=-\infty$ and $t_{n+1}=+\infty$ as above. We always identify the graph $G$ with the graph with the vertex set $\{t_\alpha\mid \alpha\in P\}$ and the edge set $\{(t_\alpha, t_\beta)\mid (\alpha,\beta)\in G\}$.  So  a partitional graph is regarded as a graph whose vertices are points of $[-\infty,\infty]$ with the multiplicity which equals  the cardinality of the corresponding piece. For a partitional graph $J$, we call the  graph in $\GG^{g+}$ obtained by forgetting the discrete vertices and the multiplicity from $J$ the {\em underlying generating graph of $J$}. The underlying classed configuration of the  generating graph  is called the {\em underlying classed configuration} of $J$. The {\em complexity $c(J)$ of  $J$} is the complexity of its abstract graph, which equals the complexity of the underlying generating graph. 

\item Let $\GP_n$ be the set of partitional graphs. We topologize $\GP_n$ as follows:
For  $p\geq 1$ and $(A,G)\in \GP_n$  with $A=(t_1,\dots, t_n)$, let $V_p(A,G)\subset \GP_n$ be the subset of $(B,H)$ satisfying the following conditions.
\begin{itemize}
\item The abstract partition $Q$ of $B$ is a subdivision of the abstract partition $P$ of $A$, or equals  $P$,
\item $\delta_{PQ}(H)=G$, and
\item $d_{sc}(s_i,t_i)<1/p$ (see subsection \ref{SSNT}).
\end{itemize}

We topologize $\GP_n$ so that  $\{V_p(A,G)\}_{p\geq 1}$ is an open basis of $(A,G)$ for any $(A,G)\in\GP_n$. 
\item We give $\GP_n$ a structure of topological poset as follows : There is a morphism $(A_1,G_1)\to (A_2,G_2)$ if and only if $A_1=A_2$ and  there is a map $G_1\to G_2$ in $\GG(P)$, where $P$ is the partition of $A_1$.   Let $u$ be an element of the realization $|\GP_n|$. The {\em (underlying) abstract and geometric partitions, abstract graph, geometric points and complexity $c(u)$ of} $u$ are defined to be those of the last $\GP_n$-object of $u$, respectively. Let $P$ be the abstract partition of $u$. We define the {\em cardinality} $|u|$ of $u$ by $|u|=|P^\circ|$.
\item Let $\bar \Sigma'_{nc}\subset \Gamma_n\times|\GP_n|$ denote the subspace of elements $(f, u)$  such that $f$ respects the underlying classed configuration of  the last $\GP_n$-element of $u$.  Let $\bar \Sigma^0_{nc}\subset \bar \Sigma'_{nc}$ be the subset of elements of a form $(f,\tau_0J_0+\cdots +\tau_mJ_m)$ (for some $m\geq 0$, with $\tau_0\not=0$)  such that  the graph  $J_0$ has a non-empty edge set. We set
\[
\bar \Sigma=(\bar \Sigma_{nc}')^*/(\bar \Sigma_{nc}^0)^*.
\]
\item We define a  map $\phi_0:\bar \Sigma\to \Sigma$ by forgetting the discrete vertices and the multiplicity of vertices.
\item We define the {\em cardinality filtration} $\{F_l=F_l(\bar \Sigma)\}$ and  {\em complexity filtration } $\{\FFF_k=\FFF_k(\bar \Sigma)\}$ on $\bar \Sigma$ by $F_l=\{*\}\cup \{(f,u) \mid |u|\leq l\}$ and $\FFF_k=\{*\}\cup\{(f,u)\mid c(u)\leq l\}$. We also define the {\em cardinality filtration} $\{F_l(\Sigma)\}$ on $\Sigma$ similarly by using the cardinality of the vertex set of the last $\GG^{g+}$-element of an element $u\in |\GG^{g+}|$. We sometimes write $\Sigma_{\leq l}=F_l(\Sigma)$.
\item For $P_0\in \PP_n$, we denote by $\bar \Sigma(P_0)\subset \bar \Sigma$ the subspace of the basepoint and elements $(f,u)$ such that $P_0$ is equal to or a subdivision of the underlying abstract partition of $u$. 
\item Let $X=\Sigma$ or $\bar\Sigma$. We also consider the complexity and cardinality filtrations on a subquotient of $X$ obtained as the image of the intersection of the  subspace and each stage of the  filtration by the quotient map. For example, the complexity filtration $\{\FFF_k\}_k$ on $\Sigma_{\leq l}$ is given by $\FFF_k=\FFF_k(\Sigma)\cap \Sigma_{\leq l}$. 
\end{enumerate}
\end{defi}
For a polynomial $P(t)=a_{2N-1}t^{2N-1}+\cdots +a_0$, we set
\[
|P|=\sqrt[]{a_0^2+\cdots +a_{2N-1}^2}
\]
For  a triple $(P_1,P_2,P_3)$ of polynomials of degree $\leq 2N-1$, we set 
$|(P_1,P_2,P_3)|=\sqrt[]{|P_1|^2+|P_2|^2+|P_3|^2}$. 
Throughout the paper, we fix the  distance on $\tilde \Gamma_N$  defined by $d(f,g)=|f-g|$.
\begin{prop}\label{Lisom_sigma}
We use the notations in Definition \ref{Dbar_Sigma}. Let $P_0\in \PP_n$ be a partition with $l_0:=|P_0^\circ|<\frac{3n}{5}$.
The  map $\phi_0 :\bar \Sigma\to \Sigma$  is continuous and its restriction \[
\bar \Sigma(P_0)\to \Sigma_{\leq l_0}
\] induces an  isomorphism between $E_2$-pages of the spectral sequences associated the complexes $\bar C_*(\bar \Sigma(P_0))$ and $\bar C_*(\Sigma_{\leq l_0})$ with  the cardinality filtrations.
\end{prop}
The following definition will be used in the proofs of Propositions \ref{Lisom_sigma},  \ref{Phomology_U_1} and \ref{Punstable_diff}.
\begin{defi}\label{Dpart_config}
A pair $(A,H)$ of a geometric partition $A$ and a classed configuration $H$ is called a {\em partitional configuration} if the geometric points of $H$ belong to the set of geometric points of $A$. Two partitional configurations $(A,H)$ and $(A',H')$ are said to be {\em equivalent} if there is a diffeomorphism $g:\RR\to \RR$ which gives an equivalence between $H$ and $H'$ and preserves the points of geometric partitions including the multiplicity. A partitional configuration $(A,H)$ is {\em regular} if the set of geometric points  of $A$ in $\RR$ equals the set of geometric points of  $H$.
\end{defi}

\begin{proof}[Proof of Proposition \ref{Lisom_sigma}]
 $\bar \Sigma$ is clearly a first-countable space so we shall show that $\phi_0$ is sequentially continuous. We fix a point $f_0\in \Gamma_n$. Let $\{(g_k, u_k)\}$ be a sequence in $\bar \Sigma$ with limit $(g_\infty, u_\infty)$. Set $J_k=\LL(u_k)$. If a non-discrete vertex of $J_k$ goes to $\pm \infty$ as $k\to \infty$, the distance $d(g_k,f_0)$ goes to infinity and we have $(u_\infty, g_\infty)=*$ (even when the coefficient of $J_k$ in $u_k$ goes to $0$). Therefore, if $(u_\infty,g_\infty)\not=*$, the absolute values of the non-discrete points of $J_{k}$ are bounded by a constant independent on $k$. This observation implies the continuity.\\
\indent We shall prove the latter claim. Let $Y:\Delta_{l_0}^{op}\to \CH_\kk$ be the functor given by taking cohomology of the $l_0$-truncated Sinha cosimplicial model $\SSS^{\leq l_0}$  in the termwise manner, $Y=H^*(\SSS^{\leq l_0})$. Here, cohomology group is regarded as a complex with zero-differential. Its $l$-th term  is the graded commutative algebra generated by elements $g_{ij}$  of degree $2$ ($1\leq i,j\leq l$) with relations $g_{ij}=-g_{ji} (i\not=j),\ (g_{ij})^2=0,\ g_{ij}g_{jk}+g_{jk}g_{ki}+g_{ki}g_{ij}=0$. The differential of $NY$ is given by the signed sum of the maps  $\partial_i(g_{jk})=g_{f(j),f(k)}$ where $f:\{1,\dots, n\}\to \{1,\dots, n-1\}$ is the order-preserving surjection which shrinks $i, i+1$.  \\
\indent We shall prove that  the map on the $E_1$-pages induced by the map $\phi_0:\bar \Sigma(P_0)\to \Sigma_{\leq l_0}$ is identified with the unit $N\CECHF \mathcal{F}^*Y\to NY$ of the adjoint in Lemma \ref{Lcofinal}.  Let $\{F_l\}_l$ denote the cardinality filtration on $\Sigma_{\leq l_0}$. We will first show the $E_1$-page of the spectral sequence associated to  $\{\bar C_*(F_l)\}_l$ is\vspace{1mm} isomorphic to the normalization $NY$ up to a degree shift. 
 \vspace{1mm}  We consider an auxiliary filtration $\{\bar\FFF_k\}$ which is the complexity filtration on the quotient $F_l/F_{l-1}$.The connected components of   $\bar \FFF_k-\bar \FFF_{k-1}$ correspond to the equivalence classes of classed configurations of cardinality $l$ and complexity $k$. If $\bar \FFF_k-\bar \FFF_{k-1}\not=\emptyset$, we have $k\leq 2l_0-1<6n/5$. In this range, by Lemma \ref{Lgeneral_posi}, the one-point compactification of a component of $\bar \FFF_k-\bar \FFF_{k-1}$ is  homeomorphic to a space of the following form:
\[
S^l\wedge (|\GG^H|/|\GG^{<H}|)\wedge \chi(\Gamma_n, H)^*.
\] 
Here, 
\begin{enumerate}
\item $S^l$ is the $l$-dimensional sphere regarded as the one-point compactification of $\{(t_1,\dots, t_l)\in \RR^l\mid t_1<\cdots <t_l\}$, 
\item $\GG^H$ is the subposet of $\GG^{g+}$ consisting of subgraphs of generating graphs of a fixed classed configuration $H$ whose geometric points are precisely $1,\dots, l\in \RR$ and whose complexity is $k$ (including the generating graphs themselves), and
\item $\GG^{<H}\subset \GG^H$ is the subposet of graphs which are not the generating graphs of  $H$,
\end{enumerate}
see \cite{vassiliev,vassiliev1} for details. The homology $\bar H_*( (|\GG^H|/|\GG^{<H}|)\wedge \chi(\Gamma_n, H)^*)$ is isomorphic to the subspace of $Y([l])^{2k}$ spanned by monomials whose set of  subscripts $(i,j)$ forms the edge set of a generating graph of $H$ with the degree shift $6n-2k\leftrightarrow 2k$. This can be shown by the same argument as the computation of the $E_2$ page of Bendersky-Gitler sequence (see e.g.\cite{FT}) or see \cite{vassiliev1} for another proof. The spectral sequence for $\{\bar \FFF_k\}$ degenerates at $E_1$-page by degree reason so $H_*(F_l/F_{l-1})\cong \oplus_{k\leq 2l}H_*(\bar \FFF_k/\bar \FFF_{k-1})$. With this decomposition, we shall identify the $d_1$-differential  $H_*(F_l/F_{l-1})\to H_*(F_{l-1}/F_{l-2})$ (extension problems will not occur again by degree reason). Let $\{\bar F_l\}_{l\geq 0}$ be the cardinality filtration on $\FFF_k/\FFF_{k-1}$ where $\FFF_k$ is the complexity filtration on $\Sigma_{\leq l_0}$. Obviously, we have $\bar \FFF_k/\bar \FFF_{k-1}=\bar F_l/\bar F_{l-1}$. Both of  the inclusion $\FFF_k\to \Sigma_{\leq l_0}$ and the quotient map $\FFF_k\to \FFF_k/\FFF_{k-1}$ induces the maps between the filtered quotient by cardinality as follows.
\[
F_l/F_{l-1}=F_l/F_{l-1}(\Sigma_{\leq l_0})\leftarrow F_l/F_{l-1}(\FFF_k)\to F_l/F_{l-1}(\FFF_k/\FFF_{k-1})=\bar F_l/\bar F_{l-1}=\bar \FFF_{k}/\bar \FFF_{k-1},
\]
where $F_l/F_{l-1}(Z)$ denotes $F_l(Z)/F_{l-1}(Z)$ with $F_i(Z)$ being the cardinality filtration on $Z$.
These maps induce  isomorphisms between  $(6n-2k+l)$-th homology.
In view of the above computation of $H_*(\bar \FFF_k/\bar \FFF_{k-1})$, the $d_1$-differential on the sequence associated to the filtration $\{\bar F_l\}_l$ on the filtered quotient is identified with the differential of the complex $NY$ since $d_1$ corresponds to the boundary of $\bar F_l-\bar F_{l-1}$ which appears as the limit of two successive vertices getting closer (see \cite{vassiliev} for details). The mentioned inclusion and quotient map induce maps between spectral sequences for the cardinality filtrations so we have proved the $E_1$-page for $\Sigma_{\leq l_0}$ is isomorphic to the normalization $NY$.  \\
\indent For $\bar \Sigma(P_0)$, the connected components of $\FFF_k(F_l'/F_{l-1}')-\FFF_{k-1}(F_l'/F_{l-1}')$ correspond to the equivalence classes of the partitional configurations $(A,H)$  such that $|A|=l$, $c(H)=k$ and $P_0$ is equal to or a subdivision of  the abstract partition of $A$. Here, $F_l'$ denotes the cardinality filtration on $\bar \Sigma(P_0)$ and $\FFF_i(F_l'/F_{l-1}')$ is the complexity filtration on the quotient. If $(A,H)$ is regular, the one-point compactification of the connected component is homeomorphic to  the compactification of the component labeled with the classed configuration $H$ which appeared in the above computation for $\Sigma_{\leq l_0}$.  If $(A,H)$ is not regular, the one-point compactification labeled with $(A,H)$  is homeomorphic to the smash product of a sphere and the compactification  labeled with $H$, where  dimension of the sphere is the number of geometric  points of $A$ which are not the points of $H$.  With these observations, we see that the $E_1$-page for $\{F_l'\}_l$ is isomorphic to $N\CECHF\mathcal{F}^*Y$ similarly to the case of $\Sigma_{\leq l_0}$. Since the non-regular components are sent to $*$ by $\phi_0$, the map between $E_1$-pages of the sequences for $\Sigma_{\leq l_0}$ and $\bar \Sigma(P_0)$ induced by $\phi_0$ is  identified with the unit map in Lemma \ref{Lcofinal} (see the paragraph after the lemma). Here, the lemma is applied to the restriction of $\mathcal{F}$ to the over-category $\PP_n/P_0\cong \PP_{l_0}$. Thus, by the  lemma, we have proved the latter claim.
\end{proof}

\begin{rem}\label{Rfiltration}
In the rest of the paper, we will prove  claims that a map induces an isomorphism of spectral sequences like Proposition \ref{Lisom_sigma}. In such claims, we mainly use cardinality filtrations while similar claims hold for complexity filtrations.  The reason for this  is that the cardinality filtration exists for the original Sinha sequence (before transferred by the duality) so it may be effective in a possible proof of Conjecture \ref{Conj_compati}. 
\end{rem}

\section{Translation of the punctured knot model }\label{Stranslation}
\subsection{Configuration space model $\CC$}
In the rest of the paper, for an integer $p$,  we express an element of $\RR^{6p}$ like $(x_i,v_i)_{1\leq i\leq p}$ with $x_i, v_i\in\RR^3$, or $(x,v)$ where $x=(x_1,\dots, x_p), \ v=(v_1,\dots, v_p)$. Intuitively speaking, the former components $x_i$ represent position and the latter components $v_i$ do velocity. For an element $y\in \RR^k$, $|y|$ stands for the standard Euclidean norm throughout the rest of the paper and we consider the distance on $\RR^k$ induced by the norm unless otherwise stated.  
\begin{defi}\label{Dpunctured}

Let $P\in \PP_n$ and set $p= |P|-2$. 
\begin{enumerate}

\item We define  positive numbers $\eps_P,\ \bar \eps_P$ by
\[
\eps_P=\frac{1}{10\cdot 8^{(n+5)(3n-p+5)}}, \qquad \bar \eps_P=\frac{1}{10\cdot 8^{(n+5)(p+5)}}\ .
\]

\item  We often label the $i$-th component of an element of $\RR^{6p}=(\RR^6)^{p}$ with the $i$-th piece of $P^\circ$, so an element of $\RR^{6p}$ is expressed  like $(x,v)=(x_\gamma,v_\gamma)_{\gamma \in P^\circ}$.
\item 
 For  $(v_\gamma)_{\gamma\in P^\circ}\in \RR^{3p}$ and  $\alpha,\beta\in P^\circ$, we set
\[
\begin{split}
c_\alpha= &\ \frac{|\alpha|}{20(n+2)}\ , \quad  \quad 
c^v_\alpha= c_\alpha |v_\alpha|\ ,\qquad c^v_{\alpha \beta}=   c^v_\alpha+c^v_\beta\ ,\\
&\\
c^v_{\leq \alpha} & =  \ c^{v}_\alpha+2\sum_{\gamma\in P, \gamma<\alpha}c^{v}_\gamma\ , \quad \quad
c^v_{\geq \alpha}=  c^v_\alpha+2\sum_{\gamma\in P, \gamma>\alpha}c^v_\gamma \ . \\
&
\end{split}
\]
Here, when $\gamma \in P-P^\circ$, we set $v_\gamma=(1/2,0,0)$.

\item Let $Q$ be a subdivision of $P$ and write $P=\{\alpha_0<\cdots <\alpha_{p+1}\}$ and $Q=\{\beta_0<\cdots <\beta_{q+1}\}$. We define an affine monomorphism $e_{PQ}:\RR^{6p}\to \RR^{6q}$ as follows. Let $(x_i,v_i)_{1\leq i\leq p}\in \RR^{6p}$. For convenience, we set 
\[
x_0=(-1+ c_{\alpha_0}/2,0,0),\quad x_{p+1}=(1-c_{\alpha_{p+1}}/2,0,0), \quad v_0=v_{p+1}=(1/2,0,0).
\] For $0\leq i\leq p+1$, suppose that $\alpha_i$ includes exactly $k$-pieces of $Q$, say $\beta_l,\dots, \beta_{l+k-1}$. We create the line segment which is centered at $x_i$, parallel to $v_i$, and of length $ 2c_{\alpha_i}|v_i|$, and divide this segment into the $k$ little segments of length $ 2c_{\beta_l}|v_i|,\dots , 2c_{\beta_{l+k-1}}|v_i|$ arranged in the direction of $v_i$. Let $y_{l+j-1}$ be the center of the $j$-th little segment and set $w_{l+j-1}=v_i$. We set $e_{PQ}((x_i,v_i)_{i})=(y_{m},w_m)_{1\leq m\leq q}$. For the maximum partition $Q$ consisting of the $n+1$ singletons,   we write $e_{PQ}=e_P$ (see Figure \ref{Flinear_emb}). It is clear that $e_{QR}\circ e_{PQ}=e_{PR}$ for a subdivision $R$ of $Q$. Let $\pi_P:\RR^{6n}\to e_P(\RR^{6p})=\RR^{6p}$ be the orthogonal projection. 
\end{enumerate}
\end{defi}
We will use the following observation many times implicitly.
\begin{lem}
If $Q$ is a subdivision of $P$, we have \quad
$\displaystyle
 \eps_P\leq \frac{\eps_Q}{8^{n+5}}\leq \frac{\bar \eps_Q}{8^{2(n+5)}} \leq \frac{\bar \eps_P}{8^{3(n+5)}}
$.
\hfill\qedsymbol
\end{lem}

\begin{defi}\label{Dconfig_model}
\begin{enumerate}
\item We define a functor $\mathcal{C}:\PP_n\to \CG$ as follows.
\begin{enumerate}
\item  For $P\in \PP_n$, $\mathcal{C}(P)$ is the subspace of $\RR^{6p}$ consisting of elements $(x_\gamma,v_\gamma)_{\gamma\in P^\circ}$ satisfying the following inequalities for each $\alpha\in P^\circ$ and pair $(\alpha,\beta)\in (P^\circ)^2$ with $\alpha\not=\beta$:
\[
\begin{split}
|x_\alpha|< 1- c^v_\alpha,& \qquad  |x_\alpha-x_\beta|  >  c^v_{\alpha\beta}-\frac{\eps_P}{8},\qquad 10(n+2)\bar \eps_P < |v_\alpha|< \frac{3}{4}, \\
-1+ (c^v_{\leq \alpha}+c^v_\alpha) & -\frac{\eps_P}{8}\  <_1 x_\alpha+c_\alpha v_\alpha, \quad  \quad 
x_\alpha-c_\alpha v_\alpha  <_1\  1- (c^v_{\geq \alpha}+c^v_\alpha)+\frac{\eps_P}{8}.
\end{split}
\]

\item If $Q$ is a subdivision of $P$,  the corresponding map $ \mathcal{C}(P)\to \mathcal{C}(Q)$ is given by the map $e_{PQ}$ defined above. We will see that this is well-defined in Lemma \ref{Le_incl}.

\end{enumerate} 
\item We can define an inclusion $\mathcal{C}(P)\to \PK (P)$ by creating the line segment centered at $x_\alpha$ of length $(2c_{\alpha}-\frac{1}{20(n+2)})|v_\alpha|$ parallel to $v_\alpha$ for $(x_\alpha,v_\alpha)_{\alpha\in P^\circ}$ similarly to the definition of $e_{PQ}$, and regard the union of the segments as an embedding with locally constant velocity. These inclusions form a natural transformation $\mathcal{C}\to \PK$. 
\end{enumerate}

\end{defi}

\begin{lem}\label{Le_incl}
If $Q$ is a subdivision of $P$, the inclusion $e_{PQ}(\mathcal{C}(P))\subset \mathcal{C}(Q)$ holds.
\end{lem}
\begin{proof}
Take $(x_\gamma,v_\gamma)_{\gamma\in P^\circ}\in \mathcal{C}(P)$ and put $e_{PQ}(x_\gamma,v_\gamma)=(y_\gamma,w_\gamma)_{\gamma\in Q^\circ}$.
Let $\alpha \in Q^\circ$ and take $\alpha_1\in P$ with $\alpha_1\supset \alpha$.  Let $\beta$ be the union of pieces $\gamma\in Q^\circ$ satisfying $\gamma\subset \alpha_1$ and $\gamma>\alpha$. Since $w_\gamma=v_{\gamma_1}$ for $\gamma\subset \gamma_1$ we have
\[
\begin{split}
y_\alpha+c_\alpha w_\alpha  &  =x_{\alpha_1}+c_{\alpha_1}v_{\alpha_1}-2c_\beta v_{\alpha_1}\\
&  >_1 -1+(c^v_{\leq \alpha_1}+c^v_{\alpha_1})-\eps_P/8-2c_{\beta}|v_{\alpha_1}| \\
&>-1+ (c_{\leq \alpha}^w+c_\alpha^w)-\eps_Q/8,
\end{split}
\]
which is one of the defining inequalities of $\mathcal{C}(Q)$. We shall consider the inequality involving $|y_{\alpha}-y_\beta|$ for $\alpha,\beta\in Q^\circ$. Let $\alpha_1,\beta_1$ be the pieces of $P$ satisfying $\alpha\subset\alpha_1,\beta\subset \beta_1$. We shall show the inequality for the case of $\alpha_1,\beta_1\in P^\circ$ and $\alpha_1\not=\beta_1$. The other cases are similar and easier. We easily see
\[
|y_\alpha-x_{\alpha_1}|\leq (c_{\alpha_1}-c_{\alpha})|v_{\alpha_1}|, \qquad |y_\beta-x_{\beta_1}|\leq (c_{\beta_1}-c_{\beta})|v_{\beta_1}|.
\] 
It follows that
\[
\begin{split}
|y_\alpha-y_\beta| & \geq |x_{\alpha_1}-x_{\beta_1}|-|y_\alpha-x_{\alpha_1}|-|y_\beta-x_{\beta_1}| \\ 
& > c^v_{\alpha_1,\beta_1}-\frac{\eps_P}{8} -(c_{\alpha_1}-c_{\alpha})|v_{\alpha_1}| -(c_{\beta_1}-c_{\beta})|v_{\beta_1}|   \\
& >c^w_{\alpha\beta}-\frac{\eps_Q}{8}.
\end{split}
\]
Similarly,  the rest of  defining  inequalities of $\mathcal{C}(Q)$ hold for $(y,w)$ and we have proved the claim.
\end{proof}

\begin{figure}
\begin{center}
\input{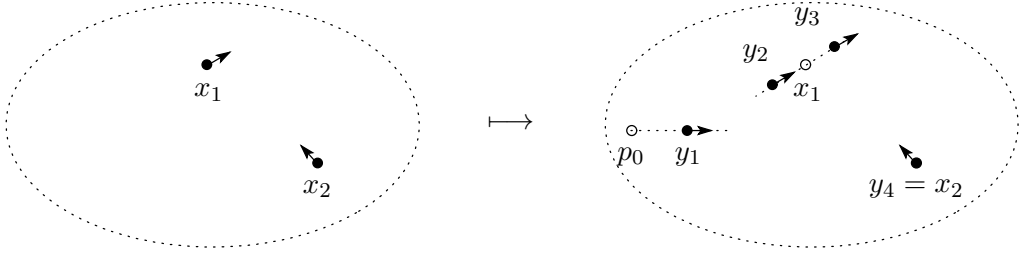}
\end{center}
\caption{the map $e_{P}$ for $P=\{\{01\},\{23\},\{4\},\{5\}\}$ ($p_0=(-1,0,0)$) } \label{Flinear_emb}
\end{figure}
We will see that the inclusion $\mathcal{C}(P)\to \PK(P)$ is a homology isomorphism later (see Lemma \ref{Lnathomologyiso}).

\subsection{Thom space model $\TT$}
In this subsection we construct the Thom space model, which is similar to the one given in \cite{moriya3}, but we take  tangent vectors into account here. More concretely, in \cite{moriya3}, we gave a space level lifting of Poincar\'e-Lefschetz duality $H^*(Conf_p(IntD^3))\cong H_*((D^3)^p, Fat\cup \partial (D^3)^p)$ which is compatible with the structure of embedding calculus, where $Conf_p(Int(D^3))$ is the ordered configuration space of $p$ points in the open unit disk, and $Fat$ is the fat diagonal, $Fat=\{(x_1,\dots, x_p)\in (D^3)^p \mid x_i=x_j\text{\ for some }\ i\not=j\}$.  Here, we give a similar lifting of the duality 
\[
H^*(Conf_p(IntD^3)')\cong H_*((D^3)^{2p}, \{ Fat\times (D^3)^p\}\cup \{(D^3)^p\times V_0\}\cup \partial (D^3)^{2p}),
\] where $Conf_p(Int D^3)'$ is the ordered configuration space of $p$ points with  tangent vectors $v$ satisfying $0<|v|<1$ attached,  and $V_0=\{(v_1,\dots, v_p)\in (D^3)^p\mid v_i=0\ \text{for some}\ i\}$.
\begin{defi}
Let $P\in \PP_n$. Put $p=|P^\circ|$. We  use the notations of Definition \ref{Dpunctured}.
\begin{enumerate}
\item 
We let $\nu_P\subset \RR^{6n}$ denote the open subset of elements $(y,w)$ satisfying the following conditions:
\begin{enumerate}
\item The distance between $(y,w)$ and the subspace $e_P(\RR^{6p})$ is $<\eps_P$,
\item $|y_i|< 1$ and $|w_i|<1$ for $1\leq i\leq n$, and
\item for $(x,v)=\pi_P(y,w)$,  the following inequalities hold for each  $\alpha\in P^\circ$
\[
-1+ (c^v_{\leq \alpha}+c^v_\alpha)-\eps_P   <_1 x_\alpha+ c_\alpha v_\alpha,  \qquad 
x_\alpha- c_\alpha v_\alpha    <_11- (c^v_{\geq \alpha}+c^v_\alpha )+\eps_P.
\]
\end{enumerate}
\item  For  $\alpha, \beta\in P^\circ$ with $\alpha\not=\beta $,  we define a   subspace $D_{\alpha\beta}$ of $\RR^{6p}$ by
\[
D_{\alpha\beta}=D_{\alpha\beta}(P)=\{(x,v) \mid |x_\alpha-x_\beta|\leq d_{\alpha\beta}(P,v)\}, 
\]
where 

\[
d_{\alpha\beta}(P,v)=\max\{\bar \eps_P,\  c_{\alpha \beta}^v-\eps_P\}.
\]
For $\alpha\in P^\circ$, we also set 
\[
D_{\alpha\alpha}=D_{\alpha\alpha}(P)=\{(x,v)\mid |v_\alpha|\leq \bar \eps_P\}.
\]

\item We set
\[
\TT_{\emptyset_P}=\RR^{6n}/(\RR^{6n}-\nu_P).
\]
 For two pieces $\alpha, \beta\in P^\circ$ with $\alpha\leq \beta$, 
we define a subspace $\TT_{\alpha\beta}\subset  \TT_{\emptyset_P}$ by
\[
 \TT_{\alpha\beta}=\{(y,w)\in \TT_{\emptyset_P}\mid \pi_P(y,w)\in D_{\alpha\beta}\}.
\]
(This definition means $*\in \TT_{\alpha\beta}$ since $\pi_P^{-1}(D_{\alpha\beta})\cap (\RR^{6n}-\nu_P)\not=\emptyset$, see subsection \ref{SSNT} for our convention.)
\end{enumerate}
\end{defi}
\begin{defi}
Let $V$ be an affine space and $Y$ a subset of $V$.
\begin{enumerate} 
\item We say $Y$ is {\em star-shaped} if there is a point $y_0\in Y$ such that for any $y\in Y$, the line segment between $y_0$ and $y$ is included in $Y$. We call a point $y_0$ satisfying this condition a {\em center of }$Y$.
\item 
 Suppose $Y$ is a compact star-shaped subset. We say $Y$ is {\em regular} (in $V$) if there is a center $y_0$ of $Y$ such that $y_0\in Int(Y)$ and for any $y\in Y$ other than $y_0$, there exists a unique $t>0$ such that $(1-t)y_0+ty\in \partial Y$. We call $y_0$ a {\em regular center} of $Y$. Here, $Int(Y)$ is the interior of $Y$ and $(1-t)y_0+ty$ denotes the point dividing the line segment between $y_0$ and $y$ by the rate $t : 1-t$. (When $t>1$, this is an externally dividing point.)
\end{enumerate}
\end{defi}

The following lemma is obvious.
\begin{lem}\label{Lnu_P_star}
The closure $\bar\nu_P$ of the open subset $\nu_P\subset \RR^{6n}$ is a regular star-shaped subset  with a regular center $e_P(0)$, where $e_P$ is the affine map used in the definition of $\nu_P$.\hfill\qedsymbol
\end{lem}

\begin{lem}\label{Ldiagonal_bound}
Let $Q$ be a subdivision of a partition $P\in \PP_n$. We have $\nu_P\subset \nu_Q$.
In particular, the identity on $\RR^{6n}$ induces the collapsing map 
$\delta'_{PQ}:\TT_{\emptyset_Q}\to \TT_{\emptyset_P}.
$
\end{lem}
\begin{proof}
For $(y,w)\in \RR^{6n}$, we set
\[
\begin{split}
\pi_{Q}(y,w)&=(x_\gamma,v_\gamma)_{\gamma\in Q^\circ}, \\
\pi_{P}(y,w)&=(\bar x_{\gamma},\bar v_{\gamma})_{\gamma\in P^\circ},\\
e_{PQ}(\pi_P(y,w))&=(x'_\gamma,v'_\gamma)_{\gamma\in Q^\circ} .
\end{split}
\] Suppose $(y,w)\in \nu_P$.   
 Intuitively speaking, $(x_\gamma,v_\gamma)$ and $( x'_{\gamma},v'_{\gamma})$ are different in general but the difference is sufficiently small since we have chosen the width $\eps_P$ of the  neighborhood $\nu_P$ sufficiently small. We easily see that $(x'_{\gamma},v'_\gamma)$ satisfies the defining inequalities of $\nu_Q$ in which $\eps_Q$ is replaced with $\eps_P$. Since $\eps_P\ll \eps_Q$, the range given by the inequalities  is narrower  than  the range of  the original inequalities of $\nu_Q$ and the margin covers the difference between the two points $(x,v)$ and $(x',v')$ so we have $(y,w)\in \nu_Q$. We shall give a rigorous proof.  
 Since the image of $e_P$ is contained in the image of $e_Q$ and the map $\pi_{Q}$ sends $(y,w)$ to its closest point in the image of $e_Q$, we have
\begin{equation}
|(y,w)-e_Q(x_\gamma,v_\gamma)|\leq |(y,w)-e_Q(x'_\gamma,v'_\gamma)|=|(y,w)-e_P(\pi_P(y,w))|< \eps_P\ (<e_Q). \label{Eemb}
\end{equation}
Thus,  $(y,w)$ satisfies the first defining condition of $\nu_Q$. By this inequality, we have
\[
|e_Q(x'_\gamma,v'_\gamma)-e_Q(x_\gamma, v_\gamma)|\leq |e_Q(x'_\gamma,v'_\gamma)-(y,w)|+|(y,w)-e_Q(x_\gamma,v_\gamma)|\leq 2\eps_P.
\]
As $e_Q$ expands the distance, we have
\begin{equation}
|(x'_\gamma,v'_\gamma)_\gamma-(x_\gamma,v_\gamma)_\gamma |\leq 2\eps_P \quad (\text{so} \ |x'_\gamma-x_\gamma|,\ |v'_\gamma-v_\gamma|\leq 2\eps_P\ \text{for each}\ \gamma\in Q^\circ). \label{Edifference}
\end{equation}
Let $\alpha_1 \in P^\circ$ and $\alpha\in Q^\circ$ with $\alpha\subset \alpha_1 $. We have $\bar x_{\alpha_1 }- c_{\alpha_1 }\bar v_{\alpha_1 }<_11-(c^{\bar v}_{\geq \alpha_1 }+c^{\bar v}_{\alpha_1 })+\eps_P$. Let $\beta$ be the union of pieces of $Q^\circ$ which are included in $\alpha_1 $ and smaller than $\alpha$. By definition of $e_{PQ}$, we have $x'_\alpha=\bar x_{\alpha_1 }+(- c_{\alpha_1 }+2 c_\beta+ c_\alpha)\bar v_{\alpha_1 }$ and $v'_\alpha=\bar v_{\alpha_1 }$ so with the inequality (\ref{Edifference}),
\[
\begin{split}
x'_\alpha- c_\alpha v'_\alpha& =\bar x_{\alpha_1 }- c_{\alpha_1 }\bar v_{\alpha_1 }+2 c_{\beta}\bar v_{\alpha_1 } \\
&<_1 1-(c^{\bar v}_{\geq \alpha_1 }+c^{\bar v}_{\alpha_1 })+\eps_P+2 c_\beta|\bar v_{\alpha_1 }| \\
&=1-(c^{v'}_{\geq \alpha}+c^{v'}_{\alpha})+\eps_P,  \\
\therefore \ x_\alpha- c_\alpha v_\alpha &<_11-(c^{v}_{\geq \alpha}+c^{v}_{\alpha})+5\eps_P<1-(c^{v}_{\geq \alpha}+c^{v}_{\alpha})+\eps_Q.
\end{split}
\]
Similarly, we can show that the other inequality for the first coordinate holds so $(y,w)\in \nu_Q$. 
\end{proof}
\begin{lem}\label{Ldiagonal_incl}
Let $Q$ be a subdivision of $P$ and $\alpha, \beta \in Q^\circ $ pieces with $\alpha\leq \beta$.  Let $\alpha_1 , \beta_1  \in P$ be the pieces which include $\alpha$, $\beta$ respectively. Let $\delta'_{PQ}:T_{\emptyset_Q}\to T_{\emptyset_P}$ denote the map given in Lemma \ref{Ldiagonal_bound}.
We have 
\[
\delta'_{PQ}(\TT_{\alpha \beta})\ \subset \
\left\{
\begin{array}{cc}
\{*\} & (\text{if $\alpha_1 =\min P$ or $\beta_1 =\max P$}), \\
\TT_{\alpha_1,  \beta_1 } & (\text{ otherwise }).
\end{array}
\right.
\] 

\end{lem}
\begin{proof}
Throughout the proof, let $(y,w)$ be an element of  $\pi_Q^{-1}(D_{\alpha\beta})$. We also assume $(y,w)\in \nu_P$ since otherwise $\delta'_{PQ}(y,w)=*$. We use the notations $(x_\gamma,v_\gamma), \ (x'_\gamma,v'_\gamma)$ and $(\bar x_\gamma, \bar v_\gamma)$ from the proof of the  previous lemma. 
 We shall show the claim in the case that $\beta_1 $ is the maximum and $\alpha\not=\beta$. We may assume $\alpha=\alpha_1 $ since general subdivisions factor through this case. Intuitively speaking, since $\beta_1$ is the maximum, the $\beta_1$-component of $e_P$ is the fixed point  $((1-c_{\geq \beta_1}/2,0,0),(1/2,0,0))$, which equals $(\bar x_{\beta_1},\bar v_{\beta_1}$).   Since $\beta\subset \beta_1$,    the point $x_\beta$ is also pinned near $(1-c^v_{\geq \beta}, 0,0)$. Since $\alpha_1=\alpha$, $(\bar x_{\alpha_1}, \bar v_{\alpha_1})=(x_\alpha,v_\alpha)$.  The distance $|x_\alpha-x_\beta|$ is bounded by $d_{\alpha\beta}(Q,v)$, which is smaller than $c^v_{\alpha\beta}$, so
\begin{spacing}{1.2}
\noindent together with the observation on $x_\beta$, the first coordinate of $\bar x_\alpha-c_\alpha \bar v_\alpha$ must be larger than \\$1-(c^{\bar v}_{\geq \alpha}+c^{\bar v}_\alpha)+\eps_P$, the defining upper bound of $\nu_P$. We shall give a rigorous proof. 
By definition of $e_{PQ}$, $v'_\beta=(1/2,0,0)$. So,  by the inequality (\ref{Edifference}) of the proof of the previous lemma, $|v_\beta|\geq 1/2-2\eps_P$. We have
\end{spacing}
\[
 c^v_{\alpha \beta}-\eps_Q    \geq \frac{|v_\beta|}{20(n+2)}-\eps_Q\geq \frac{1}{10(n+2)}\bigl(\frac{1}{4}-\eps_P\bigr)-\eps_Q>\bar\eps_Q.
\]
This implies $d_{\alpha\beta}(Q,v)=c_{\alpha\beta}^v-\eps_Q$.
Again by definition of $e_{PQ}$, $x'_\beta=(1-\frac{1}{2} c_{\geq \beta},0,0)$. So
\[
\begin{split}
\bar x_\alpha-c_\alpha \bar v_\alpha=x_\alpha- &c_{\alpha}v_\alpha =  x'_\beta+(x_\alpha-x_\beta)-(x'_\beta-x_\beta)- c_{\alpha}v_\alpha \\
 &\geq _11-\frac{1}{2} c_{\geq \beta} -|x_\beta-x_\alpha|-|x'_\beta-x_\beta|- c_\alpha^v \\
&\geq 1-\frac{1}{2} c_{\geq \beta}-( c^v_{\alpha\beta}-\eps_Q)-2\eps_P- c^v_\alpha \\
&> 1 - (c^{\bar v}_{\geq \alpha}+c^{\bar v}_\alpha)+\eps_P. 
\end{split}
\]
Here we use the inequality (\ref{Edifference}) in the previous proof and the equality  $\bar v_{\gamma'}=v'_{\gamma}$ for $\gamma\subset \gamma'$. 
This inequality implies that the conditions $(y,w)\in\pi_Q^{-1}(D_{\alpha\beta})$ and  $(y,w)\in \nu_P$ do not hold simultaneously. Thus,  we have $\delta'_{PQ}(\TT_{\alpha\beta})=*$. 
We shall show the claim in the case that $\beta_1 $ is the maximum and $\alpha=\beta$. By the inequality (\ref{Edifference}) (in the previous proof), the condition $(y,w)\in \nu_P$ implies $|v_\beta|\geq \frac{1}{2}-2\eps_P$, which contradicts $|v_\beta|\leq \bar \eps_Q$, so we have $\delta'_{PQ}(\TT_{\alpha\beta})=*$. The case that $\alpha_1 $ is the minimum is completely similar.\\
\indent We shall show the remaining part of the claim.  In the rest of this proof, we assume that $\alpha_1 $ is not the minimum and $\beta_1 $ is not the maximum. We shall show the claim in the case $\alpha\not=\beta$ and $\alpha_1 =\beta_1 $. Intuitively speaking, by definition of the map $e_{PQ}$, the distance $|x'_\alpha-x'_\beta|$ is equal to $c_{\alpha\beta}|\bar v_{\alpha_1}|$ or more.  The distances between $x'_\alpha$ and $x_\alpha$, $x'_\beta$ and $x_\beta$ are sufficiently small. If $c^v_{\alpha\beta}-\eps_Q> \bar \eps_Q$,   $c_{\alpha\beta}|\bar v_{\alpha_1}|$ is larger than the upper bound $d_{\alpha\beta}(Q,v)$ of the distance $|x_\alpha-x_\beta|$, and the difference is too large to be covered by the distances between the points so we see $\pi_Q(y,w)\not \in D_{\alpha\beta}$, which is a contradiction. If $c^v_{\alpha\beta}-\eps_Q\leq \bar \eps_Q$, it follows that $|\bar v_{\alpha_1}|$ is relatively small. Together with $\bar \eps_Q\ll \bar \eps_P$, we have $|\bar v_{\alpha_1}|\leq \bar\eps_P$ or equivalently, $\pi_P(y,w)\in D_{\alpha_1,\alpha_1}$. We shall give a rigorous proof. Consider the case of  $ c_{\alpha\beta}^v-\eps_Q>\bar \eps_Q$. By  definition of the map $e_{PQ}$, the distance between $x'_\alpha$ and $x'_\beta$ is $\geq$ $ c^{v'}_{\alpha\beta}$. 
Together with the inequality (\ref{Edifference}), we have the following inequality.
\[
\begin{split}
|x_\alpha-x_\beta| & \geq |x'_\alpha-x'_\beta|-|x_\alpha-x'_\alpha|-|x_\beta-x'_\beta| \\
                          & \geq  c^{v'}_{\alpha\beta}-4\eps_P \\
                          & \geq  c^{v}_{\alpha\beta}-2\eps_P/10-4\eps_P 
>d_{\alpha\beta}(Q,v).
\end{split}
\]
This inequality implies $\pi_{Q}(y,w)\not\in D_{\alpha\beta}$, which contradicts the assumption so $\delta_{PQ}(y,w)=*$. Consider the other case $ c_{\alpha\beta}^v-\eps_Q\leq \bar \eps_Q$. By definition, we have
\[
\frac{(|\alpha||v_\alpha|+|\beta||v_\beta|)}{20(n+2)}\leq \eps_Q+\bar \eps_Q.
\]
By this inequality and the inequality (\ref{Edifference}), we have
\[
A(|\bar v_{\alpha_1 }|-2\eps_P)\leq \eps_Q+\bar \eps_Q\qquad \text{where}\qquad  A=\frac{(|\alpha|+|\beta|)}{20(n+2)}.
\]
So
\[
|\bar v_{\alpha_1 }|\leq 2\eps_P+\frac{\eps_Q+\bar\eps_Q}{A} <\bar\eps_P.
\]
Thus, we have $\pi_P(y,w)\in D_{\alpha_1 ,\alpha_1 }$.\\
\indent Secondly, we shall show the case of $\alpha=\beta$. Since $\pi_Q(y,w)\in D_{\alpha\alpha}$, we have $|v_\alpha|<\bar \eps_Q$.
So
\[
|\bar v_{\alpha_1 }|=|v'_\alpha|\leq |v_\alpha|+2\eps_P<\bar \eps_Q+2\eps_P<\bar\eps_P,
\]
which means $\pi_P(y,w)\in D_{\alpha_1 ,\alpha_1 }$.\\
\indent Thirdly we consider the case of  $\alpha_1 \not=\beta_1 $. Suppose  that $ c^v_{\alpha\beta}-\eps_Q>\bar \eps_Q$.
 By definition, we have $e_{PQ}(\bar x_{\gamma},\bar v_\gamma)=(x'_\gamma,v'_\gamma)$.  By the inequality (\ref{Edifference}) again, we have
\[
\begin{split}
|x'_\alpha-x'_\beta|\leq  &  |x'_\alpha-x_\alpha|+|x_\alpha-x_\beta|+|x_\beta-x'_\beta| \\
\leq & 4\eps_P+d_{\alpha\beta}(Q,v)
\leq  c^{v'}_{\alpha\beta}-\eps_Q+6\eps_P.
\end{split}
\]
 By the definition of $e_{PQ}$, we have 
\[
|\bar x_{\alpha_1 }-x'_\alpha|\leq  (c^{v'}_{\alpha_1 }-c^{v'}_{\alpha}).
\]
Together with the similar inequality for $\beta$, we see
\[
\begin{split}
|\bar x_{\alpha_1 }-\bar x_{\beta_1 }|\leq  & |\bar x_{\alpha_1 }-x'_\alpha|+|x'_\alpha-x'_\beta|+ |\bar x_{\beta_1 }-x'_\beta| \\
\leq & (c^{v'}_{\alpha_1 }-c^{v'}_{\alpha})+( c^{v'}_{\alpha\beta}-\eps_Q+6\eps_P)+ (c^{v'}_{\beta_1 }-c^{v'}_{\beta})\\
<&   d_{\alpha_1 \beta_1 }(P,\bar v).
\end{split}
\]
Suppose alternatively, $ c^v_{\alpha\beta}-\eps_Q\leq \bar \eps_Q$. By this condition, we easily see
\[
|v_\alpha|, |v_\beta|\leq 20(n+2)(\eps_Q+\bar \eps_Q)
\]
and 
\[
c^{\bar v}_{\alpha_1 }-c^{v'}_{\alpha} =\frac{ (|\alpha_1 |-|\alpha|)|\bar v_{\alpha_1 }|}{20(n+2)}\leq n(\eps_Q+\bar \eps_Q)+ \eps_P/10.
\]
With the similar inequality for $\beta$, we have
\[
\begin{split}
|\bar x_{\alpha_1 }-\bar x_{\beta_1 }|\leq  & |\bar x_{\alpha_1 }-x'_\alpha|+|x'_\alpha-x'_\beta|+ |\bar x_{\beta_1 }-x'_\beta| \\
\leq & c^{\bar v}_{\alpha_1 }-c^{v'}_{\alpha}+(\bar \eps_Q+4\eps_P)+ c^{\bar v}_{\beta_1 }-c^{v'}_{\beta}\\
<&   2n(\eps_Q +\bar \eps_Q)+2\eps_P/10+(\bar \eps_Q+4\eps_P)<\bar \eps_P.
\end{split}
\]
Thus, we have shown the claim.
\end{proof}
\begin{defi}\label{DThom}

 For $P\in \PP_n$, we define a space $\TT(P)\in \CG_*$ by 
\[
\TT(P)=\TT_{\emptyset_P} / \TT_{fat},\quad \text{where} \quad \TT_{fat}=\bigcup_{\alpha,\beta \in P^\circ, \alpha\leq\beta}\TT_{\alpha\beta}.
\]
 If $Q$ is a subdivision of $P$, the map $\delta'_{PQ}:\TT_{\emptyset_Q}\to \TT_{\emptyset_P}$ in Lemma \ref{Ldiagonal_bound} induces the map $\TT(Q)\to \TT(P)$ by Lemma \ref{Ldiagonal_incl}.  These spaces and maps form a functor $\TT:(\PP_n)^{op}\to \CG_*$, which we call the {\em Thom space model}. 
\end{defi}
We shall clarify the  homotopy type of $\TT(P)$.
\begin{defi}
Let $V$ be an affine space and $W\subset V$ an $m$-dimensional affine subspace. Let $N_1, N_2$ be two subsets of $V$. We say $(W, N_1, N_2)$ is an {\em $m$-dimensional spherical triple} if the following conditions hold:
\begin{enumerate}
\item $W\subset N_1$,
\item $N_2$ is a regular star-shaped subset of $V$ with a regular center in $W$, and
\item the inclusion $(W\cap N_2, W\cap \partial N_2)\to (N_1\cap N_2, N_1\cap \partial N_2)$ is a relative homotopy equivalence.
\end{enumerate}
Here, a map $f:(X,A)\to (Y,B)$ of pairs is a {\em relative homotopy equivalence} if it induces a homotopy equivalence $X/A\to Y/B$ between the quotient spaces.
\end{defi}
In the main examples we deal with, we take $V=\RR^{6n}$. By definition, $W\cap N_2$ is a  regular star-shaped subset of $W$, which easily implies the following lemma.
\begin{lem}\label{Lspherical}
For   an $m$-dimensional spherical triple $(W,N_1,N_2)$, the spaces $ (W\cap N_2)/(W\cap \partial N_2)$ and $(N_1\cap N_2)/(N_1\cap \partial N_2)$ are homotopy equivalent to the $m$-dimensional sphere $S^m$. \hfill\qedsymbol
\end{lem}

\begin{lem}\label{Lstarshaped_T}
For $P\in \PP_n$ with $|P^\circ|=p$ and $\alpha, \beta\in P^\circ$, we set
\[
D_{\alpha\beta}^0=\left\{
\begin{array}{cc}
\{(x_\gamma,v_\gamma)\in \RR^{6p} \mid x_\alpha=x_\beta\} & if\ \ \alpha\not=\beta, \\
\{(x_\gamma,v_\gamma)\in \RR^{6p} \mid v_\alpha=0\} & if\ \ \alpha=\beta.
\end{array}
\right.
\]
For a graph $G\in \GG(P)$ (see Definition \ref{Dbar_Sigma}), we set $D_G=\cap_{(\alpha,\beta)\in E(G)}D_{\alpha\beta}$  and $D_G^0=\cap_{(\alpha,\beta)\in E(G)}D_{\alpha\beta}^0$.

The neighborhood $D_G$ of $D_G^0$ is fiberwise star-shaped, i.e. if $\pi^\vee:D_G\to D^0_G$ is the orthogonal projection, the fiber of $\pi^\vee$ over a point in $D^0_G$ is star-shaped and has the point in $D^0_G$ as its center.  
\end{lem}
\begin{proof} 
This might not be obvious since the width $d_{\alpha\beta}(P,v)$ of the neighborhood depends on $v$. 
Suppose $(x,v)\in D_G$. Let $(x^t,v^t)=t(x,v)+(1-t)\pi^\vee(x,v)$ for $0\leq t\leq 1$, so $(x^1,v^1)=(x,v)$.
If $(\alpha,\alpha)\in G$, $v^t$ slides $v_\alpha$   to zero, otherwise it does not change $v_\alpha$.  If $  c^v_{\alpha\beta}-\eps_P\leq \bar\eps_P$ for all $(\alpha,\beta)\in G$ with $\alpha\not=\beta$, the point $(x^t,v^t)$ clearly belongs to $D_G$. Suppose $ c_{\alpha\beta}^v-\eps_P>\bar \eps_P$ for some $(\alpha,\beta)\in G$.  We may assume $|v_\alpha|\geq |v_\beta|$, which implies $|v_\alpha|\geq 10\bar\eps_P$ and $(\alpha,\alpha)\not\in G$.  
It follows that  the $\alpha$-component of $v^t$ 
 is the same as $v$, and  we have
\[
\begin{split}
(  c_{\alpha\beta}^{v^t}-\eps_P)-t( c^v_{\alpha\beta}-\eps_P)& 
\geq (1-t)(  c_\alpha^v-\eps_P)\\
&
\geq (1-t)\Bigl(\frac{\bar \eps_P}{2(n+2)}-\eps_P\Bigr)\geq 0.\\
\therefore
|x_\alpha^t-x_\beta^t|=t|x_\alpha-x_\beta|  &  \leq td_{\alpha\beta}(P,v)<d_{\alpha\beta}(P, v^t).
\end{split}
\]
Thus,  $(x^t,v^t)\in D_G$.
\end{proof}
\begin{lem}\label{Lfunctor_neat}
With the notation of Lemma \ref{Lstarshaped_T}, the triple $(\pi_P^{-1}(D_G^0), \pi_P^{-1}(D_G), \bar \nu_P)$ is a $(6n-3c(G))$-dimensional spherical triple, where $\bar \nu_P$ is the closure of $\nu_P$, and $c(G)$ is the complexity, see Definition \ref{Dbar_Sigma}.
\end{lem}
\begin{proof}
We use the notations of the proof of the previous lemma. We extend the map $\pi^\vee$ to a map $\pi_P^{-1}(D_G)\to \pi_P^{-1}(D_{G}^0)$  by setting
\[
\pi^\vee(y,w):=e_P\{\pi^\vee(\pi_P(y,w))\}+(y,w)-e_P(\pi_P(y,w)).
\] 
The $\pi^\vee$ on the right-hand side is the one defined in the previous lemma. Let $(y,w)\in \pi_P^{-1}(D_G)\cap \partial \nu_P$. We set $(y^t,w^t)=t(y,w)+(1-t)\pi^\vee(y,w)$. By  the previous lemma, we have $(y^t,w^t)\in \pi^{-1}_P(D_G)$.  We first show $e_P(0)\not=(y^t,w^t)$ for any $t$. Intuitively, this is because the width of the regular neighborhood  $D_G$ is sufficiently small. For $\alpha\in P^\circ$,  we have 
\[
|\pi_P(y,w)_\alpha-\pi^\vee(\pi_P(y,w))_\alpha|\ \leq \ \frac{1}{| C|}\sum_{\beta\in C} \tilde d_{\alpha\beta}(P,v)+\bar \eps_P\ \leq \ 
\frac{1}{20}+\bar\eps_P,
\]
where for $z\in \RR^{6p}$,  $z_\alpha\in \RR^6$ denotes the component labeled with $\alpha$, and $|C|$ denotes the number of vertices in the connected component $C$ of $G$ including the vertex $\alpha$, and $\tilde d_{\alpha\beta}(P,v)$ is the minimum of $d_{\alpha,\gamma_1}(P,v)+d_{\gamma_1,\gamma_2}(P,v)+\cdots +d_{\gamma_a,\beta}(P,v)$ for the paths $(\alpha,\gamma_1,\dots, \gamma_a,\beta)$ in $G$ with $v$ being the velocity component of $\pi_P(y,w)$.   
Therefore, by definition of $\pi^\vee$, for $i\in \alpha$ we have
\[
|(y_i,w_i)-\pi^\vee(y,w)_i|=|e_{P}(\pi_P(y,w))_i-e_{P}(\pi^\vee(\pi_P(y,w)))_i|\leq (1+1/10)(\frac{1}{20}+\bar \eps_P)<1/2.
\]
Here, the subscript $i$ denotes the $i$-th component.
By this inequality, if $|y_i|=1$ or $|w_i|=1$, $|(y^t_i,w^t_i)|\geq 1/2$. In the case that one of the defining inequalities of $\nu_P$ on the first coordinate ( the inequalities involving $<_1$) becomes an equality, we see $y^t\not=0$ similarly. If the distance between $(y,w)$ and the image of $e_P$ is equal to $\eps_P$, the distance of $(y^t,w^t)$ from the image is also $\eps_P$ since $e_P\{\pi^\vee(\pi_P(y,w))\}-e_P(\pi_P(y,w))$ is parallel to the image of $e_P$. Thus,  we have $e_P(0)\not=(y^t,w^t)$ under the assumption $(y,w)\in\partial\nu_P\cap \pi_P^{-1}(D_G)$. \\
\indent By Lemma \ref{Lnu_P_star}, there is a unique $s>0$ such that $(y^{st},w^{st}):=(1-s)e_P(0)+s(y^t,w^t)\in \partial  \nu_{P}$ (see Figure \ref{Fretraction}). We shall show this point is also in $\pi_P^{-1}D_G$, which implies that we can define a retracting homotopy $h : (\pi_P^{-1}D_G)\cap\partial \nu_P\to (\pi_P^{-1}D^0_G)\cap\partial \nu_P$ on the boundary by $h((y,w),t)=(y^{st},w^{st})$ since $(y^{s,0},w^{s,0})\in \pi_P^{-1}D^0_G$.
We put $( x , v )=\pi_P(y,w)$, $( x ^t, v ^t)=\pi_P(y^t,w^t)$ and $(x^{st},v^{st})=\pi_P(y^{st},w^{st})$. Clearly, $( x ^t, v ^t)=t( x , v )+(1-t)\pi^\vee( x , v )$. \\
\indent (i) Let $(\alpha,\beta)\in G$ with $\alpha\not=\beta$. \\
\indent (i-a) Suppose $ c_{\alpha\beta}^{ v }-\eps_P\geq \bar \eps_P$. In this case, by the proof of the previous lemma, we have $| x ^t_\alpha- x ^t_\beta|\leq  c^{ v ^t}_{\alpha\beta}-\eps_P$.  If $s\geq  1$, obviously
\[
| x ^{st}_\alpha- x ^{st}_\beta|=s| x ^t_\alpha- x ^t_\beta|\leq s(c_{\alpha\beta}^{ v ^t}-\eps_P)\leq  sc^{ v ^t}_{\alpha\beta}-\eps_P\leq  c^{ v ^{st}}_{\alpha\beta}-\eps_P.
\]
So we assume $0<s<1$. 
If a number $i$ belongs to the minimum or maximum piece, the $i$-th component of the image $e_P(\RR^{6p})$ consists of   one point so the $i$-th components of $e_P(\pi_P(y,w))$ and $e_P(\pi^\vee \pi_P(y,w))$ are the same, and the $i$-th component of $(y^t,w^t)$ is stationary for $t$. Therefore,  $|y^t_i|<1$ and $|w^t_i|<1 $. By this observation, at least one of the following cases occurs :
\begin{enumerate}
\item[(i-a-1)] $|sy^t_i|=1$ or $|sw^t_i|=1$ for some $i\in P^\circ$ (see Definition \ref{Dpartition0}), \vspace{2mm}
\item[(i-a-2)] $s( x ^t_\gamma- c_\gamma  v ^t_\gamma)=_1 1- (c_{\geq \gamma}^{  v ^{st}}+c_\gamma^{v ^{st}})+\eps_P$ or $s ( x ^{t}_\gamma+ c_\gamma  v ^t_\gamma)=_1-1+ (c_{\leq \gamma}^{v ^{st}}+c_\gamma^{v ^{st}})-\eps_P$ for some $\gamma\in P^\circ$,  or\vspace{2mm}
\item[(i-a-3)] the distance between $(y^{st},w^{st})$ and the image of $e_P$ is equal to $\eps_P$.
\end{enumerate}

 We shall consider the  first case (i-a-1), in particular the case of $|sy^t_i|=1$ since the other case is similar. By the argument in the proof of $(y^t,w^t)\not=e_P(0)$, we have 
$
|y^t_i|\leq 1+(1-t)d
$
where
$
d=(1+1/10)(\frac{1}{20}+\bar \eps_P), 
$ 
 so
\[
1=|sy^t_i|\leq s(1+(1-t)d)\qquad \therefore \quad s\geq \frac{1}{1+(1-t)d}.
\]

Since $| x ^{st}_\alpha- x ^{st}_\beta|=|s x ^t_\alpha-s x ^t_\beta|\leq std_{\alpha\beta}(P, v )$, we want to show $  c^{v ^{st}}_{\alpha\beta}-\eps_P-st(c^{ v }_{\alpha\beta}-\eps_P)\geq 0$.  We may assume $| v _\alpha|\geq | v _\beta|$. By $  c_{\alpha\beta }^{ v }-\eps_P>\bar \eps_P$, we have $| v _\alpha|\geq 10\bar \eps_P$ and $v_\alpha^t=v_\alpha$. Then we have 
$
 c^{ v ^t}_{\alpha\beta}\geq  c^{ v }_\alpha
$ and 
\[
\begin{split}
 c^{ v ^{st}}_{\alpha\beta}-\eps_P-st( c^{ v }_{\alpha\beta}-\eps_P)& 
\geq  s(1-t) c^{ v }_\alpha-(1-st)\eps_P \\
&\geq \frac{(1-t)A}{1+(1-t)d}-\left(1-\frac{t}{1+(1-t)d}\right)\eps_P  \qquad (\text{where we set  $A = c^{ v }_\alpha$})\\
& = \frac{ A (1-t)}{1+(1-t)d}-\frac{1+(1-t)d-t}{1+(1-t)d}\eps_P \\
&=\frac{(- A +(d+1)\eps_P)t+ A -(d+1)\eps_P}{1+(1-t)d}.
\end{split}
\]
Clearly, we have
\[
\frac{\eps_P(d+1)}{ A }\leq \frac{(n+2)(1+1/10)(1/10 +2\bar \eps_P)\eps_P}{\bar \eps_P}<1.
\] Therefore,  we have $ A >\eps_P(d'+1)$ and  have proved the desired inequality. For the  case (i-a-2), suppose $ s(x^{t}_\gamma-c_\gamma v_\gamma)=_11- (c_{\geq \gamma}^{ v ^{st}}+c_\gamma^{v^{st}})+\eps_P$.  Put $T(v)=1- (c_{\leq \gamma}^{ v }+c^{ v }_\gamma)+\eps_P$. We have 
\[
T( v)-(1-t) \bar \eps_P\leq  T(v^t)\leq T(v^{st})=_1 s( x ^{t}_\gamma- c_\gamma v ^t_\gamma)\leq_1 s(T(v)+2(1-t)d),
\]
where $d$ is the number given in the previous case. By setting $d'=2d/T( v),\ \eps'=\bar \eps_P/T(v)$ and $r=1-t$, we have
$
s\geq (1-r\eps')/(1+rd')
$.
By this inequality, we have
\[
\begin{split}
 c^{ v ^{s t}}_{\alpha\beta}-\eps_P-st( c^{ v }_{\alpha\beta}-\eps_P)
& \geq  s(1-t) c^{ v }_{\alpha}-(1-st)\eps_P \\
&\frac{r(1-r\eps') A }{1+rd'}-\left(1-\frac{(1-r)(1-r\eps')}{1+rd'}\right)\eps_P \\
&=\frac{r A -r^2\eps' A -(1+rd'-(1-(1+\eps')r+r^2\eps'))\eps_P}{1+rd'}\\
&=\frac{r^2(-\eps' A +\eps'\eps_P)+r( A -\eps_P(d'+\eps'+1))}{1+rd'}.
\end{split}
\]
We easily see that last formula is $>0$ if $0<r<1$.  For the  case (i-a-3), as mentioned in the former part of this proof, the distance between $(y^t,w^t)$ and $e_P(\RR^{6p})$ is $\leq \eps_P$ so this case cannot occur under the assumption $0<s<1$. \\
\indent (i-b) Suppose $ c^{ v }_{\alpha\beta}-\eps_P<\bar \eps_P$. In this case, if $0<s\leq 1$, we have $| x ^{st}_\alpha- x ^{st}_\beta|\leq s\bar\eps_P\leq \bar \eps_P$ so we may assume $s>1$. Under this assumption, the distance between $(y,w)$ and $e_P(\RR^{6p})$ must be $<\eps_P$. Since we have assumed $(y,w)\in \partial \nu_P$, one of the following cases occurs
\begin{enumerate}
\item[(i-b-1)] $|y_i|=1$ or $|w_i|=1$ for some $i\in P^\circ$.
\item[(i-b-2)] $ x _\gamma- c_\gamma  v _\gamma=_1T(v)$ or $ x _\gamma+ c_\gamma  v _\gamma=_1-1+ (c^{ v }_{\leq \gamma}+c^{ v }_\gamma)-\eps_P$ for some $\gamma \in P^\circ$.
\end{enumerate}
In the first case, say $|y_i|=1$. We have $s(1-(1-t)d)\leq |sy_i^t|\leq 1$ so $s\leq 1/(1-(1-t)d)$. We have
\[
| x ^{st}_\alpha- x ^{st}_\beta|\leq st| x _\alpha- x _\beta|\leq \frac{t}{1-(1-t)d}\bar \eps_P<\bar \eps_P.
\]
In the second case, say $ x _\gamma- c_\gamma v _\gamma =_1T(v)$, we have $s(T(v)-2(1-t)d)\leq _1s( x _\gamma^t- c_\gamma  v ^t_\gamma)\leq _1T(v)+(1-t) \bar \eps_P/10$ and  $st\leq \frac{t(1+(1-t)\eps')}{1-(1-t)d'}<1$. \\
\indent (ii) For the case $(\alpha,\alpha)\in G$, we can verify the codition $| v ^{st}_\alpha|\leq \bar \eps_P$ similarly to the above case $ c^{ v }_{\alpha\beta}-\eps_P< \bar \eps_P$. \\
\indent Thus, we have proved $(y^{st},w^{st})\in \pi_P^{-1}D_G$ in any case. 
 The whole retracting homotopy\\ $(\pi_P)^{-1}(D_G)\cap \bar\nu_P\to (\pi_P)^{-1}(D^0_G)\cap \bar\nu_P$ is constructed by taking a `collar' (regular neighborhood) of $\pi_P^{-1}(D_G)\cap \partial \nu_P$ in $\pi_P^{-1}(D_G)\cap \bar \nu_P$ such that the image by  $\pi^{\vee}$ of the part of boundary of the collar which is not a part of $\partial \nu_P\cup \partial \pi_P^{-1}(D_G)$ is included in $\bar \nu_P$ and connecting $(1-t)id+t\pi^{\vee}$ and the constructed homotopy on the boundary through this collar.
\end{proof}
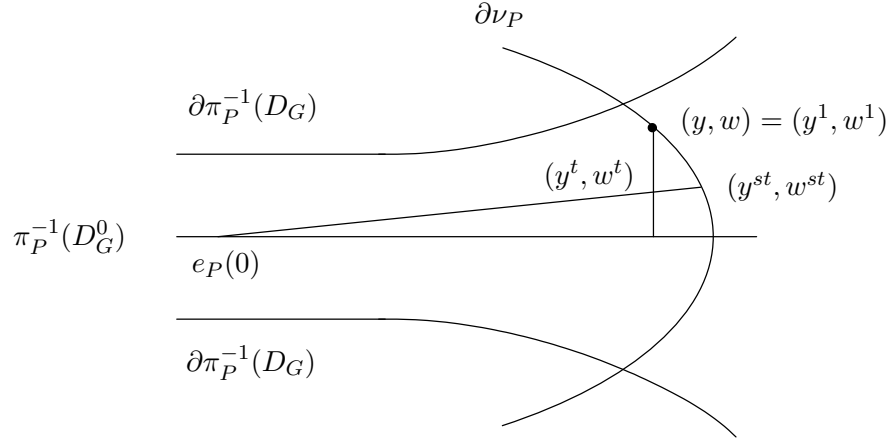
\begin{figure}
\begin{center}
{\unitlength 0.1in%
\begin{picture}(41.3600,22.7500)(0.2000,-23.0500)%
%
\special{pn 8}%
\special{pa 1121 1256}%
\special{pa 4156 1256}%
\special{fp}%
%
\special{pn 8}%
\special{ar 1337 1256 2592 1209 5.3252827 0.9592651}%
%
\special{pn 8}%
\special{pa 1121 1688}%
\special{pa 2212 1688}%
\special{fp}%
%
\special{pn 8}%
\special{pa 2178 1688}%
\special{pa 2274 1688}%
\special{pa 2370 1691}%
\special{pa 2434 1695}%
\special{pa 2530 1704}%
\special{pa 2561 1708}%
\special{pa 2593 1712}%
\special{pa 2657 1722}%
\special{pa 2688 1727}%
\special{pa 2720 1732}%
\special{pa 2752 1738}%
\special{pa 2783 1744}%
\special{pa 2815 1751}%
\special{pa 2846 1757}%
\special{pa 2877 1764}%
\special{pa 2909 1772}%
\special{pa 2940 1779}%
\special{pa 3002 1795}%
\special{pa 3033 1804}%
\special{pa 3063 1812}%
\special{pa 3094 1821}%
\special{pa 3125 1831}%
\special{pa 3155 1840}%
\special{pa 3186 1850}%
\special{pa 3216 1860}%
\special{pa 3247 1870}%
\special{pa 3337 1903}%
\special{pa 3397 1927}%
\special{pa 3426 1939}%
\special{pa 3456 1951}%
\special{pa 3485 1964}%
\special{pa 3515 1976}%
\special{pa 3544 1990}%
\special{pa 3573 2003}%
\special{pa 3660 2045}%
\special{pa 3718 2075}%
\special{pa 3774 2105}%
\special{pa 3802 2121}%
\special{pa 3830 2138}%
\special{pa 3884 2172}%
\special{pa 3911 2191}%
\special{pa 3936 2209}%
\special{pa 3962 2229}%
\special{pa 3987 2249}%
\special{pa 4011 2270}%
\special{pa 4034 2292}%
\special{pa 4048 2305}%
\special{fp}%
%
\special{pn 8}%
\special{pa 1121 824}%
\special{pa 2212 824}%
\special{fp}%
%
\special{pn 8}%
\special{pa 2178 824}%
\special{pa 2274 824}%
\special{pa 2370 821}%
\special{pa 2402 819}%
\special{pa 2434 816}%
\special{pa 2466 814}%
\special{pa 2498 810}%
\special{pa 2529 807}%
\special{pa 2593 799}%
\special{pa 2625 794}%
\special{pa 2656 789}%
\special{pa 2688 784}%
\special{pa 2720 778}%
\special{pa 2751 772}%
\special{pa 2783 766}%
\special{pa 2814 760}%
\special{pa 2846 753}%
\special{pa 2877 746}%
\special{pa 2908 738}%
\special{pa 2939 731}%
\special{pa 2971 723}%
\special{pa 3033 707}%
\special{pa 3095 689}%
\special{pa 3125 680}%
\special{pa 3156 671}%
\special{pa 3187 661}%
\special{pa 3247 641}%
\special{pa 3278 630}%
\special{pa 3308 620}%
\special{pa 3338 609}%
\special{pa 3368 597}%
\special{pa 3398 586}%
\special{pa 3428 574}%
\special{pa 3457 561}%
\special{pa 3487 549}%
\special{pa 3545 523}%
\special{pa 3661 467}%
\special{pa 3689 452}%
\special{pa 3718 437}%
\special{pa 3746 422}%
\special{pa 3774 406}%
\special{pa 3802 389}%
\special{pa 3829 373}%
\special{pa 3856 356}%
\special{pa 3883 338}%
\special{pa 3909 320}%
\special{pa 3936 301}%
\special{pa 3961 282}%
\special{pa 3987 263}%
\special{pa 4011 243}%
\special{pa 4036 222}%
\special{pa 4048 211}%
\special{fp}%
%
\special{pn 8}%
\special{pa 3616 685}%
\special{pa 3616 1256}%
\special{fp}%
\put(43.0000,-6.4500){\makebox(0,0){$(y,w)=(y^1,w^1)$}}%
%
\special{pn 8}%
\special{pa 1337 1256}%
\special{pa 3863 997}%
\special{fp}%
\put(32.8000,-9.3500){\makebox(0,0){$(y^t,w^t)$}}%
\put(42.9000,-9.9500){\makebox(0,0){$(y^{st},w^{st})$}}%
\put(13.8000,-14.0500){\makebox(0,0){$e_P(0)$}}%
\put(15.2000,-5.7500){\makebox(0,0){$\partial\pi_P^{-1}(D_G)$}}%
\put(15.1000,-19.1500){\makebox(0,0){$\partial\pi_P^{-1}(D_G)$}}%
\put(5.6000,-12.5500){\makebox(0,0){$\pi_P^{-1}(D_G^0)$}}%
\put(28.1000,-0.9500){\makebox(0,0){$\partial \nu_P$}}%
%
\special{sh 1.000}%
\special{ia 3610 685 21 21 0.0000000 6.2831853}%
\special{pn 8}%
\special{ar 3610 685 21 21 0.0000000 6.2831853}%
\end{picture}}%
\end{center}
\caption{definition of $(y^{st},w^{st})$ } \label{Fretraction}
\end{figure}
In view of Lemmas \ref{Lfunctor_neat} and \ref{Lspherical}, we have the following lemma.
\begin{prop}\label{PThomhomotopy}
With the notation of Lemma \ref{Lstarshaped_T}, we have the following homotopy equivalences
\[
\cap_{(\alpha,\beta)\in G}\TT_{\alpha\beta}\ \simeq\  (\pi_P^{-1}(D^0_{G})\cap \bar\nu_P)/ (\pi_P^{-1}(D^0_{G})\cap \partial \nu_P) \simeq\  S^{6n-6p}\wedge \{D^0_G\cap (D^3)^{2p}\}/\{D^0_G\cap \partial (D^3)^{2p}\}.
\]
Here, the first equivalence is induced by Lemmas \ref{Lspherical} and \ref{Lfunctor_neat} and the second one by the orthogonal decomposition $\RR^{6n}=e_P(\RR^{6p})\oplus \RR^{6n-6p}$. By  putting these  equivalences into together for all $G\in \GG(P)$, we have the following homotopy equivalence
\[
\TT(P)\ \simeq\  \bar\nu_P/\{(\cup_{(\alpha,\beta)}\pi_P^{-1}(D^0_{\alpha\beta})\cap \bar\nu_P)\cup \partial \nu_P\}\ \simeq\  S^{6n-6p}\wedge(D^3)^{2p}/\{(Fat\times (D^3)^p)\cup ((D^3)^p\times V_0)\cup \partial (D^3)^{2p}\}.
\]
Here, $(\alpha,\beta)$ runs through the range of $\alpha,\beta\in P^\circ$, $\alpha\leq\beta$, and $Fat$ and $V_0$ are the subspaces defined in the beginning of this subsection.\hfill\qedsymbol
\end{prop}
\subsection{Equivalence between   $\PK$ and $\TT$}
\begin{lem}\label{Lnathomologyiso}
The natural transformation $\mathcal{C}\to \PK$ defined in Definition \ref{Dconfig_model}  induces a homology isomorphism at each $P\in \PP_n$.
\end{lem}
\begin{proof}
Let 
$\tilde{\mathcal{C}}(P)$ be the set of $(x,v)\in \RR^{6p}$ satisfying all the defining conditions of $\mathcal{C}(P)$ but $|x_\alpha-x_\beta|>  c_{\alpha\beta}^v-\eps_P/8$. For $\alpha\not=\beta$, we set $\tilde D_{\alpha\beta}=\{(x,v)\in \tilde{\mathcal{C}}(P)\mid |x_\alpha-x_\beta|\leq  c_{\alpha\beta}^v-\eps_P/8\}$ and \\ 
$\tilde D_{\alpha\beta}^0=\{(x,v)\in \tilde{\mathcal{C}}(P)\mid x_\alpha=x_\beta\}$. For a graph $G$ without loops, we set $\tilde D_G=\cap_{(\alpha,\beta)\in G}D_{\alpha\beta}$ and $\tilde D^0_G=\cap_{(\alpha,\beta)\in G}D^0_{\alpha\beta}$.
Similarly to the proof of Lemma \ref{Lfunctor_neat}, we see that the pair $(\tilde D_G, \tilde D_G\cap \partial \tilde{\mathcal{C}(P)})$ is relatively homotopy equivalent to $(\tilde D_G^0, \tilde D_G^0\cap \partial \tilde{ \mathcal{C}}(P))$.  This implies that the quotient space $\tilde D_G/(\tilde D_G\cap\partial\tilde{\mathcal{C}}(P))$ is homotopy equivalent to the sphere $S^{6p-3c(G)}$. With this fact, the proof is completely similar to Lemma 2.7 of  \cite{moriya3}.
\end{proof}

\begin{defi}\label{Dspectrafunctor} In this definition, we regard the sphere $S^k$ as the compactification $(\RR^k)^*$.
\begin{enumerate}

\item A {\em spectrum} $X$ is a sequence of pointed spaces $X_0, X_1, \dots$ with a  structure map $S^1\wedge X_k\to X_{k+1}$ for each $k\geq 0$. A {\em morphism (or map)} $f:X\to Y$ of spectra is a sequence of pointed maps $f_0:X_0\to Y_0, f_1:X_1\to Y_1,\dots$ compatible with the structure maps. Let $\SP$ denote the category of spectra and their maps. For a spectrum $X$,  $\pi_k(X)$ denotes the colimit of the sequence $\pi_k(X_0)\to \pi_{k+1}(X_1)\to \cdots $ defined by the structure maps. A map $f:X\to Y$ is called a {\em stable homotopy equivalence} if it induces an isomorphism $\pi_k(X)\to \pi_k(Y)$ for any integer $k$.
\item For a spectrum $X$ and unpointed space $U$, we define a spectrum $\Map(U,X)$ as follows. We define  $\Map(U,X)_k$ as the space  of (unpointed) continuous maps $U\to X_k$ with the compact-open topology. The basepoint is the constant map to the basepoint of $X_k$. The structure map is the one obviously induced by that of $X$.
\item We define a functor $\TT^S: \PP_n^{op}\to \SP$ as follows. Set $\TT^S(P)_k=S^{k-6n}\wedge \TT(P)$ if $k\geq 6n$,  and $\TT^S(P)_k=*$ otherwise. These spaces form a spectrum with the obvious structure map. The map corresponding to a map $P\to Q$ is also obviously induced from that of the Thom space model $\TT$. 
\item For a positive number $\delta$, we define a spectrum $\Sphere_\delta$ as follows. We set $\Sphere_{\delta, k}=\{ y\in \RR^k\}/\{y \mid |y|\geq \delta\}$. The structure map $S^1\wedge \Sphere_{\delta, k}\to \Sphere_{\delta, k+1}$ is the obvious collapsing map.
\item We define a functor $\mathcal{C}^\dagger:\PP_n^{op}\to \SP$ as follows.
Set $\mathcal{C}^{\dagger}(P)=\Map(\mathcal{C}(P),\Sphere_{\delta})$ where  $\delta =\eps_P/8$ (see Definition \ref{Dpunctured}). For a map $P\to Q$, the corresponding map is the pullback by the induced map $\mathcal{C}(P)\to \mathcal{C}(Q)$ followed by the pushforward by the collapsing map $\Sphere_{\eps_Q/8}\to \Sphere_{\eps_P/8}$.

\item We define a functor $\mathcal{C}^\vee:\PP_n^{op}\to \SP$ as follows. Let $\Sphere$ denote the sphere spectrum given by $\Sphere_k=S^k$.
Set $\mathcal{C}^{\vee}(P)=\Map(\mathcal{C}(P),\Sphere)$.   For a map $P\to Q$, the corresponding map is the pullback by the induced map.
\item We define a map $\TPhi= \TPhi_{P,k}:\RR^k\to \mathcal{C}^\dagger (P)_k$ by
\[
\RR^k\ni y\longmapsto \{(x_\gamma,v_\gamma) \mapsto (y-(0,e_P(x_\gamma,v_\gamma))\}\in\mathcal{C}^\dagger(P)_k.
\]
$\TT^S(P)_k$ is naturally identified with Thom space associated to the regular neighborhood  $\RR^{k-6n}\times \nu_P$ of the embedding $0\times e_P:\RR^{6p}\to \RR^k$ (with some extra collapsed points). $\TPhi_{P,k}$ factors through $\TT^S(P)_k$ as in Theorem  \ref{TAtiyahdual}, and  these maps form a natural transformation $\Phi : \TT^S\to \mathcal{C}^\dagger$. We see that this is well-defined below.
\item A natural transformation $p_* :\mathcal{C}^\vee\to \mathcal{C}^\dagger$ is defined by the pushforward by the obvious collapsing map $p :\Sphere\to \Sphere_{\delta}$. 
\end{enumerate}
\end{defi}
The following equivalence is a variation of the one given in \cite{moriya3}. If it is projected to the stable homotopy category, it is a special case of Atiyah duality which states an equivalence between the Spanier-Whitehead dual of a manifold and Thom spectrum of its normal bundle. We need point-set level compatibility so we have been taking care about parameters. 
\begin{thm}\label{TAtiyahdual}
Under the notations of Definition \ref{Dspectrafunctor},
the natural transformation $\Phi$ is well-defined, and the two natural transformations $\Phi$ and $p_*$ are termwise stable homotopy equivalences (i.e. they induce a stable homotopy equivalence at each object).
\end{thm}
\begin{proof}
We shall show the map $\TPhi$ factors through $\TT^S(P)_k$. For notational simplicity, we consider the case of $k=6n$. The other cases will follow completely similarly.  Let $(y,w)\in \RR^{6n}$ be an element with $\TPhi(y,w)\not= *$.  There exists an element $(x_\gamma,v_\gamma)\in \mathcal{C}(P)$ such that  $| (y,w)- e_P(x_\gamma,v_\gamma)|<\eps_P/8$ holds, so  we have $|(y,w)- e_P(\pi_{P}(y,w))|<\eps_P/8$ and
\[
|\pi_P (y,w)-(x_\gamma,v_\gamma)|\leq|e_P(\pi_P y)-e_P(x_\gamma,v_\gamma)|\leq | e_P(\pi_P (y,w))-(y,w)|+|(y,w)-e_P(x_\gamma,v_\gamma)|<\eps_P/4.
\]
 If we write  $\pi_P(y,w)=(\bar x_\gamma,\bar v_\gamma)_\gamma$, it follows that $|(\bar x_\alpha,\bar v_\alpha)-(x_\alpha,v_\alpha)|<\eps_P/4$ for each  $\alpha\in P^\circ$. We have
\[
\begin{split}
\bar x_\alpha - c_\alpha \bar v_\alpha  &  <_1 1- (c^v_{\geq \alpha}+c^v_\alpha)+\frac{(1+ c_\alpha)\eps_P}{4} +\frac{\eps_P}{8}\\
& <_1 1- (c^{\bar v}_{\geq \alpha}+c^{\bar v}_\alpha)+\eps_P. \\
\end{split}
\]

By similar  argument, we can verify the other condition on the first coordinate, and $|y_j|, |w_j|<1$. Thus, we have $(y,w)\in \nu_P$. In other words, $\TPhi(\RR^{6n}-\nu_P)=*$. 
 We also see
\[
\begin{split}
|\bar x_\alpha -\bar x_\beta| & \geq  |x_\alpha-x_\beta|-|x_\alpha-\bar x_\alpha |-|x_\beta-\bar x_\beta| \\
 & >  c^v_{\alpha\beta}-\eps_P/8-\eps_P /2> c^{\bar v}_{\alpha\beta}-\eps_P \\
&>\bar \eps_P \qquad (\because |v_\gamma|\geq 10(n+2)\bar \eps_P).
\end{split}
\]
This implies $\TPhi(\pi_P^{-1}(D_{\alpha\beta}))=*$. Thus, $\TPhi$ factors through $\TT^S$. Now the claim of the theorem follows from the classical Atiyah duality  (see \cite{browder} for example).

\end{proof}
\begin{rem}
The construction of the Thom space model $\TT$ and the proof of Theorem \ref{TAtiyahdual} were hinted by Cohen's construction \cite{cohen}. Applications of Atiyah (or Dold-Puppe) duality to the embedding calculus or the little disks operads also have been  given by Ching-Salvatore \cite{CS} and Malin \cite{malin}.
\end{rem}
\subsection{A resolution of $\TT$}
\begin{defi}\label{Dresol_T}
For a space $X\in\CG$, let $X_+\in\CG_*$ denote  $X$ with a disjoint basepoint. We define a functor $\TT^c:\PP_n^{op}\to \CG_*$ as follows:
\begin{enumerate}
\item For  a graph $G\in\GG(P)$, we define a subspace $\TT_G\subset \TT_{\emptyset_P}$ by
\[
\TT_G=\bigcap_{(\alpha,\beta)\in G} \TT_{\alpha\beta}.
\]
\item For a partition $P\in \PP_n$, let $U^{\TT}(P)\subset \TT_{\emptyset_P}\wedge  (|\GG(P)|_+) $ denote the subspace consisting of (the basepoint and) elements $(x,v;u)$  satisfying $(x,v)\in \TT_{\LL(u)}$, where $\LL(u)$ is the last $\GG(P)$-element of $u$, see subsection \ref{SSNT}.  
\item We set 
\[
\TT^c(P)=U^{\TT}(P)/\sim.
\]
Here,  for an element of a form $(x,v; u=\tau_0G_0+\cdots+\tau_mG_m)$ with $\tau_0\not=0$, we declare that $(x,v;u)\sim *$ \ if and only if $G_0$ is not  the graph with the empty edge set (and there is no other identification given by $\sim$). For a subdivision $Q$ of $P$, the structure map $\TT^c(Q)\to \TT^c(P)$ is given by
\[
(x,v;\tau_0G_0+\cdots +\tau_mG_m)\mapsto (x,v; \tau_0\delta_{PQ}(G_0)+\cdots +\tau_m\delta_{PQ}(G_m)).
\]
See Definition \ref{Dbar_Sigma} for $\delta_{PQ}$. This map is well-defined by Lemma \ref{Ldiagonal_incl}. 
\end{enumerate}

\end{defi}
The following is obvious.
\begin{lem}
The natural transformation $\TT^c\to \TT$ given by collapsing the elements $(x,v; u)$ such that $\LL(u)$ has a non-empty edge set, induces a homotopy equivalence $\TT^c(P)\to \TT(P)$ for each object $P\in \PP_n$. \hfill\qedsymbol
\end{lem}
\begin{defi}\label{Dreal}
Let $X:\Delta_l\to \CG_*$ be a functor. We denote by $|X|$  the {\em (geometric) realization of} $X$, defined by
\[
|X|=\bigvee_{0\leq i\leq l} X([i])\wedge (\Delta^i_+)/\sim
\]
where $\Delta^i$ is the topological standard $i$-simplex and the equivalence relation $\sim$ is generated by $(f^*x,a)\sim (x,f(a))$ with $(x,a)\in X([i])\wedge (\Delta^j_+)$ and $f:[j]\to [i]\in \Delta_l$.
\end{defi}
\begin{defi}\label{Dintermediate_space}
\begin{enumerate}
\item Suppose a positive integer $p(u)$, subspaces $D_u\subset \RR^{6p(u)}$, $N_u\subset \RR^{6n}$, and an affine map $E_u:\RR^{6p(u)}\to \RR^{6n}$ are given for each $u\in |\GP_n|$, see Definition \ref{Dbar_Sigma}. We define a subspace 
\[
\tilde U=\tilde U(D_u,E_u,N_u)\subset \RR^{6n}\times |\GP_n|
\]
by $\tilde U=\{(x,v;u)\mid (x,v)\in \Pi_u^{-1}(D_u)\cap N_u\}$, where $\Pi_u:\RR^{6n}\to E_u(\RR^{6p(u)})=\RR^{6p(u)}$ is the orthogonal projection. (If $D_u=\emptyset$ for an element $u\in |\GP_n|$, we understand that $\tilde U\cap \RR^{6n}\times  \{u\} =\emptyset$.)
We define a quotient space  $U(D_u,E_u,N_u)$ by
\[
 U(D_u,E_u,N_u)=(\tilde U)^*/\sim.
\]
Here $(\tilde U)^*$ is the one-point compactification of $\tilde U$ defined in part 1, and for an element $(x,v;u=\tau_0J_0+\cdots +\tau_mJ_m) $ with $\tau_0\not=0$, we declare that $(x,v;u)\sim *$ if and only if   $J_0$ has a non-empty edge set (and there is no other identification given by $\sim$).  (Note that $\tilde U(D_u,E_u,N_u)$ and $U(D_u,E_u,N_u)$ are not determined by the triple for one specific $u$ but by the family $\{(D_u,E_u,N_u)\mid u\in |\GP_n|\}$.)
\item For $X=U(D_u,E_u,N_u)$, we define the {\em cardinality filtration} $\{F_l(X)\}_l$ on $X$ by 
$F_l=\{*\}\cup \{ (x,v;u)\mid  |u|\leq l\}$.

\item For $u\in |\GP_n|$, let $P_u$ be the underlying abstract partition of $u$ and $G_u\in\GG(P_u)$ the graph of $u$. We set $p(u)=|P^\circ|$ and 
\[
U_{\TT}=U(D_u, E_u, N_u) \ \text{where}\ D_u=D_{G_u},\ E_u=e_{P_u},\ \text{and}\  N_u=\nu_{P_u}
\]
\end{enumerate}
\end{defi}
Since $D_G$ is closed, the boundary $\partial \pi_P^{-1}(D_G)\cap \nu_P$ is not collapsed in $U_{\TT}$ (as long as the first graph of $u$ has the empty edge set), while $\partial \nu_P$ is collapsed. \\
\indent For a functor $X:C\to \CG_*$, let $\bar C_*(X): C\to \CH_{\kk}$ denote the functor obtained by taking $\bar C_*$ in the termwise manner. Recall the functor $\CECHF$ from Definition \ref{Dnormal}. 
\begin{prop}\label{LThomSinha} 
\begin{enumerate}
\item A geometric partition $A$ with $|A|=p$ gives an element of $sc(A)\in \Delta^p=\{0\leq t_1\leq\dots\leq t_p\leq 1\}$ through the map $\frac{1+sc}{2}:[-\infty,\infty]\to [0,1]$, by forgetting the multiplicity of geometric points. The map 
\[
\RR^{6n}\times |\GP_n|\ni \ (x,v;\tau_0(A,G_0)+\cdots +\tau_m(A,G_m))\ \mapsto \ (x,v;\tau_0G_0+\cdots +\tau_mG_m;sc(A))\ \in \bigsqcup_{P\in\PP_n}
\TT^c(P)\times \Delta^{|P^\circ|}
\] 
induces a well-defined homeomorphism $\phi_{\TT}:U_{\TT}\to |\CECHF(\TT^c)|$.
\item The Eilenberg-Zilber shuffle map $\bar C_*(\CECHF(\TT^c)([k]))\otimes C_*(\Delta^k)\to \bar C_*(\CECHF(\TT^c)([k])\wedge (\Delta^k_+))$ induces a chain map
\[
EZ:\tot N\bar C_*(\CECHF(\TT^c))\to \bar C_*(|\CECHF(\TT^c)|)\stackrel{\phi_{\TT}}{\cong} \bar C_*(U_{\TT}).
\]
When we filter the total complex $\tot N\bar C_*(\CECHF(\TT^c))$ by the simplicial degree and  $\bar C_*(U_{\TT})$ by the cardinality filtration $\bar C_*(F_l)$, the map $EZ$ preserves the filtration and induces an isomorphism of the spectral sequences from the $E_1$-page. In particular, the spectral sequence associated to $\{\bar C_*(U_{\TT}), \bar C_*(F_l)\}$ is isomorphic to the $n$-truncated Sinha sequence from $E_2$ under the shift of degree $(-p,6n-q)\leftrightarrow (p,q)$.
\end{enumerate}
\end{prop}
\begin{proof}
Part 1 is obvious routine work unwinding definitions.\\
\indent We shall prove part 2. The well-definedness of the map $EZ$ and the compatibility with the filtration are also routine. The functor $\CECHF(\TT^c)$ is extended to a simplicial space by the same formula as in Definition \ref{Dnormal} and the realization in the sense of Definition \ref{Dreal} and the standard realization of the simplicial space  are homeomorphic. Under a moderate condition which is satisfied by our simplicial space, the realization is (homotopy equivalent to) the homotopy colimit, which implies  the isomorphism from $E_1$. For the last claim, by Theorem \ref{TAtiyahdual}, we have a zigzag $\bar C_{6n-*}(\TT^c)\simeq C^*(\PK)$ of natural quasi-isomorphisms,  see  the proof of Proposition 3.8 of \cite{moriya3} for a detailed argument. With this fact, the claim  follows from  the compatibility of $\CECHF$ and the chain functor $\bar C_*$, Theorem \ref{TSinha} and Lemma \ref{Lcofinal}.
\end{proof}
Let $P_0\in \PP_n$ with $|P_0^\circ|=l_0$. The over-category $\PP_n/P_0$ is isomorphic to $\PP_{l_0}$. Let $U_\TT(P_0)\subset U_\TT$ be the subspace consisting of the basepoint and elements $(x,v;u)$ such that $P_0$ is equal to or a subdivision of the underlying abstract partition of $u$. The statement similar to Proposition \ref{LThomSinha} holds if we replace the triple $(\TT^c, U_{\TT}, \text{the $n$-truncated Sinha sequence})$ with the triple consisting of the restiction of $\TT^c$ to  $\PP_n/P_0$, the subspace $U_{\TT}(P_0)$, and the $l_0$-truncated Sinha sequence.

\section{The evaluation map}\label{SU_1}

In this and  following two sections, we connect the two spaces $\bar \Sigma$ and $U_{\TT}$ given in Definitions \ref{Dbar_Sigma}, \ref{Dintermediate_space} by a zigzag of maps and prove that the maps induce isomorphisms of spectral sequences.
\subsection{Outline of the construction}\label{SSintuition}
We want to connect the spaces $\bar \Sigma$ and $U_{\TT}$ by maps. The reader would find similarity between the definitions of $\bar \Sigma$ and $U_{\TT}$. Roughly speaking, $\bar \Sigma$ consists of pairs of a graph and polynomial map  which may have singularities.  The points in the domain $\RR$ which form a multiple point of the map correspond to  the vertices in a common connected component of the graph.  The space $U_{\TT}$ consists of pairs of a graph and point in (the inverse image of ) a regular neighborhood of the diagonal which equates the components of the product corresponding to the vertices in a common connected component of the graph.  Therefore, if we evaluate the polynomial map at the  vertices (points in $\RR$) of the graph, the sets of multiple points are mapped into the diagonals. This is the basic idea of construction of map $:\bar\Sigma\to U_\TT$ as mentioned in Introduction while we actually construct a zigzag for technical reasons. \\
\indent Our space of polynomial maps $\Gamma_n$ is $6n$-dimensional so we need to evaluate a polynomial (and its derivative) at {\em distinct} $n$ points so that the evaluation map induces a bijection $\Gamma_n\to \RR^{6n}$ and a well-defined map from the one-point compactification.   For a vertex of multiplicity $k$ (see Definition \ref{Dbar_Sigma}), we split  the vertex into $k$ distinct points by regarding the vertex as $k$ overlapping points and sliding the  points through different distances in its neighborhood. If we take the neighborhood sufficiently small, the image of evaluation map is included in a  regular neighborhood of the corresponding diagonal. In addition, if the points $\pm\infty$ have multiplicity $k$, we also slide  $k-1$ copies of $\pm\infty$  into $\RR$. Precisely speaking, we need to restrict the evaluation map to a compact subset of $\Gamma_n$ to determine a standard of ``sufficiently small neighborhood" in which we slide points, so we  take a sufficiently large disk in $\Gamma_n$ for each geometric partition. We need to take the sliding operation continuously on a partitional graph or its underlying  geometric partition. This needs some care. For example, consider a graph with two vertices of multiplicity $2$, $3$ respectively, and take the limit of these vertices getting closer. We  will first  define a width of slide depending only on each graph. 
The width of slide  should be getting smaller so that the slided points do not collide but if so, the width must tend to zero in the limit, which contradicts with continuity (see Figure \ref{Fslide}). To remedy this,  we connect the sliding operation of the graph with two vertices and that of the limit graph by homotopy. We will define a continuous map giving the sliding operation on the whole space of geometric partitions by induction on the number of distinct geometric points. \\
\indent In the above example, suppose the two vertices are connected by an edge. When they are sufficiently close, the sliding operation is approximated by that of the limit graph by the effect of the homotopy. Therefore,  the width of slide is not so small that the image by the evaluation map is  included in a designated regular neighborhood of the diagonal corresponding to the graph with two vertices. Nevertheless, it is included in the one corresponding to the limit graph (which has a loop). Because of this phenomenon, we define the target space of the evaluation map so that for a graph whose vertex set can be split into  subsets such that the vertices  in a common subset   are much closer   than the vertices in different subsets, the diagonal space associated to the graph is the one corresponding to the graph obtained by condensing the  vertices in each set into one vertex. In the example, we agree that the two vertices are  in a common subset if they are sufficiently close. (The subspace associated to the graph with the two close vertices is the subspace of elements whose velocity component corresponding to the limit vertex is zero, rather than a diagonal subspace). \\
\indent We will construct the following zigzag:

\[
\bar \Sigma\stackrel{\psi}{\longrightarrow} U_1 \stackrel{\phi_1}{\longleftarrow} U_2 \stackrel{\phi_2}{\longrightarrow}  U_3\stackrel{\phi_3}{\longleftarrow} U_4\stackrel{\phi_4}{\longrightarrow} U_5 \stackrel{\phi_5}{\longleftarrow}  U_6 \stackrel{\phi_6}{\longrightarrow}  U_7 \stackrel{\phi_7}{\longleftarrow}  U_{\TT}.
\]
The main technical difficulties in these three sections are construction of $\psi, U_1$ and $U_2$ and proof of properties of these spaces and map. These are given in sections \ref{SU_1} and  \ref{SU_2}. In the present section \ref{SU_1} we construct the evaluation map $\psi$. 
The function $\rho_?$ which will be introduced in Definition \ref{Dgen_partition} gives the width of slide mentioned above and $N(A)$ and a notion of frame in Definitions \ref{Dgen_partition} and  \ref{Dframe} give the standard for a graph having  subsets of vertices   such that  the vertices in a common  subset   are much closer to one another than the vertices in different subsets. The map $\bar\psi_0$ defined in the paragraphs after Lemma \ref{LN} gives the sliding operation.
The target of $\psi$, the space $U_1$ is defined by glueing Thom spaces of regular neighborhoods of diagonals as in the definition of  $U_\TT$, but the shape of  open sets the outside of which we collapse and the width of regular neighborhoods of the diagonals corresponding to graphs are different from those of $U_\TT$. 
 In  sections \ref{SU_2} and \ref{SU_3}, we make adjustments on these data.  In section \ref{SU_2}, we define the space $U_2$ and prove that the maps $\psi$ and  $\phi_1:U_2\to U_1$ induce isomorphisms between homology of filtered quotients. The spaces $U_3,\ U_4,\ U_5,\ U_6$ and $U_7$ (and maps between them) are defined in section \ref{SU_3} to make minor adjustments. The proofs in  section \ref{SU_3} are similar to those in section \ref{Stranslation} so most of them are omitted.
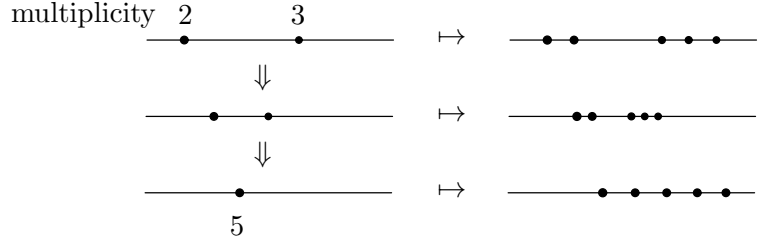
\begin{figure}
\begin{center}
{\unitlength 0.1in%
\begin{picture}(37.6500,11.0000)(0.3000,-11.3000)%
%
\special{pn 8}%
\special{pa 605 226}%
\special{pa 1895 226}%
\special{fp}%
%
\special{sh 1.000}%
\special{ia 800 226 20 20 0.0000000 6.2831853}%
\special{pn 8}%
\special{ar 800 226 20 20 0.0000000 6.2831853}%
%
\special{pn 8}%
\special{ar 1400 226 0 0 1.5707963 1.3654009}%
%
\special{sh 1.000}%
\special{ia 1400 226 17 17 0.0000000 6.2831853}%
\special{pn 8}%
\special{ar 1400 226 17 17 0.0000000 6.2831853}%
%
\special{pn 8}%
\special{pa 2505 226}%
\special{pa 3795 226}%
\special{fp}%
%
\special{sh 1.000}%
\special{ia 2700 226 21 21 0.0000000 6.2831853}%
\special{pn 8}%
\special{ar 2700 226 21 21 0.0000000 6.2831853}%
%
\special{pn 8}%
\special{ar 3300 226 0 0 1.5707963 1.3734008}%
%
\special{sh 1.000}%
\special{ia 3300 226 18 18 0.0000000 6.2831853}%
\special{pn 8}%
\special{ar 3300 226 18 18 0.0000000 6.2831853}%
%
\special{sh 1.000}%
\special{ia 2840 226 20 20 0.0000000 6.2831853}%
\special{pn 8}%
\special{ar 2840 226 20 20 0.0000000 6.2831853}%
%
\special{sh 1.000}%
\special{ia 3440 226 18 18 0.0000000 6.2831853}%
\special{pn 8}%
\special{ar 3440 226 18 18 0.0000000 6.2831853}%
%
\special{sh 1.000}%
\special{ia 3585 226 17 17 0.0000000 6.2831853}%
\special{pn 8}%
\special{ar 3585 226 17 17 0.0000000 6.2831853}%
%
\special{pn 8}%
\special{pa 600 626}%
\special{pa 1890 626}%
\special{fp}%
%
\special{sh 1.000}%
\special{ia 955 626 20 20 0.0000000 6.2831853}%
\special{pn 8}%
\special{ar 955 626 20 20 0.0000000 6.2831853}%
%
\special{sh 1.000}%
\special{ia 1240 626 17 17 0.0000000 6.2831853}%
\special{pn 8}%
\special{ar 1240 626 17 17 0.0000000 6.2831853}%
%
\special{pn 8}%
\special{pa 2500 626}%
\special{pa 3790 626}%
\special{fp}%
%
\special{sh 1.000}%
\special{ia 2855 626 21 21 0.0000000 6.2831853}%
\special{pn 8}%
\special{ar 2855 626 21 21 0.0000000 6.2831853}%
%
\special{sh 1.000}%
\special{ia 3140 626 18 18 0.0000000 6.2831853}%
\special{pn 8}%
\special{ar 3140 626 18 18 0.0000000 6.2831853}%
%
\special{sh 1.000}%
\special{ia 2935 626 20 20 0.0000000 6.2831853}%
\special{pn 8}%
\special{ar 2935 626 20 20 0.0000000 6.2831853}%
%
\special{sh 1.000}%
\special{ia 3210 626 17 17 0.0000000 6.2831853}%
\special{pn 8}%
\special{ar 3210 626 17 17 0.0000000 6.2831853}%
%
\special{sh 1.000}%
\special{ia 3280 626 17 17 0.0000000 6.2831853}%
\special{pn 8}%
\special{ar 3280 626 17 17 0.0000000 6.2831853}%
%
\special{pn 8}%
\special{pa 595 1026}%
\special{pa 1885 1026}%
\special{fp}%
%
\special{sh 1.000}%
\special{ia 1090 1026 20 20 0.0000000 6.2831853}%
\special{pn 8}%
\special{ar 1090 1026 20 20 0.0000000 6.2831853}%
%
\special{pn 8}%
\special{pa 2495 1026}%
\special{pa 3785 1026}%
\special{fp}%
%
\special{sh 1.000}%
\special{ia 2990 1026 21 21 0.0000000 6.2831853}%
\special{pn 8}%
\special{ar 2990 1026 21 21 0.0000000 6.2831853}%
%
\special{sh 1.000}%
\special{ia 3160 1026 20 20 0.0000000 6.2831853}%
\special{pn 8}%
\special{ar 3160 1026 20 20 0.0000000 6.2831853}%
%
\special{sh 1.000}%
\special{ia 3325 1026 20 20 0.0000000 6.2831853}%
\special{pn 8}%
\special{ar 3325 1026 20 20 0.0000000 6.2831853}%
%
\special{sh 1.000}%
\special{ia 3485 1026 20 20 0.0000000 6.2831853}%
\special{pn 8}%
\special{ar 3485 1026 20 20 0.0000000 6.2831853}%
%
\special{sh 1.000}%
\special{ia 3635 1026 20 20 0.0000000 6.2831853}%
\special{pn 8}%
\special{ar 3635 1026 20 20 0.0000000 6.2831853}%
\put(22.0000,-2.2000){\makebox(0,0){$\mapsto$}}%
\put(22.0000,-6.2000){\makebox(0,0){$\mapsto$}}%
\put(22.0000,-10.2000){\makebox(0,0){$\mapsto$}}%
\put(12.0000,-4.2000){\makebox(0,0){\rotatebox{-90}{$\Rightarrow$}}}%
\put(12.0000,-8.2600){\makebox(0,0){\rotatebox{-90}{$\Rightarrow$}}}%
\put(2.7500,-0.9500){\makebox(0,0){multiplicity}}%
\put(8.0500,-0.9500){\makebox(0,0){2}}%
\put(13.9500,-0.9500){\makebox(0,0){3}}%
\put(10.7500,-11.9500){\makebox(0,0){5}}%
\end{picture}}%
\end{center}
\caption{discontinuity of the sliding operation } \label{Fslide}
\end{figure}

\subsection{Definition of $U_1$ and $\psi:\bar \Sigma\to U_1$}\label{SSpsi}
Recall the spaces  $\tilde \Gamma_n$, $\chi(V,F)$, $AG_i$  and the map $I:\tilde \Gamma_n\to \tilde \Gamma_N$ from section \ref{Spreliminary}, and the distance $d_{sc}$ on $[-\infty,\infty]$ from subsection \ref{SSNT}.
\begin{defi}
\begin{enumerate}
\item For $0\leq l\leq n$, we define  subspaces $\PPP_{n,l}\subset \PPP_n$ and $\GP_{n,l}\subset \GP_n$ by
\[
 \PPP_{n,l}=\{ A\in \PPP_n \mid |A|\geq l\}, \qquad \GP_{n,l}=\{(A,G)\in \GP_n \mid  |A|\geq l\}.
\] 

\item Let $A=(t_1,\dots, t_n)$ be a geometric partition and $Q\in \PP_n$  a partition (which may not be the underlying partition of $A$). For a piece $\alpha \in Q$ , we set $t_\alpha=t_i$ with $i=\min\alpha$, where we agree that $t_0=-\infty$ and $t_{n+1}=\infty$.

\item For a geometric partition $A=(t_1,\dots,t_n)$ with underlying partition $P$, let $N_\infty(A)\subset \PPP_n$ be the subspace consisting of $A'=(t'_1,\dots,t'_n)$ satisfying the following condition: 
For  $\alpha\in P^\circ $ or $\alpha=\min P$ and for $j\in\alpha$, the  inequality
\[
 d_{sc}(t'_j, t_\alpha)\leq \frac{j-\min{\alpha}}{n+5}\min\{d_{sc}(t_\beta, t_\gamma)\mid \beta,\gamma \in P,  \beta\not=\gamma \}
\] 
holds, and for $\alpha=\max P$ and $j\in \alpha$,
\[
 d_{sc}(t'_j, \infty )\leq \frac{n+1-j}{n+5}\min\{d_{sc}(t_\beta, t_\gamma)\mid \beta,\gamma \in P, \beta\not=\gamma  \} 
\] holds.
The former inequality implies $t'_\alpha=t_\alpha$ for $\alpha\in P^\circ$ .

\item We fix a sequence $\{\WW_i\}_{i\geq 1}$ of compact neighborhoods  of $I(\tilde \Gamma_n)$ in $AG_{6n}$ such that  for each integer $i$, $\WW_{i}\supset \WW_{i+1}$, and $\WW_i$ satisfies the condition of Lemma \ref{Leval} for  $T=i$. For $A=(t_1,\dots, t_n)\in \PPP_n$, we set $\WW_A=\WW_{i_A}$ where  $i_A$ is the least positive integer  larger than \\ $\max\{ |t_i|\mid 1\leq i\leq n,\  t_i\not=\pm \infty\}$.
\item Let $V\subset \tilde \Gamma_N$ be an affine subspace and $J\in \GP_n$ a partitional graph. Let $F_J$ be the underlying classed configuration of $J$. We set $\chi(V,J):=\chi(V,F_J)$, see section \ref{Spreliminary}.
\item In the rest of the paper, we fix an embedding  $f_0\in \tilde\Gamma_N$.  $D_R(\tilde\Gamma_N)\subset \tilde \Gamma_N$ denotes the closed $6N$-dimensional disk centered at $f_0$ with radius  $R$ (see the paragraph before  Proposition \ref{Lisom_sigma} for the distance). For an affine subspace $V\subset \tilde \Gamma_N$,  we set $D_R(V)=V\cap D_R(\tilde\Gamma_N)$. For $A\in\PPP_n$, let $\langle A\rangle\subset \GP_n$ be the set of partitional graphs whose underlying geometric partition  coincides with $A$.  
\item Hereafter, we fix a continuous function
\[
R_?:\bigsqcup_{0\leq l\leq n}(\PPP_{n,l}-\PPP_{n,l+1})\to \RR_{>0},\qquad A\ \mapsto\  R_A
\]
stisfying the following conditions: 
\begin{enumerate}
 \item $D_{R_A/2}(\tilde \Gamma_N)\cap \chi(V,J)\not =\emptyset$ for any $A\in \PPP_n$, $J\in \langle A\rangle$ and $V\in \WW_{A}$, and
\item   $R_A\geq R_{A'}$ if $A\in N_\infty(A')$.
\end{enumerate}
Here, $\RR_{>0}$ is the set of positive real numbers (and $\PPP_{n,n+1}=\emptyset$).
\end{enumerate}
\end{defi}
For $A\in\PPP_n$ and an integer $l$, there is at most one $A'\in\PPP_n$ satisfying $A\in N_\infty (A')$ and $|A'|=l$. With this observation and Lemma \ref{Leval}, The function $R_?$ is easily constructed using a partition of unity on $\PPP_{n,l}-\PPP_{n,l+1}$. Note that it is impossible to extend $R_?$ to a continuous function on the entire space $\PPP_{n}$.
\begin{defi}
\begin{enumerate}
\item Put $v_0=\frac{1}{\sqrt[]{3}}(1,1,1)$ and let $\langle-,-\rangle$ be the standard inner product of $\RR^3$.  Let $T_A$ be the minimum of the  numbers $T$ satisfying the following condistions for any $f\in D_{R_A}(\tilde\Gamma_N)$:
\begin{enumerate}
\item $T\geq 10$, \vspace{2mm}
\item  $\langle f(-T/2), v_0\rangle\leq -T/2$ and $\langle f(T/2), v_0\rangle\geq T/2$,  \vspace{2mm}
\item  for $|t|\leq T/2$, \ \ $\langle f(-T/2),v_0\rangle \leq  \langle f(t),v_0\rangle \leq  \langle f(T/2),v_0\rangle$ holds, and    \vspace{2mm}
\item for $|t|\geq T/2$,\  \ $\langle f'(t),v_0\rangle\geq 10$ holds.
\end{enumerate}
Third condition implies that $f(t)$ with $|t|\leq T/2$ belongs to the domain bounded by the two planes which is perpendicular to $v_0$ and includes $f(-T/2)$ or $f(T/2)$. 
This defines a continuous function 
\[
T_?:\bigsqcup_{0\leq l\leq n}(\PPP_{n,l}-\PPP_{n,l+1})\to \RR_{>0},\qquad A\ \mapsto T_A\ .
\]
\end{enumerate}
\end{defi}
Clearly, we have $T_A\geq T_{A'}$ if $A\in N_\infty(A')$. 
\begin{defi}\label{Dgen_partition}
\begin{enumerate}

\item For $t\in \RR$, we set $|t|'=\min\{1, |t|\}$.

\item Let $V\subset \tilde \Gamma_N$ be an affine subspace and  $\delta$ a positive number,  and  $J=(A,G)\in \GP_n$ a partitional graph with $A=(t_1,\dots,t_n)$. Let $\NN_\delta(V,J)\subset V$ be the subset of elements $f$ which satisfy the following conditions:
\begin{enumerate}
\item If $(\alpha,\beta)\in G$ and $\alpha<\beta$,  then $|f(t_\alpha)-f(t_\beta)|\leq \delta |t_\alpha-t_\beta|'$, and
\item if $(\alpha,\alpha)\in G$,  then $|f'(t_\alpha)|\leq \delta$.
\end{enumerate} 


\item Hereafter, we fix  a continuous function 
\[
\delta_?:\bigsqcup_{0\leq l\leq n}(\PPP_{n,l}-\PPP_{n,l+1})\to \RR_{>0}
\] 
satisfying the following conditions:
\begin{enumerate}
\item  $\delta_A<1$ for any $A\in\PPP_n$, 
\item $\delta_A\leq \delta_{A'}$ if $A\in N_\infty(A')$, and 
\item for any $A\in \PPP_n$, $J\in \langle A\rangle $ and $V\in \WW_{A}$, the distance from any point of $\NN_{\delta_A}(V,J)\cap D_{R_A}(V)$ to the subspace $\chi(V,J)$ is $<R_A/10$.
\end{enumerate}

\item 

For $A\in \PPP_n$, we set 
\[
V_A=\max\{\, |f'(t)|\, \mid -T_A\leq t\leq T_A,\ \  f\in D_{R_A}(\tilde \Gamma_N)\}.
\]
Hereafter, we fix  a continuous function on
\[
\rho_?:\bigsqcup_{0\leq l\leq n}(\PPP_{n,l}-\PPP_{n,l+1})\to \RR_{>0}
\] 
satisfying the following conditions for any $A=(t_1,\dots, t_n) \in \PPP_n$:

\begin{enumerate}
\item  The following inequalities hold :
\[
0<\rho_A<\frac{\delta_A \cdot \min\{\,1,\, d_{sc}(t_\alpha,t_\beta) \mid \alpha,\beta\in P,\alpha\not=\beta\}}{8^{(n+5)(p+10)}V_A},\qquad \rho_A<1-sc(2T_A).
\] 
\item If  $f$ is an element of $D_{R_A}(\tilde\Gamma_N)$ and if two numbers $t, t'\in \RR$ satisfy $|t|\leq  T_A$ and $|t'-t|\leq \rho_A$,
the following  inequality holds :
\[
|f'(t')-f'(t)|\leq \frac{\delta_A}{ 8^{(n+5)(p+10)}}.
\]

\end{enumerate} 

\item Let $A\in \PPP_{n,l}-\PPP_{n,l+1}$ and $P$ the underlying partition of $A$. Let 
\[
N(A)\subset \PPP_{n,l}
\] be the subset consisting  of $A'=(t'_1,\dots,t'_n)$ satisfying the following conditions:\vspace{2mm}
\begin{enumerate}
\item If $\alpha\in P^\circ$ and $j\in\alpha$, \ \ $\displaystyle 
t'_{j}-t_\alpha\leq \frac{j-\min{\alpha}}{n+5}\rho_A
$.\\
\item If 
 $\alpha=\max P$ and $j\in \alpha$, \ \ $\displaystyle 1-sc(t'_j)\leq \frac{n+1-j}{n+5}\rho_A$. \\
\item If $\alpha=\min P$ and $j\in \alpha$, \ \ $\displaystyle sc(t'_j)-(-1)\leq \frac{j}{n+5}\rho_A$. 
\end{enumerate} 

\item Let $u=\tau_0J_0+\cdots +\tau_mJ_m\in |\GP_{n}|$ be an element with $\tau_m\not=0$. Put $l=|u|$.   We let $N(u)$ denote the subspace of $ |\GP_{n,l}|$ consisting of elements $u'$ which have an expression 
$u'=\tau'_0J'_0+\cdots +\tau'_{m'}J'_{m'}$  satisfying the following conditions: 
\begin{enumerate}
\item $m=m'$ and $\tau_i=\tau'_i$ for $0\leq i\leq m$.
\item  $A'\in N(A)$, where $A$ and $A'$ are the underlying geometric partitions of $J_i$ and $J'_i$. 
\item $\delta_{PQ}(G'_i)=G_i$, where $P$ and $Q$ are the underlying partitions of $u$ and  $u'$, and $G_i$ and $G_i'$ are the underlying abstract graphs of $J_i$ and $J'_i$ respectively.
\end{enumerate}
 Here, by the condition (b), $Q$  is equal to $P$, or a subdivision of $P$, so the map $\delta_{PQ}$ is well-defined. This definition of $N(u)$ does not depend on the expression of $u$.
\item For an element $u\in |\GP_n|$ with underlying geometric partition $A$, we set
\[
R_u =R_A,\quad T_u=T_A,\quad \delta_u=\delta_A,\quad \text{and} \quad \rho_u=\rho_A.
\]
\end{enumerate}
\end{defi}
The functions $\delta_A$ and $\rho_A$ are  constructed similarly to $R_A$. 
Intuitively speaking, the distance between two geometric points of an element of $N(A)$ which belong to a common piece of $A$ is much smaller than those which belong to different pieces. The number  $\rho_A$ is much smaller than $\rho_{A'}$ if $A\in N(A')$ and $A\not= A'$. Parts 1 and 2 of the following lemma follow from these observations and part 3 follows from the continuity of $\rho_?$. We will sometimes use this lemma  implicitly. 
\begin{lem}\label{LN}
\begin{enumerate}
\item Let $A', A''\in \PPP_n$. Suppose $N(A')\cap N(A'')\not=\emptyset$ and $|A'|\leq |A''|$. Then $A''\in N(A')$.
\item Let $A_1,A_2, A_3\in \PPP_n$ with $|A_1|\leq |A_2|\leq |A_3|$. Then, $A_2\in N(A_1)$ if and only if $A_3\in N(A_1)$.
\item Let $\partial N(A)$ be the boundary of $N(A)$ as a subspace of $N_\infty(A)$ so  for an element of $\partial N(A)$, at least one of the inequalities in  (5) of Definition \ref{Dgen_partition} is an equality. Then, we have
\[
\bigcup_{A}\partial N(A)=\partial \left(\bigcup_{A} N(A)\right)
\]
Here, $A$ runs through $\PPP_{n,l}-\PPP_{n,l+1}$ and the right-hand side is the boundary as the subspace of $\PPP_{n,l}$.\hfill\qedsymbol
\end{enumerate}
\end{lem}

We shall define a continuous map
\[
\bar \psi_l:\PPP_{n,l}\to \RR^n
\]
inductively on $l$ in the decending order. Set
\[
\bar \psi_{n}(A)=A.
\]
This is well-defined since  all of the  geometric points of $A\in \PPP_{n,n}$ belong to $\RR$. \\
\indent  We set
\[
N_{l}=\cup_{A\in \PPP_{n,l}-\PPP_{n,l+1}}N(A)\subset \PPP_{n,l}.
\] 
Let $\partial N_l$ denote the boundary of $N_l$ in $\PPP_{n,l}$. We fix a continuous function $g:N_{l}\to [0,1]$ such that there are a neighborhood $N'$ of $\PPP_{n,l}-\PPP_{n,l+1}$ and a neighborhood $N''$ of $\partial N_{l}$ (both taken in $N_l$) such that $g(N')=1$ and $g(N'')=0$. This function exists by continuity of $\rho_?$. Define a map $\psi': N_{l}\to \RR^n$ as follows : For $A\in N_{l}$ there is a unique $A'\in \PPP_{n,l}-\PPP_{n,l+1}$ such that $A\in N(A')$ by Lemma \ref{LN}. Write $A'=(t'_1,\dots, t'_n)$ and let $P$ be the underlying partition of $A'$. The point $\psi'(A)=(s_i)_{1\leq i\leq n}$ is defined by

\[
s_{i}=\left\{
\begin{array}{ll}
t'_{\alpha}+\frac{(i-\min\alpha)}{n}\rho_{A'} & \text{if $i\in \alpha $ for some piece  $\alpha\in P^\circ$}, \\
&\\
sc^{-1}\left(1-\frac{n+1-i}{n}\rho_{A'}\right)  & \text{if $i$ belongs to  $\max P$,} \\
& \\
sc^{-1}\left(-1+\frac{i}{n}\rho_{A'}\right) & \text{if $i$ belongs to $\min P$.} 
\end{array}\right.\vspace{2mm}
\] 
Suppose $\bar\psi_{l+1}$ has been defined. We define $\bar \psi_l$ by
\[
\bar \psi_{l}=
\left\{
\begin{array}{cl} 
g\cdot \psi'+(1-g)\cdot\bar \psi_{l+1} & \text{ on } N_{l}, \\
\\
\bar \psi_{l+1} & \text{ on } \PPP_{n,l}-N_{l}.
\end{array}
\right.
\]
By the continuity of $\rho_?$, the subset $N_{l}$ is a closed neighborhood of $\PPP_{n,l}-\PPP_{n,l+1}$ in $\PPP_{n,l}$, so $\bar \psi_l$ is continuous. We have obtained $\bar \psi_0 :\PPP_n\to \RR^n$ by induction. We  write
 \[
\bar\psi_0(A)=(\bar\psi_0(A,1),\dots, \bar \psi_0(A,n)).
\] 
\begin{defi}\label{Dframe}
\begin{enumerate}
\item We say  $A'\in \PPP_n$ is the {\em frame} of $A$ if $A\in N(A')$ and the cardinality of $A'$ is the minimum among $A_1$ satisfying $A\in N(A_1)$. We say $u'\in|\GP_n|$ is the {\em frame} of $u$ if $u\in N(u')$ and the cardinality of $u'$ is the minimum among $u_1$ satisfying $u\in N(u_1)$.
\item Let $u\in |\GP_n|$. Let $A$ be the geometric partition of $u$ and $A'$ the frame of $A$. We denote by $P$ and $Q$  the abstract partitions of $A'$ and $A$, respectively, by $G$ the abstract graph of $u$. We say $u$ is {\em degenerate at $\infty$} if $\delta_{PQ}(G)$ has an edge incident with $\max P$, and $u$ is {\em degenerate at $-\infty$} if $\delta_{PQ}(G)$ has an edge incident with $\min P$. We say {\em $u$ is degenerate at $\pm\infty$} if $u$ is degenerate at $\infty$ {\em or} $-\infty$.
\end{enumerate}
\end{defi}
The frame is actually unique by Lemma \ref{LN}. The vertices $\pm\infty$ are discrete in a partitional graph by definition, so
the geometric partition of the frame of $u$ and the frame of the geometric partition of $u$ are the same if and only if $u$ is not degenerate at $\pm \infty$. 
The following lemma is a list of basic properties of $\bar \psi_0$. We will use implicitly this lemma many times.
\begin{lem}\label{Lproperty_psi2}
\begin{enumerate}
\item  Set 
\[
Q_{n,l}=\PPP_{n,l}-\PPP_{n,l+1}-\bigcup_{A\in \PP'_n-\PPP_{n,l}}N(A)\ .
\]
There is a positive number $d$ such that for any $A\in Q_{n,l}$ the distance of any two distinct geometric points of $A$ ( including $\pm\infty$, measured by $d_{sc}$) is  $\geq d$.
\item For any $A\in \PPP_n$ and  number $1\leq i\leq n-1$, $\bar \psi_0(A,i)<\bar\psi_0(A,i+1)$.
\item Let $A=(t_1,\dots, t_n)\in \PPP_n$. Let $P$ the underlying partition of the frame of $A$ and  $\alpha$  a piece of $P^\circ$. Set $i=\min\alpha$. Then we have $\bar\psi_0(A,i)=t_i\ (=t_\alpha)$.
\item Let $A=(t_1,\dots,t_n)\in \PP'_n$ with underlying partition $Q$, and $\alpha$ a piece of $Q^\circ$. Put $i=\min\alpha,\ j=\max\alpha$. Let $i',j'$ be two numbers with $i\leq i'<j'\leq j$. We have
$t_{j'}-t_{i'}\leq \bar\psi_0(A,j)-\bar\psi_0(A, i)$. 
\item For $A, A'\in \PPP_n$, suppose $A'$ is the frame of $A$. Let $P$ be the underlying partition of $A'$. For $\alpha\in P^\circ$, set $i=\min\alpha$ and $j=\max\alpha$. We have $\bar\psi_0(A,j)-\bar\psi_0(A,i)\leq \rho_{A'}$.
\end{enumerate}
\end{lem}
\begin{proof}
We shall prove part 1. Since the closure of $\PPP_{n,l}-\PPP_{n,l+1}$ in $\PP'_n$ is compact, suppose there is a sequence$\{A_p\}_{p\geq 0}$ in $Q_{n,l}$ converging to an element $A_\infty$ with $|A_\infty|\leq l-1$. Let $P$ be the underlying partition of $A_\infty$. Let $A'_p$ be the geometric partition such that its underlying partition  is $P$  and for each $\alpha\in P^\circ$, the $i$-th geometric points of $A_p$ and $A'_p$ are the same for $i=\min\alpha$. For a sufficiently large $p$, clearly we have $A_p\in N(A'_p)$ since $\rho_?$ is continuous. This is a contradiction.  Part 2 follows from part 1. Parts 3, 4 and 5 are obvious from the construction.
\end{proof}
\begin{defi}\label{Dpsi_0}
We define a map 
\[
\psi_0:|\GP_n|\to \RR^n\quad\text{ by }\quad \psi_0(u)=\bar \psi_0(A_u)
\] where $A_u$ is the underlying geometric partition of $u$. Similarly to $\bar \psi_0$, we write $\psi_0(u)=(\psi_0(u,1),\dots \psi_0(u,n))$ and set $\psi_0(u,\alpha)=\psi_0(u,\min\alpha)$ for a piece $\alpha$ of a partition with $\min\alpha\not=0,n+1$. 
\end{defi}

{\bf Convention} : Since $\PP_{n}'$ is homeomorphic to $\Delta^n$, there exists an  integer $i_0$ larger than \\
$\max\{|\bar \psi_0(A,i)|\mid i=1,n,\  A\in \PP_{n}'\}$. Hereafter, we fix such an $i_0$ and  assume that $\Gamma_n$ is an element of $\WW_{i_0}$ (as well as it satisfies the condition of Lemma \ref{Lgeneral_posi}). If $A\in \PPP_n$ is a frame, we have $\Gamma_n\in \WW_A$.

\begin{defi}\label{De_u} Let $u\in |\GP_n|$.
\begin{enumerate}
\item We set
\[
\chi(u)=\chi(\Gamma_n,F_u),\qquad \NN_\delta(u)=\NN_\delta(\Gamma_n,\LL(u)),
\] 
where $F_u$ is the underlying classed configuration of the last $\GP_n$-element $\LL(u)$ of $u$. Let $G$ be the underlying abstract  graph of $u$. We mean  $(\alpha,\beta)\in E(G)$ by $(\alpha,\beta)\in u$. 
\item Let $u'\in |\GP_n|$ be an element with  underlying partition $P$. Put $p=|u'|$. We define a map $\tilde e_{u,u'}:\RR^{6p}\to \RR^{6n}$ as follows: Let $(x_\gamma, v_\gamma)_{\gamma\in P^\circ}\in \RR^{6p}$. For $\alpha\in P^\circ$ and $j\in \alpha$, we set
\[
y_j=x_\alpha+(\psi_0(u,j)-\psi_0(u,\alpha))v_\alpha,\qquad w_j=v_\alpha.
\]

If $\alpha\in P-P^\circ$ and $j\in \alpha$, we set
\[
y_j=(2\psi_0(u,j),0,0),\qquad w_j=(2,0,0).
\] 
Then, we set 
\[
\tilde e_{u,u'}(x_\alpha,v_\alpha)=(y_j,w_j)_{1\leq j\leq n}.
\]
The map $\tilde e_{u,u'}$ only depends on the geometric partitions of $u, u'$ but we use this notation to ease  notations. \\
\indent Let  $\delta>0$. We denote by $\Delta(u,u',\delta)$  the subspace of $\RR^{6p}$ consisting of $(x_\gamma,v_\gamma)_{\gamma}$ satisfying the following conditions:
\begin{enumerate}
\item $\displaystyle |x_\alpha-x_\beta|\leq \frac{\delta}{8^{(n+5)(p+5)}}|\psi_0(u,\alpha)-\psi_0(u,\beta)|'$\  if $ (\alpha,\beta)\in u' $ and $\alpha\not=\beta$ (see Definition \ref{Dgen_partition}).\\
\item $\displaystyle |v_\alpha|\leq \frac{\delta}{8^{(n+5)(p+5)}}$ \ if $(\alpha,\alpha)\in u'$.
\end{enumerate}
\item We set 
\[
e_u=\tilde e_{u,u'},
\] where $u'$ is any element whose geometric partition equals the frame of the geometric partition of $u$. Let $\pi_u:\RR^{6n}\to e_u(\RR^{6p})=\RR^{6p}$ be the orthogonal projection.
\item  We set 
\[
 \Delta_u=
\left\{
\begin{array}{ll}
\emptyset & \ \text{if $u$ is  degenerate at $\pm \infty$}, \\
\Delta(u,u',\delta_{u'}) & \text{otherwise},
\end{array}
\right. 
\]
where $u'$ is the frame of $u$. 
\end{enumerate}
\end{defi}

\begin{defi}\label{DU_1}
\begin{enumerate}
\item We define a map 
\[
\psi_* :\Gamma_n\times |\GP_n|\to \RR^{6n}
\]
by 
\[
\begin{split}
\psi_*(f,u)& =M_0^{\times 2n} (f^{\times n}(\psi_0(u)), (f')^{\times n}(\psi_0(u)))\\
& \\
& \bigl(\ \  =\bigl(\,M_0\cdot f(\psi_0(u,i)),\,  M_0\cdot f'(\psi_0(u,i))\,\bigr)_{i}\ \ \bigr).
\end{split}
\]
Here, $M_0\in SO(3)$ is a fixed matrix such that $M_0\cdot \frac{1}{\ \sqrt[]{3}}(1,1,1)=(1,0,0)$. We define a map $\psi_u:\Gamma_n\to \RR^{6n}$ by $\psi_u(f)=\psi_*(f,u)$.
\item We set
\[
U_1= U(\Delta_u,e_u,N_u)\quad\text{with}\quad N_u=Int(\psi_u(D_{R_{u'}}(\Gamma_n))).
\]
Here, $u'$ is the frame of $u$ and $Int$ denotes the interior (see Definition \ref{Dintermediate_space}).

\item We define a based map $ \psi : \bar \Sigma\to U_1$ as follows.  Let $(f,u)\in \bar \Sigma-\{*\}$. If $u$ is  degenerate at $\pm\infty$, we set $\psi(f,u)=*$\ . Otherwise, we set $\psi(f,u)= (\psi_*(f,u), u)$, which is regarded as $*$ if $\psi_*(f,u)\not\in Int(\psi_u (D_{R_{u'}}(\Gamma_n)))$. 
\end{enumerate}
\end{defi}
 By the choice of $\Gamma_n$ (see the convention below Definition \ref{Dpsi_0}), $\psi_u$ is a bijective. Since $\Delta_u$ is closed, the boundary $\partial \pi_u^{-1}(\Delta_u)\cap   Int(\psi_u(D_{R_{u'}}(\Gamma_n)))$ is not collapsed in $U_1$ (as long as the first graph of $u$ has the empty edge set), while $\partial  \psi_u(D_{R_{u'}}(\Gamma_n))$ is collapsed. \\ If $u\in |\GP_n|$ tends to an element degenerate at $\pm\infty$, a point $(x,v;u)$ in $U_1$ tends to $*$ by the definition of $\Delta_u$. 
\subsection{Basic properties of  $U_1$ and $\psi$}

\begin{lem}\label{Lkeymap} We use the notations in Definitions \ref{De_u} and \ref{DU_1}.
\begin{enumerate}
\item The map  $ \psi : \bar \Sigma\to U_1$ is well-defined, i.e. for a pair $(f,u)$ with $f\in \chi(u)$ and $u$ being not degenerate at $\pm \infty$,  we have $\psi_*(f,u)\in \pi_u^{-1}(\Delta_u)$.
 \item Let $u\in |\GP_n|$ be an element which is not degenerate at $\pm \infty$, $u'$ the frame of  $u$.  The following inclusions hold:
\[
\chi(u')\cap D_{R_{u'}}(\Gamma_n)\ \subset\  \psi_{u}^{-1}(\pi_{u}^{-1}(\Delta_{u}))\cap D_{R_{u'}}(\Gamma_n)\ \subset\  \NN_{\delta_{u'}}(u')
\]
\item The map  $ \psi : \bar \Sigma\to U_1$ is continuous.
\end{enumerate}
\end{lem}
\begin{proof}
We shall prove (1). Let $u'$ be the frame of $u$. We denote by $P$ and $Q$ the partitions of $u'$ and $u$, respectively. 

Let $f\in D_{R_{u'}}(\Gamma_n)$ and  $\alpha\in P^\circ$.  For $1\leq j\leq n$, we put $\psi_0(u,j)=s_j$. By Lemma \ref{Lproperty_psi2}, we have $s_j-s_\alpha \leq \rho_{u'}$ where  $j\in \alpha$, and by our convention, $s_\alpha=s_{\min\alpha}$. By the definition of $\rho_{u'}$, we have 
\begin{equation}\label{Eqeval}
|e_{u,\alpha}(f(s_\alpha),f'(s_\alpha))-(f(s_j), f'(s_j))_{j\in \alpha}|\ \leq\ \bigl(\sum_{j\in \alpha}(s_j-s_\alpha)^2+n\bigr)^{1/2} \  \delta'\leq 2n \delta',\quad \text{where}\quad \delta'=\frac{\delta_{u'}}{ 8^{(n+5)(p+10)}}
\end{equation}
where $p=|P^\circ|$ and $e_{u,\alpha}(-)$ is the $|\alpha|$-tuple of $j$-th components of $e_u(-)$ for $j\in\alpha$. Note that $e_{u,\alpha}(-)$ depends only on the $\alpha$-th component of an element of $\RR^{6p}$ and we regard it as a map $\RR^6\to \RR^{6|\alpha|}$.

We set $(\bar x_\gamma,\bar v_\gamma)_{\gamma\in P^\circ}=\pi_u((f(s_j),f'(s_j))_{1\leq j\leq n})$.  
The inequality (\ref{Eqeval}) implies $|f'(s_j)-\bar v_\alpha|\leq 2n\delta'$ for $j\in\alpha $. We have
\begin{equation}
\begin{split}
|f(s_j)-(f(s_\alpha)+(s_j-s_\alpha)\bar v_\alpha)|&\leq |f(s_j)-(f(s_\alpha)+(s_j-s_\alpha)f'(s_\alpha))|+|(s_j-s_\alpha)(f'(s_\alpha)-\bar v_\alpha)| \\
&\leq  \delta'(s_j-s_\alpha)+2n\delta'(s_j-s_\alpha)\\
&= (2n+1)\delta'(s_j-s_\alpha). \label{EQdiff_piece}
\end{split}
\end{equation}
Therefore,  we have 

\[
|(f(s_j),f'(s_j))_{j\in\alpha}-e_{u,\alpha}(\bar x_\alpha,\bar v_\alpha)|  \leq |(f(s_j),f'(s_j))_{j\in\alpha}-e_{u,\alpha}(f(s_\alpha),\bar v_\alpha)|\leq n(2n+1)\delta' (s_k-s_\alpha),
\]
where $k=\max\alpha$. In particular, we have 
\begin{equation}
 |\bar x_\alpha-f(s_\alpha)|  \leq (2n^2+n)\delta'(s_k-s_\alpha). \label{EQdiff_piece2}
\end{equation}
Let $(t_1,\dots, t_n)$ be the geometric partition of $u$.\\
\indent (i)\quad 
Let $(\alpha,\beta)\in u'$\ ($\alpha\not=\beta$). Since $u\in N(u')$, there are two pieces  $\alpha_1, \beta_1$ of $Q$ such that \\ 
$|t_{\alpha_1}-s_\alpha|, |t_{\beta_1}-s_\beta|\leq \rho_{u'}$, and $(\alpha_1,\beta_1)\in u$. For $f\in \chi(u)\cap D_{R_{u'}}(\Gamma_n)$, by definition of $\rho_{u'}$, we have 
\[
|f(s_\alpha)-f(s_\beta)| < |f(s_\alpha)-f(t_{\alpha_1})|+|f(t_{\beta_1})-f(s_\beta)|\leq 2\delta'|s_\alpha-s_\beta|'.
\]  
Hence  $|\bar x_\alpha-\bar x_\beta|\leq (4n^2+2n+2)\delta'|s_\alpha-s_\beta|'$ by the inequality (\ref{EQdiff_piece2}). \\
\indent (ii)\quad If $(\alpha,\alpha)\in u'$ and there is a piece $\alpha_1$ of $Q$ such that $\alpha_1\subset \alpha$ and  $(\alpha_1,\alpha_1)\in u$,  for $f\in \chi(u)\cap D_{R_{u'}}(\Gamma_n)$, we have $|f'(s_\alpha)|\leq \delta'$ by definition of $\rho_{u'}$. By the inequality (\ref{Eqeval}), we have $|\bar v_\alpha-f'(s_\alpha)|\leq 2n\delta'$ so 
\[
|\bar v_\alpha|\leq (2n+1)\delta'.
\]
\indent (iii)\quad If $(\alpha,\alpha)\in u'$ and there is an edge $(\alpha_1,\alpha_2)\in u$ with $\alpha_1\not=\alpha_2$ and $\alpha_1,\alpha_2\subset \alpha$,  for $f\in \chi(u)\cap D_{R_{u'}}(\Gamma_n)$, by the inequality (\ref{EQdiff_piece}), 
\[
\begin{split}
|(s_k-s_\alpha)\bar v_\alpha| & \leq |f(t_{\alpha_1})-f(t_{\alpha_2})|+|f(s_\alpha)-f(t_{\alpha_1})|+|f(t_{\alpha_2})-f(s_k)|+|f(s_\alpha)+(s_k-s_\alpha)\bar v_\alpha-f(s_k)| \\
&\leq \delta'|s_\alpha-t_{\alpha_1}|+\delta'|t_{\alpha_2}-s_k|+(2n+1)\delta'|s_k-s_\alpha|\leq (2n+4)\delta'|s_k-s_\alpha|.\\
\therefore\ |\bar v_\alpha|& \leq (2n+4)\delta'.
\end{split}
\]
Together with $(4n^2+2n+2)\delta',\ (2n+4)\delta'\leq \frac{\delta_{u'}}{8^{(n+5)(p+5)}}$, the inequalities obtained in (i),(ii) and (iii) imply $(\bar x_\gamma,\bar v_\gamma)\in \Delta_u$.\\
\indent The left inclusion of part 2 follows from an argument completely similarly to the proof of the previous part.\\
\indent 
Let us prove the right inclusion of part 2. Suppose $f\in D_{R_{u'}}(\Gamma_n)$ and $\psi_{u}(f)\in \pi^{-1}_{u}(\Delta_{u})$. Let $(\alpha, \beta)\in u'$ with $\alpha\not=\beta$ and put $k_1=\max\alpha$ and $k_2=\max\beta$. By the inequality in the proof of part 1,  we have
\[
\begin{split}
|f(s_\alpha)-f(s_\beta)|  &  \leq |f(s_\alpha)-\bar x_\alpha|+|\bar x_\alpha-\bar x_\beta|+ |\bar x_\beta-f(s_\beta)|\\
&  \leq (2n^2+n)\delta'(s_{k_1}-s_\alpha +s_{k_2}-s_\beta)+\frac{\delta_{u'}}{8^{(n+5)(p+5)}}|s_\alpha-s_\beta|'\leq \delta_{u'}|s_\alpha-s_\beta|'.
\end{split}
\]
By Lemma \ref{Lproperty_psi2}, $s_\alpha$ and $s_\beta$ are equal to the corresponding geometric points of $u'$.\\
Let $(\alpha,\alpha)\in u'$. We have
\[
\begin{split}
|f'(s_\alpha)|\leq |\bar v_\alpha|+|\bar v_\alpha-f'(s_\alpha)|&\leq \delta_{u'}.
\end{split}
\]
Thus,  we have $f\in \NN_{\delta_{u'}}(u')$.\\
\indent For part 3, consider a sequence $(u_k,f_k)$ in $\bar \Sigma$ with  limit $(u_\infty, f_\infty)$ and let $u_\infty'$ be the frame of $u_\infty$. The case which needs care is that $u_k$ is not degenerate at $\pm\infty$ for $k<\infty$ but $u_\infty$ is. Say it is degenerate at $\infty$. For any sufficiently large $k$, there is a  non-discrete geometric  point $t^k_i$ of $u_k$ such that $t^k_i>T_{u'_\infty}$ by definition of $\rho_?$. Therefore, by definition of $T_{u'_\infty}$, we have $\chi(u_k)\cap D_{R_{u'_\infty}}(\tilde \Gamma_N)=\emptyset$, which implies $\lim_{k\to \infty}\psi(u_k,f_k)=*$.
\end{proof}
\begin{lem}\label{Lneat}
Let $V\subset \Gamma_n$ be an affine subspace. Fix arbitrary element $f_1\in V$. Suppose $V$ is the set of $f\in \Gamma_n$ satisfying  the following simultaneous equation 
\[
\{q_{1j}a_1+\cdots q_{Mj}a_M=0\}_j.
\]
Here, $M=6n$, and   $q_{1j},\dots, q_{Mj}\in \RR^3$ are constants, and variables $a_1,\dots, a_M$ are given by $f-f_1=a_1v_1+\cdots +a_Mv_M$, where $v_1,\dots, v_M$ are a fixed basis of the vector space made  by performing the parallel transport  $-f_1$ to $V$. Let $N(V)\subset \Gamma_n$ be the subset given by the following simultaneous inequalities
\[
\{|q_{1j}a_1+\cdots q_{Mj}a_M|<\delta'_j\}_j 
\]
for some constants $\delta'_j>0$, and $u\in |\GP_n|$ the frame of some element of $|\GP_n|$. Put $D=D_{R_u}(\Gamma_n)$. Suppose $\dim V=\dim \chi(u)$ and  $\chi(u)\cap D\subset N(V)\cap D\subset \NN_{\delta_u}(u)$. Then,
$(V,N(V),D)$ is a $(6n-3c(u))$-dimensional spherical triple. 
Here $c(u)$ is the complexity, see Definition \ref{Dbar_Sigma}. Moreover, the inclusion $(\chi(u)\cap D, \chi(u)\cap \partial D)\to (N(V)\cap D, N(V)\cap \partial D)$ is a relative homotopy equivalence.
\end{lem}
\begin{proof}
Let $p_V:\Gamma_n\to V$ be the orthogonal projection.
Since $u$ is a frame, by definitions of $R_?$, $\delta_?$ and $T_0$ (see the paragraph under Definition \ref{Dpsi_0}), the tangent of the angle between $V$ and $\chi(u)$ is smaller than $\frac{1/10}{1/\sqrt[]{3}-1/10}$. By this observation, we can take $g_0\in V\cap D$ such that $g_0\not\in p_V(N(V)\cap \partial D)$. We fix  such a $g_0$.  We construct a retracting homotopy $H_t:\partial D\cap N(V) \to \partial D\cap V$ as follows.
For $f\in \partial D\cap N(V)$, we set $f^t=tp_V(f)+(1-t)f$ and $H_t(f)=f^{st}:=(1-s)g_0+sf^t$ for a unique positive number $s$ such that $f^{st}\in \partial D$. There is a unique number $k > 0$ such that $f^{kt}(:=(1-k)g_0+kf^t)=f+l(p_V(f)-g_0)$ for some $l\geq 0$. The coefficients of the linear combination of $v_1,\dots,v_M$ expressing the vector $p_V(f)-g_0$, $a_1,\dots, a_M$ satisfy the  equations
$
q_{1j}a_1+\cdots +q_{Mj}a_M=0,
$
so for the coefficients of $f^{kt}-f_1$ and $f-f_1$ the values of $|q_{1j}a_1+\cdots +g_{Mj}a_M|$ are the same and  the point $f^{kt}$ belongs to $N(V)$. By definition of $g_0$, we have $f^{kt}\not\in Int(D)$ so $k\geq s$.  Since $f^{st}$ is on the line segment between $g_0$ and $f^{kt}$, we have $f^{st}\in N(V)$ and the retracting homotopy $H_t$ is well defined. This is easily extended to a retractiong homotopy on $N(V)\cap D$. By the mentioned observation on the angle between $\chi(J)$ and $V$, the restriction $p_V|_{\chi(u)}:\chi(u)\to V$ is bijective, which implies  latter part of the claim. 
\end{proof}
\begin{lem}\label{LDelta0}
We use the notation of  Lemma \ref{Lkeymap} (2) and put $D=D_{R_{u'}}(\Gamma_n)$. Let $\Delta_u^0\subset \RR^{6|u'|}$ be the subspace consisting of $(x_\gamma,v_\gamma)$ satisfying the following conditions :
\begin{enumerate}
\item $x_\alpha=x_\beta$ if $(\alpha,\beta)\in u'$ and $\alpha\not=\beta$, and
\item  $v_\alpha=0$ if $(\alpha,\alpha)\in u'$. 
\end{enumerate} Under the assumption of Lemma \ref{Lkeymap} (2),  $(\pi_u^{-1}(\Delta_u^0), \pi_{u}^{-1}(\Delta_{u}), \psi_{u}(D))$ is a $(6n-3c(u'))$-dimensional spherical triple and the map  
\[
\psi_{u}: (\chi(u')\cap D, \ \chi(u')\cap \partial D)\to (\pi_{u}^{-1}(\Delta_{u})\cap \psi_{u}(D),\  \pi_{u}^{-1}(\Delta_{u})\cap \psi_{u}(\partial D))
\]
is a relative homotopy equivalence.
\end{lem}
\begin{proof}
Apply Lemmas \ref{Lkeymap} and \ref{Lneat} to $V=\psi_{u}^{-1}(\pi_u^{-1}(\Delta^0_u))$, $N(V)=\psi_{u}^{-1}(\pi_u^{-1}(\Delta_u))$, and $u$ in the lemma replaced with $u'$.
\end{proof}

\section{The space $ U_2 $}\label{SU_2}
\begin{defi}\label{DU_n}
\begin{enumerate}

\item Let $u\in |\GP_n|$ and $A'\in \PPP_n$.  Let $P$ be the underlying partition of $A'$. Write $\alpha_0=\min P$ and $\alpha_{p+1}=\max P$. Set $\psi_0(u,j)=s_j$ for $1\leq j\leq n$ and 
\[
\bar\delta_{u,A',i,j}=\max\{|f'(t)-f'(s_i)|\mid f\in D_{R_{A'}}(\Gamma_n), \ s_i\leq t\leq s_j\}.
\] We also set
\[
\begin{split}
\bar R_{u,A'} & =\max\{\, |f(s_k)|\ \mid\ \  f\in D_{R_{A'}}(\Gamma_n),\  k=1,\dots,  n\}, \\
V_{u,A',i} & =\max\{\, |f'(t)|\ \mid \ \  f\in D_{R_{A'}}(\Gamma_n),\ \  |t|\leq \max\{T_{A'}, |s_i|\}\ \}
\end{split}
\] for $1\leq i\leq n$. We denote by $\nu(u,A')\subset \RR^{6n}$ the subspace consisting of $(y_i,w_i)_i$ satisfying the  following conditions: 
\begin{enumerate}
\item For each piece $\alpha\in P$, and for each pair $i,j\in \alpha$ with $1\leq i<j\leq n$, 
\[
\begin{split}
|y_j-(y_i+(s_j-s_i)w_i)| & < \bar\delta_{u,A',i,j}(s_j-s_i),\qquad\text{and} \\
|w_i-w_j| & < \bar \delta_{u,A',i,j}.\\
\end{split}
\] 

\item 
For any $1\leq i\leq n$, the inequalities \ \ $|y_i|<  \bar R_{u,A'}$ \ \ and \ \ $|w_i|< V_{u,A',i}$\ \  hold.\vspace{2mm}
\item  Putting $i_0=\max\alpha_0$,\quad  $i_1=\min\alpha_{p+1}$, the following inequalities hold.\vspace{2mm}
\begin{enumerate}
\item $y_i-y_{i-1}>_1 s_i-s_{i-1}$ \quad for \quad $i\in \alpha_0 \cup \alpha_{p+1}-\{0,1,i_1, n+1\}$,  \vspace{2mm}

\item $w_i>_1 1$ \quad for \quad $i\in \alpha_0\cup \alpha_{p+1}-\{0,n+1\}$, \vspace{2mm}
\item $y_{j}-y_{i_0}>_1-T_{A'}-s_{i_0}$  \quad and \quad $y_{i_1}-y_{j}>_1 s_{i_1}-T_{A'}$\quad for \quad $j\in P^\circ$, \vspace{2mm}
\item  $y_{i_0}<_1s_{i_0}$ and  $y_{i_1}>_1s_{i_1}$, \vspace{2mm}
\end{enumerate}
where in (iii) and (iv), the conditions involving $i_0$ (resp. $i_1$) are empty if $i_0=0$ (resp. $i_1=n+1$). 
\end{enumerate}

\item For  $u\in |\GP_n|$,  we set
\[
\nu_u=\nu(u,A'),
\]
where $A'$ is the frame of the geometric partition of  $u$.
\item  We set
\[
 U_2 = U(\Delta_u,e_u, \nu_u).
\]
See Definition \ref{De_u}  for $\Delta_u$ and $e_u$, and see also Definition \ref{Dintermediate_space}. Clearly, the inclusion $Int(\psi_u(D_{R_{u'}}(\Gamma_n)))\subset \nu_u$ holds. We denote by  $\phi_1: U_2\to U_1$ the collapsing map.
\end{enumerate}
\end{defi} 
The space $\nu_u$ depends only on the geometric partition of $u$.\\
\indent We easily see that $\tilde U(\Delta_u,e_u, Int(\psi_u(D_{R_{u'}}(\Gamma_n))))$ is open in $\tilde U(\Delta_u,e_u, \nu_u)$ so the map $\phi_1$ is continuous (see Definition \ref{Dintermediate_space}). We sometimes call the inequalities in (c) of (1) of Definition \ref{DU_n} the {\em conditions on the first coordinates}. These are used to prove Lemma \ref{Lneighbor_incl} which is, in turn, used to prove Lemma \ref{Lndr}, see also Remark \ref{Rfirst_coord}.
\begin{rem}
We shall explain the intuition on the definition of the neighborhood $\nu_u$. We want to connect the spaces $U_1$ and $U_{\TT}$ by maps. To do so, it is natural to try to connect the compactification of $\pi^{-1}_u(\Delta_u)\cap Int(\psi_{*,u}(D_{R_u}(\Gamma_n)))$ and the space $\TT_G$,  which serves as a building block of $U_\TT$ (see Definition \ref{Dresol_T}) by a chain of collapsing maps. It would be simple if there were an inclusion from  the intersection of the $\eps$-neighborhood of $e_u(\RR^{6p})$ and a product of 3-disks to the space $ Int(\psi_{u}(D_{R_u}(\Gamma_n)))$ or  of reverse direction for sufficiently small $\eps$, but this seems to be too much to hope because it is difficult to control the behaviour of  values of polynomial maps $f(t)$ for large $t$. If we cut off the components corresponding to the minimum and maximum pieces of the underlying partition and consider the $\eps$-neighborhood for the rest of components, this problem is resolved but this modification does not  do well with  the NDR-property of filtrations which we discuss in next subsection. The definition of $\nu_u$ resolves these issues. It has the inclusions from both of $Int(\psi_{u}(D_{R_u}(\Gamma_n)))$ and the intersection of $\eps$-neighborhood of the image of $e_u$ and the product of disks.
\end{rem}
\begin{lem}\label{Lneat_U_n}
Let $u\in |\GP_n|$. Suppose that $u$ is not degenerate at $\pm\infty$. The triple
$(\pi_u^{-1}(\Delta_u^0), \pi_u^{-1}(\Delta_u), \bar \nu_u)$ is a $(6n-3c(u))$-dimensional spherical. Here, we use the notations of Definitions  \ref{De_u}, \ref{DU_1}, \ref{DU_n}, Lemma \ref{LDelta0} and $\bar \nu_u$ is the closure of $\nu_u$ in $\RR^{6n}$. 
\end{lem}
\begin{proof}
The proof is similar to Lemma \ref{Lfunctor_neat} but we need  slight extra consideration because of the conditions on the first coordinate of components in $P-P^\circ$. We fix two numbers $a, b$ such that $a\leq -T_{u'},\ T_{u'}\leq b$. Let $\alpha_0, \alpha_{p+1}$ be the first and last pieces of $P$ respectively. Put $l=n-|\alpha_0|-|\alpha_{p+1}|$. We let a component of an element $(y_i,w_i)_i\in \RR^{6l}$ labeled by $i=|\alpha_0|+1,\dots, n-|\alpha_{p+1}|$. Let $\nu'\subset \RR^{6l}$ be the set of $(y,w)$ which satisfies the inequalities involving $\bar \delta_{u,A',i,j}$ and  the inequalities on $|y_i|$, $|w_i|$ among the defining inequalities of $\nu_u$, and the inequality
\[
a-s_{i_0}-T_{u'}<_1 y_i<_1b-s_{i_1}+T_{u'}\qquad\text{where}\qquad i_0=\max\alpha_0, \ 
i_{1}=\min\alpha_{p+1}.
\]
Intuitively speaking, $\nu'$ is the set obtained by forgetting the first $|\alpha_0|$ and last $|\alpha_{p+1}|$ components from $\nu_u$. Let $e':\RR^{6p}\to \RR^{6l}$ be the affine map given by removing the corresponding components from $e_u$, and $\pi':\RR^{6l}\to e'(\RR^{6p})=\RR^{6p}$ the orthogonal projection. Let $\pi^\vee:\Delta_u\to \Delta^0_u$ be the orthogonal projection and we extend this map to a map $\RR^{6l}\to \RR^{6l}$ by
\[
\pi^\vee(y,w)=(y,w)+e'(\pi^\vee(\pi'(y,w)))-e'(\pi'(y,w))
\]
for $(y,w)\in \RR^{6l}$ as in the proof of Lemma \ref{Lfunctor_neat}. If one of the inequalities involving $\bar\delta_{u,A',i,j}$ becomes an equality for $(y,w)$, the same holds for $\pi^\vee(y,w)$ so $e'(0)\not\in \pi^\vee((\pi')^{-1}(\Delta_u)\cap \partial \nu')$. By this observation, we can construct a retracting homotopy $(\pi')^{-1}(\Delta_u)\cap \partial \nu'\to (\pi')^{-1}(\Delta_u^0)\cap \partial \nu'$ and extends it to a whole homotopy $H^{a,b}_t:\pi^{-1}(\Delta_u)\cap \bar \nu'\to \pi^{-1}(\Delta_u^0)\cap \bar \nu'$ which depends continuously on $(a,b)$, similarly to the lemma. We can define a homotopy 
$\tilde H_t:\pi_u^{-1}(\Delta_u)\cap \bar \nu_u \to \pi_u^{-1}(\Delta_u^0)\cap \bar \nu_u$ by
\[
\tilde H_t(x,v)=(y^0,w^0, H_t^{a,b}(y^1,w^1), y^2,w^2),
\]
where $(y^0,w^0)$ and $(y^2,w^2)$ are the first $|\alpha_0|$ and last  $|\alpha_{p+1}|$ components of $(y,w)$, respectively, $(y^1,w^1)$ is the rest of components, and $a$ and $ b$ are the first components of $y_{i_0}$ and $y_{i_1}$, respectively.
\end{proof}

\subsection{NDR-property of the filtrations on  $U_2$}
\begin{defi}
Let $X$ be a topological space and $A\subset X$ a closed subspace. We say $(X,A)$ is an {\em NDR pair} if there is a neighborhood $N$ of $A$ such that $A$ is a deformation retract of $N$.
\end{defi}
If $(X,A)$ is an NDR pair, the natural map $C_*(X)/C_*(A)\to \bar C_*(X/A)$ is a quasi-isomorphism. Let $\{F_l\}$ be a filtration on a space such that $(F_{l+1},F_l)$ is an NDR-pair. Then the $E_1$-page of the associated spectral sequence is $\oplus_l H_*(F_{l+1}/F_l)$. This fact is fundamental for Vassiliev's and our arguments. The verification of the property that the pair is NDR is often omitted. For the filtrations on the spaces $\Sigma$ and $\bar \Sigma$ defined in section \ref{Spreliminary}, this property holds since the involved spaces are semi-algebraic sets. The somewhat technical definition of the space $\nu_u$ used in the definition of $U_2$ is motivated by ensuring the NDR-property of filtrations so we give an outline of the proof of the property.
\begin{defi}\label{DX[P_0]}
Let $X$ be a space of the form $X=U(D_u,E_u,N_u)$ (see Definition \ref{Dintermediate_space}), and $P_0\in \PP_n$ a partition.
\begin{enumerate}
\item We denote by $X[P_0]\subset X$ the subspace of the basepoint and the elements $(x,v;u)$ such that $P_0$ is equal to or a subdivision of the  partition of the frame of $u$. 

\item We define the {\em rough cardinality filtration} $\{F_l\}$ on $X[P_0]$ by \\
$(y,w;u)\in F_l \iff (y,w;u)=*$ or $|u'|\leq l$, where $u'$ is the frame of $u$.\\
We define the {\em rough complexity filtration $\{\mathscr{F}_k\}$} on $F_l/F_{l-1}$ by \\
$
(y,w; u)\in \mathscr{F}_k \iff (y,w;u)=*$ or  $c(u')\leq k$, where $u'$ is the frame of $u$.\\
We also define the rough cardinality and rough complexity filtrations on a subquotient of $X[P_0]$ similarly to Definition \ref{Dbar_Sigma}.

\end{enumerate}
\end{defi}
We  will apply this definition to $X=U_1,\dots, U_7$.

\begin{lem}\label{Lneighbor_incl}
Let $u, u_1\in |\GP_n|$. Let $A, A_1$ be the geometric partitions of $u, u_1$ respectively, and $A'$  the frame of the geometric partition of $u$. Suppose that $A\in \partial N(A')$ (see Lemma \ref{LN}), $u\in N(u_1)$ and that there is a sequence $\{\tilde u_i\}_i$ in $|\GP_n|$ such  that $\lim_{i\to\infty}\tilde u_i=u$ and  the abstract partition of the frame of $\tilde u_i$ is the same as the partition of $u_1$ for each $i$. Then,
\begin{enumerate}
\item $\bar \nu_u\ (\ =\overline{\nu(u,A')}\ )\ \subset \nu(u,A_1)$, and
\item $\bar \nu_u\cap \pi^{-1}_{u,u_1}(\Delta(u,u_1,\delta_{u_1}))\subset Int(\pi_u^{-1}(\Delta_u))$.
\end{enumerate} 
Here, $\pi_{u,u_1}$ denotes the orthogonal projection $\RR^{6n}\to \tilde e_{u,u_1}(\RR^{6|u_1|})=\RR^{6|u_1|}$  and see Definition \ref{De_u} for other notations.
\end{lem}
\begin{proof}
Part 1 is clear. (In this part, we only use the condition $u\in N(u_1)$.)\\
We shall prove part 2. Let $u'$ be the frame of $u$. We first consider the case that $u$ is not degenerate at $\pm\infty$. For simplicity we assume that $u_1=u$ and $|u|=|u'|+1$ (so $A_1=A$). Let $P, Q$ denote the abstract partitions of $u', u$ respectively. Set $p=|u'|$.  We may assume that the piece $\alpha$ of $P$ which is subdivided by $Q$ is  non-discrete. (If it is dicrete, the proof is trivial.)  Put $i=\min\alpha$, $k=\max\alpha$, and $s_j=\psi_0(u,j)$.
Let $(y,w)\in \bar \nu_u\cap \pi^{-1}_{u,u}(\Delta(u,u,\delta_{u}))$. Set $( x, v)=\pi_u(y,w)$.
We have the following inequality similarly to the proof of Lemma \ref{Lkeymap}.
\[
|(y_j,w_j)_{j\in\alpha}-e_{u,\alpha}( x_\alpha, v_\alpha)|  \leq |(y_j,w_j)_{j\in\alpha}-e_{u,\alpha}(y_i, v_\alpha)|\leq 
n(2n+1)\bar \delta (s_k-s_i).
\]
Here,
$\bar\delta=\bar\delta_{u,A',i,k}$ and $e_{u,\alpha}$ is the map defined in the proof.
Consequently,  we have $|y_i- x_\alpha|\leq n(2n+1)\bar \delta (s_k-s_i)$.
Let $\alpha_1<\alpha_2$ be the two pieces of $Q$ with $\alpha_1\cup\alpha_2=\alpha$. Put $l=\min{\alpha_2}$. Let $ x_{\alpha_c}$ be the $\alpha_c$-component of $\pi_{u,u}(y,w)$ for $c=1,2$. Similarly to the above inequality, we have
\[
|y_i- x_{\alpha_1}|\leq n(2n+1)\bar \delta (s_{l-1}-s_i),\qquad |y_l- x_{\alpha_2}|\leq 2n(2n+1)\bar\delta(s_k-s_l).
\]
Write  $A=(t_1,\dots, t_n)$. By  the condition $A\in \partial N(A')$, the existence of the sequence $\{\tilde u_i\}$, and Lemma \ref{Lproperty_psi2}, we have
\begin{equation}
s_{\alpha_1}(=s_i)=t_{\alpha_1}, \qquad s_{\alpha_2}(=s_l)=t_{\alpha_2}. \label{EQgeom}
\end{equation}
\indent (i)\quad If $(\alpha_1,\beta)\in u$ with some $\beta\not=\alpha_1,\alpha_2$, we have $| x_{\alpha_1}- x_{\beta}|\leq \frac{\delta_{u}}{8^{(n+5)(p+6)}}|s_{\alpha_1}-s_\beta|'$. Here, we note that the $\beta$-components are the same for the projections $\pi_u$, $\pi_{u,u}$ and $s_{\alpha_1}=s_\alpha$. 
It follows that 
\[
| x_\alpha- x_\beta|\leq | x_\alpha-y_i|+|y_i- x_{\alpha_1}|+| x_{\alpha_1}- x_{\beta}|\leq \frac{\delta_{u'}}{8^{(n+5)(p+5)}}|s_\alpha-s_\beta|'.
\]
\indent (ii)\quad If $(\alpha_2,\beta)\in u$ with some $\beta\not=\alpha_1,\alpha_2$, we have $| x_{\alpha_2}- x_{\beta}|\leq \frac{\delta_{u}}{8^{(n+1)(p+6)}}|s_{\alpha_2}-s_\beta|'$. Since $u$ is not degenerate at $\pm\infty$, by definition of $T_{A'}$, we have $|s_i|\leq T_{A'}$ and $|w_i|\leq V_{A'}$. Similarly to the equation (\ref{EQgeom}),  $s_\beta$ equals the corresponding geometric point of $A'$.  By these observations, together with definition of  $\rho_{u'}$, we have
\[
\begin{split}
|y_i-y_l| & \leq |y_l-(y_i+(s_l-s_i)w_i)|+(s_l-s_i)|w_i| \\
& \leq \bar\delta (s_l-s_i)+\frac{\delta_{u'}}{V_{A'}8^{(n+5)(p+10)}}|s_\alpha-s_\beta|'\cdot V_{A'} \\
& \leq \bar\delta (s_l-s_i)+\frac{\delta_{u'}}{8^{(n+5)(p+10)}}|s_\alpha-s_\beta|'.
\end{split}
\]

Combining these inequalities with
\[
| x_\alpha- x_\beta|\leq | x_\alpha- y_i|+| y_i- y_l|+| y_l- x_{\alpha_2}|+| x_{\alpha_2}- x_\beta|,
\]
 we  have $| x_\alpha- x_\beta| \leq \frac{\delta_{u'}}{8^{(n+5)(p+5)}}|s_\alpha-s_\beta|'$ since $|s_{\alpha_2}-s_\alpha|$ is sufficiently smaller than $|s_\alpha-s_\beta|'$.\\
\indent (iii)\quad If $(\alpha_1,\alpha_2)\in u$,  we have
\[
| x_{\alpha_1}- x_{\alpha_2}|\leq \frac{\delta_{u}}{8^{(n+5)(p+6)}}|s_{\alpha_1}-s_{\alpha_2}|.
\]
By the equalities (\ref{EQgeom}) and $\rho_u\ll \rho_{u'}$, we see $|s_i-s_{l-1}|$, $|s_l-s_k|\ll |s_{\alpha_1}-s_{\alpha_2}|$.  This observation, together  with the inequalities obtained before, implies
\[
\begin{split}
| y_{i}-y_l| & \leq \frac{3\delta_{u}}{8^{(n+5)(p+6)}}|s_{\alpha_1}-s_{\alpha_2}| \\
\therefore \ |(s_l-s_i) v_\alpha| & \leq |y_i-y_l|+|y_l-(y_i+(s_l-s_i) v_\alpha)|\leq \frac{\delta_{u'}}{8^{(n+5)(p+5)}}(s_l-s_i) \\
\therefore \ | v_\alpha| &< \frac{\delta_{u'}}{8^{(n+5)(p+5)}}.
\end{split}
\]
Thus, we have $( x_\alpha, v_\alpha)\in Int(\Delta_u)$.\\
\indent We shall consider the case that $u$ is degenerate at $\pm\infty$. The claimed inclusion is  rephrased as $\bar \nu_u\cap\pi_{u,u_1}^{-1}(\Delta(u,u_1,\delta_{u_1}))=\emptyset $. Putting the condition $A\in \partial N(A')$,  the existence of the sequence $\tilde u_i$,  the definition of  $\rho_?$,  Lemma \ref{Lproperty_psi2}, and the condition on $\nu_u$ involving $\bar \delta_{u,A',i,j}$ into together, we see that the distance $|(y_{j'},w_{j'})-(y_{j}+(s_{j'}-s_j)w_j,w_{j})|$ for $j<j'$ in a common piece which is neither $\max Q$ nor $\min Q$, is sufficiently small. This  implies that the condition $(y,w)\in \pi_{u,u_1}^{-1}(\Delta(u,A,\delta_{u_1}))$ and the conditions on the first coordinates in the definition of $\nu_u$ cannot hold simultaneously. Thus,  we have the claim.
\end{proof}

\begin{lem}\label{Lndr}
Let $X$ be either of $U_1$ or $U_2$, and $P_0$ a partition.  Let $\{F_l\}$ the rough cardinality filtration on $X[P_0]$ and $\{\mathscr{F}_k\}$ the rough complexity filtration on $F_{l+1}/F_l$. Then, the pairs $(F_{l+1},F_l)$ and $(\mathscr{F}_{k+1},\mathscr{F}_k)$ are NDR-pairs.
\end{lem}
\begin{proof}
We shall give  a proof for $X=U_2$. For simplicity, suppose that $P_0$ is the partition consisting of singletons, so $X[P_0]=X$.
For the first pair, we can construct a neighborhood $N'$ of $\cup_{|A|\leq l}N(A)$ in $\cup_{|A|\leq l+1}N(A)$ and a retraction $r':N'\to \cup_{|A|\leq l}N(A)$ by using an induction similar to the construction of the map $\bar \psi_l$ in section \ref{SU_1}. Let $\tilde N'\subset U_2$ be the subspace of elements $(y,w;u)$ such that the underlying partition of $u$ belongs to $N'$. We define a retraction $\tilde r':\tilde N'\to F_l$ by 
\[
\tilde r'(y,w;u)=
(y,w;r'(u)). 
\] Here, 
for $u=\tau_0J_0+\cdots +\tau_mJ_m$ with $J_i=(A,G_i)$, we set $r'(J_i)=(r'(A),\delta_{PQ}G_i)$, where $P$ and $Q$ are the partitions of $r'(A)$ and $A$, respectively, and set 
$
r'(u)=\tau_0r'(J_0)+\cdots +\tau_mr'(J_m)
$. The right-hand side is regarded as $*$ if $(y,w)\not \in \nu_{r'(u)}$.
 If we take $N'$ sufficiently small and take $r'$ so that $A$ and $r'(A)$ are sufficiently close,  we have $\nu_u\supset \nu_{r'(u)}$ and $\pi_u^{-1}(\Delta_u)\cap \nu_{r'(u)}\subset \pi_{r'(u)}^{-1}(\Delta_{r'(u)})$, and also have  $(y,w)\not\in \nu_{r'(u)}$ in the case of $\delta_{PQ}(G_m)\not\in\GG(P)$ by Lemma \ref{Lneighbor_incl}. Therefore,  $\tilde r'$ is well-defined. \\
\indent For the second pair, let $N''$ be the subspace of  $\cup_{c(u)\leq k+1}N(u)$ consisting of $u=\tau_0J_0+\cdots +\tau_mJ_m$ such that the sum of $\tau_i$'s such that $J_i$ has the rough complexity $k+1$ is $\leq 1/2$. Here, the rough complexity of $J_i$ is the complexity of  $J'_i$ when we write the frame of $u$ as $\tau_0J'_0+\cdots +\tau_mJ'_m$. A retraction $r'':N''\to \cup_{c(u)\leq k}N(u)$ is given by

\[
r''(u)=\frac{1}{\tau_0+\cdots +\tau_p}(\tau_0J_0+\cdots +\tau_pJ_p)
\] where $p$ is the maximum number such that $J_p$ has rough complexity $\leq k$. A retraction $\tilde r'':\{(y,w;u)\in F_{l+1}/F_l \mid u\in N''\}\to \mathscr{F}_k$ is defined by $\tilde r''(y,w;u)=(y,w;r''(u))$ under the obvious inclusion $\Delta_u\subset \Delta_{r''(u)}$. The map $\tilde r''$ is continuous by the case of degenerate $u$ in the part 2 of Lemma \ref{Lneighbor_incl}. We omit the construction of deformation homotopies since they are obvious extensions of $\tilde r'$ or $\tilde r''$. The proof for $X=U_1$ is similar. The statement corresponding to the part 1 of Lemma \ref{Lneighbor_incl} for $U_1$ is obvious from the definition of $R_?$ and the one corresponding to the part 2  follows from the part 2 itself since $\psi_u(D_{R_{u'}}(\Gamma_n))\subset \nu_u$. 
\end{proof}
\begin{rem}\label{Rfirst_coord}
Note that the retraction $\tilde r''$ in the previous proof is not continuous without the condition  on the first coordinates  in the definition of $\nu_u$. 
For example, set $k=0,\ n=2$ and consider a sequence $u_i=\frac{1}{2}J_{0,i}+\frac{1}{2}J_{1,i}$, where  $J_{0,i}$ is the graph with two vertices $1,a_i$  (of multiplicity 1) other than $\pm\infty$ and with no edge, and $J_{1,i}$ is the graph with the unique edge $(1,a_i)$ (and the same vertex set as $J_{0,i}$). Suppose that the limit $a_\infty=\lim _i a_i$ satisfies $\{-\infty, 1,a_\infty,\infty \}\in \partial(N(\{-\infty,1,\infty^2\}))$ for the geometric partitions, where $\infty^2$ denotes the point of multiplicity 2.  The space $\Delta_{u_i}$ does not get further from the origin when $i\to \infty$. If it were not for the condition, it  could hold that $\lim_{i\to \infty} r''(y,w;u_i)=(y,w;\bar u)\not=*=r''(\lim_{i\to \infty}(y,w;u_i))$  for some $(y,w)$, where $\bar u$ is  the element of $|\GP_n|$ identified with the graph with the vertex set $\{-\infty, 1,a_\infty,\infty \}$ and the empty edge set.

\end{rem}

\subsection{Homological properties of the maps $\psi:\bar \Sigma\to U_1$ and $\phi_1:U_2\to U_1$ }

\begin{prop}\label{Phomology_U_1} Let $P_0$ be a partition. 
\begin{enumerate}
\item Suppose $|P_0^\circ|\leq 3n/5$. When we consider  the cardinality filtration on $\bar \Sigma(P_0)$ and the rough cardinality filtration on $U_1[P_0]$, the restriction $\bar \Sigma(P_0)\to U_1[P_0]$  of the map $\psi$ defined in subsetion \ref{SSpsi} preserves the filtrations and induces an isomorphism between the homology of the filtered quotients $H_*(F_{l+1}/F_l)$ for each $l$.
\item When we consider the rough cardinality filtrations on $U_1[P_0]$ and $U_2[P_0]$, the restriction
$ U_2[P_0] \to U_1[P_0]$ of $\phi_1$ (see Definition \ref{DU_n}) preserves the filtrations and induces an isomorphism between the homology of the filtered quotients $H_*(F_{l+1}/F_l)$ for each $l$.
\end{enumerate}
\end{prop}
To proving this, we need the following lemma.
\begin{lem}\label{Lneighbor_partition}
Let $P$ be a partition with $|P^\circ|=l+1$. Let $[P]$ be the set of geometric partitions whose underlying abstract partition is $P$, and $\bar {[P]}$ the closure of $[P]$ in $\PPP_n$. 
The pair $(\bar{[P]}, \bar{[P]}\cap(\cup_{A\in \partial [P]}N(A)))$ is relatively homotopy equivalent to $(\Delta^{l+1},\partial\Delta^{l+1})$. 
\end{lem}
\begin{proof}

Clearly, $\bar{[P]}$ is homeomorphic to $\Delta^{l+1}$ by forgetting the multiplicity of points. Let $\partial_i[P]\subset \partial [P]$ be the union of open faces of dimension $\geq i$. We can construct a retracting homotopy $\cup_{A\in \partial_i[P]}N(A)\to \partial _i[P]$ inductively on $i$ in a manner similar to the construction of $\bar \psi_l$ in subsection \ref{SSpsi} and extend it to the homotopy inverse of the inclusion $(\Delta^{l+1},\partial \Delta^{l+1})\to (\bar{[P]}, \bar{[P]}\cap(\cup_{A\in \partial [P]}N(A)))$. 
\end{proof}

\begin{proof}[Proof of Proposition \ref{Phomology_U_1}]
We shall show part 1. Part 2 is similar. For $X=\bar\Sigma(P_0)$ (resp. $U_1[P_0]$), we denote by $\{F_l=F_l(X)\}$ the cardinality filtration (resp. the rough cardinality filtration) and by $\{\mathscr{F}_k\}$ the complexity (resp. the rough complexity) filtration on $F_{l+1}(X)/F_l(X)$. Clearly, the restriction of $\psi$ preserves the filtrations.  Let $PC$ be the space of partitional configurations and  $PC_{P_0}\subset PC$ be the subspace of partitional configurations $(A,H)$ such that $|A|={l+1}$, $c(H)=k+1$, and $P_0$ is equal to or a subdivision of the underlying partition of $A$ (see Definition \ref{Dpart_config}).  The space $\mathscr{F}_{k+1}/\mathscr{F}_{k}$ is homotopy equivalent to 
\[
\bigvee_{(A,H)}S^{l+1}\wedge (|\GG^H|/|\GG^{<H}|)\wedge S^{6n-3k-3}.
\]
Here, $(A,H)$ runs through the equivalence classes of configurations belonging to $PC_{P_0}$, and $\GG^H$ and $\GG^{<H}$ are the posets of graphs defined in the proof of Proposition \ref{Lisom_sigma}. For $X=\bar \Sigma$, this was already seen in the proof. We shall show this claim for $X=U_1$. We fix a partitional configuration $\tilde H_0=(A_0,H_0)\in PC_{P_0}$. Let $C_{\tilde H_0}\subset PC_{P_0}$ be the subspace of configurations equivalent to $\tilde H_0$. 
The underlying partitional configuration of $u\in |\GP_n|$ is by definition, the pair of the geometric partition of $u$ and the underlying classed configuration of  $\LL(u)$. 

 Let $\tilde N(A)\subset PC$ be the subset of  configurations whose geometric partition belongs to $N(A)$. Let $P$ be the underlying partition of $A_0$. Set 
\[
\begin{split}
C_{\tilde H_0}' & =C_{\tilde H_0}-\cup_{A\in \partial [P]}\tilde N(A),\\
N_{P} &=\cup_{A'}N(A'),
\end{split}
\]
where $A'$ runs through $[P]-\cup_{A\in \partial [P]}N(A)$  (see the previous lemma). 
 We fix an abstract graph $\bar G_0\in \GG(P)$ which is the underlying graph of an element of $ |\GP_n|$ whose underlying configuration is  $\tilde H_0$.
 Let $\Pi: E\to N_{P}$ be the fiber bundle whose fiber over $A\in N_P$ is  $\nu_{u_A}\cap (\pi_{u_A}^{-1}\Delta_{u_A})$, where $u_A$ is an element such that its underlying geometric partition is $A$ and the image of the abstract graph of $\LL(u_A)$  by $\delta_{PQ}$ is $\bar G_0$, where $Q$ is the underlying abstract partition of $A$. The space $\nu_{u_A}\cap (\pi_{u_A}^{-1}\Delta_{u_A})$ does not depend on choices of $u_A$ and $E$ is well-defined. We easily see that $N_P$ is contractible similarly to Lemma \ref{Lneighbor_partition}, so   the bundle $E$ is trivial. Let $M_{\tilde H_0}\subset |\GP_n|$ be the union of $N(u')$ for $u'$ such that the underlying  configuration of $u'$ belongs  to $C_{\tilde H_0}'$.  Let $M_{\tilde H_0}\to N_{P}$ be the map sending $u$ to its partition. The wedge summand of $\mathscr{F}_{k+1}/\mathscr{F}_{k}(X)$ labeled by $\tilde H_0$ is homeomorphic to the one-point compactification of 
\[
M_{\tilde H_0}\times _{N_{P}}E.
\] 
We can define a proper homotopy equivalence $M_{\tilde H_0} \to C_{\tilde H_0}'\times (\GG^{H_0}-\GG^{<H_0})$ by 
\[
u=\tau_0J_0+\cdots +\tau_mJ_m \quad \longmapsto \quad (H_u, \tau_0\delta_{PQ}(G_0)+\cdots +\tau_m\delta_{PQ}(G_m)),
\] 
where $H_u$ is the underlying partitional configuration of the frame of $u$, $G_i$ is the underlying graph of $J_i$, and $Q$ is the partition of $u$. We have
\[
M_{\tilde H_0}\times _{N_{P}}E\simeq M_{\tilde H_0}\times_{N_{P}}(N_{P}\times (\nu_{u_A}\cap \pi_{u_A}^{-1}(\Delta_{u_A}))\simeq C_{\tilde H_0}'\times  (\GG^{H_0}-\GG^{<{H_0}})\times (\nu_{u_A}\cap \pi_{u_A}^{-1}(\Delta_{u_A})) .
\] This equivalence is proper and we also have equivalences  $(\nu_{u_A}\cap \pi_{u_A}^{-1}(\Delta_{u_A}))^*\simeq S^{6n-3k-3}$ and $(C_{\tilde H_0}')^*\simeq S^{l+1}$ (see Lemma \ref{Lneighbor_partition}). Thus,  we have proved the claimed homotopy equivalence. We also easily see that  the map between $\mathscr{F}_{k+1}/\mathscr{F}_{k}$ induced by $\psi$  is homotopic to the wedge of homotopy equivalences on each of the three factors by Lemmas \ref{Lneat_U_n} and \ref{LDelta0}. Therefore,   it induces an isomorphism between the filtered quotients.
\end{proof}
\section{The spaces  $ U_3,\  U_4,\  U_5,\ U_6$ and $U_7$}\label{SU_3}
\begin{defi}\label{Dcyl} Let $u\in |\GP_n|$. We denote by $A'$ the frame of the underlying geometric partition of $u$. Let $P$ be the underlying partition of $A'$. We  put $s_i=\psi_0(u,i)$ and $p=|A'|$. 
\begin{enumerate}
\item We set
\[
\begin{split}
\mu_0 &=\frac{\min\{s_{i+1}-s_i\mid 1\leq i\leq n-1\}}{8(s_n-s_1)}\ , \\
& \\
\eps(u) &=\mu_0^{(n+5)(p+5)}\min\{\bar\delta_{u,A',i,j}\mid i,j\text{ belongs to the same piece of $P$}\}\ , \\
& \\
\bar \eps(u) &= \mu_0^{(n+5)(3n-p+5)}\delta_{A'}.
\end{split}
\]
See Definition \ref{DU_n} for $\bar \delta_{u,A',i,j}$.

\item Let $\tau \in [0,1]$.
We set $ e_{u,\tau}=(1-\tau) e_{u}+\tau e_P$, where the sum and scalar multiple are taken in the pointwise manner (see Definitions \ref{De_u} and  \ref{Dpunctured}), and
\[
c^v_{\alpha\beta,\tau}=(1-\tau)(|\alpha|_u\cdot |v_\alpha|+|\beta|_u\cdot |v_\beta|)+\tau  c^v_{\alpha\beta},
\]
where 
\[
|\gamma|_u=s_{\max\gamma}-s_{\min\gamma}+\min\{s_{i+1}-s_i\mid 1\leq i\leq n-1\}\quad \text{for}\ \gamma\in P^\circ.
\]
Let $\pi_{u,\tau}:\RR^{6n}\to e_{u,\tau}(\RR^{6p})=\RR^{6p}$ be the orthogonal projection. Let 
$
\nu_{u,\tau}\subset \RR^{6n}
$
be the subset consisting of elements $(y,w)$ satisfying the following conditions.
\begin{itemize}
\item The distance between $(y,w)$ and $e_{u,\tau}(\RR^{6p})$ is $<\eps(u)$.\vspace{2mm}
\item $|y_i|\leq (1-\tau)\bar R_{A'}+\tau ,\ |w_i|\leq 3(1-\tau)+\tau$ for each $1\leq i \leq n$.\vspace{2mm}
\item Putting $( x^\tau, v^\tau)= \pi_{u,\tau}(y,w)$,  for $\alpha\in P^\circ$,
\[
\begin{split}
(1-\tau)(-2T_{A'}-\frac{p}{n})+\tau(-1+c_{\leq \alpha }^{ v}+c_\alpha^{ v}-\eps(u)) & <_1  x_\alpha^\tau +\tau  c_\alpha  v_\alpha^\tau \qquad \text{and} \\
 x_\alpha^\tau -\tau  c_\alpha  v_\alpha ^\tau &  <_1 (1-\tau)(2T_{A'}+\frac{p}{n})+\tau(1- (c_{\geq \alpha}^{ v}+c_\alpha^{ v})+\eps(u)).
\end{split}
\]
\end{itemize}

\item 
We define a subset $\Delta_{u,\tau}\subset \RR^{6p}$ as follows:
\begin{enumerate}
\item If $u$ is degenerate at $\pm\infty$, we set $\Delta_{u,\tau}=\emptyset$.
\item Otherwise,  $\Delta_{u,\tau}$ is  the subset of elements $(x,v)$ such that \vspace{2mm}
\begin{itemize}
\item $|x_\alpha-x_\beta|\leq \max\{c^v_{\alpha\beta,\tau}-\eps(u), \bar\eps(u)\}$ for $(\alpha,\beta)\in u'$, and \vspace{2mm}
\item $|v_\alpha|\leq \bar \eps(u)$ for $(\alpha,\alpha)\in u'$,\vspace{2mm}
\end{itemize} 
where $u'$ is the frame of $u$.
\end{enumerate}

\end{enumerate}
\end{defi}

\begin{defi}\label{DU_3} We use the notations in the previous definition.
\begin{enumerate}
\item We set
\[
U_3=U(\Delta_u, e_u, \nu_{u,0}),
\]
where $\Delta_u$  is defined in Definition \ref{De_u}. 
  We easily see $\nu_{u,0}\subset \nu_u$ so we have the obvious collapsing map $\phi_2:U_2 \to  U_3$.
\item  We set
\[
U_4= U(\Delta_{u,0}, e_u,  \nu_{u,0})
\] 
Clearly, $\Delta_{u,0}\subset \Delta_u$ so we have the map $\phi_3:U_4\to  U_3$ induced by the inclusion. 

\item We define a subset
\[
\begin{split}
 \tilde U_5 & \subset \RR^{6n}\times|\GP_n|\times [0,1] \quad \text{by}\quad \\
(y,w ; u,\tau)& \in \tilde U_5 \iff (y,w ; u)\in \tilde U(\Delta_{u,\tau},e_{u,\tau},\nu_{u,\tau}), 
\end{split}
\]
and set
\[
 U_5=(\tilde U_5)^*/\sim, 
\]
where for $(y,w ; u=k_0J_0+\cdots +k_mJ_m,\tau) \in \tilde U_5$ with $k_0\not =0$, we declare $(y,w;u,\tau)\sim *$ if and  only if $J_0$ has a non-empty edge set (see Definition \ref{Dintermediate_space}).
 We define a map $\phi_4:U_4\to U_5$ to be the inclusion to $\tau=0$.
\item We set
\[
 U_6 = U(\Delta_{u,1},  e_{u,1},  \nu_{u,1}).
\]
 We define a map $\phi_5:U_6\to U_5$ to be the inclusion to $\tau=1$.
\end{enumerate}
\end{defi}
The following lemma is proved completely similarly to Lemma  \ref{Lfunctor_neat}
\begin{lem}\label{Lneat_cyl}
Let $\tau\in [0,1]$ and $u\in |\GP_n|$. Let $u'$ be the frame of $u$. If $u'$ is not degenerate at $\pm\infty$, the triple $(\pi_{u,\tau}^{-1}(\Delta_{u}^0), \pi_{u,\tau}^{-1}(\Delta_{u,\tau}),\bar  \nu_{u,\tau})$ is a $(6n-3c(u'))$-dimensional spherical triple.  Furthermore a retracting homotopy $\pi_{u,\tau}^{-1}(\Delta_{u,\tau})\cap \bar \nu_{u,\tau}\to \pi_{u,\tau}^{-1}(\Delta_{u}^0)\cap \bar\nu_{u,\tau}$ preserving $\partial \nu_{u,\tau}$ can be taken continuously on $\tau$ (see Lemma \ref{LDelta0} for $\Delta_u^0$).\hfill \qedsymbol
\end{lem}
The proof of the following claim  is completely similar to the proof of Proposition \ref{Phomology_U_1} in view of the previous lemma and the NDR-property of the filtration on $U_5$, which is similar to the proof of Lemmas \ref{Ldiagonal_bound} and \ref{Ldiagonal_incl} (rather than Lemma \ref{Lneighbor_incl}). 
\begin{prop}\label{PU_2U_6}
Let $P_0\in \PP_n$. When we consider the rough cardinality filtrations $F_l$ on the spaces $U_2[P_0],\dots, U_6[P_0]$,
the restrictions of the maps $\phi_2,\dots,\phi_5$ given in Definition \ref{DU_3} to the subspaces preserve the filtrations and induce isomorphisms between the homology of  filtered quotients $H_*(F_{l+1}/F_l)$ for each $l$. Here, $U_5[P_0]$ is the subspace of $U_5$ defined by the condition on the underlying partition similar to  $X[P_0]$ in Definition  \ref{DX[P_0]}. \hfill \qedsymbol
\end{prop}

The reader would find that $U_6$ and  $U_\TT$ (see Definition \ref{Dintermediate_space}) resemble strongly. The differences are the numbers $\eps$, $\bar\eps$ used to define the `width' of neighborhoods and whether we consider the margin $N(u)$ or not. The folloing $U_7$ adjusts the difference of $\eps$ and $\bar \eps$.

\begin{defi}\label{DU_7}
Let $u\in |\GP_n|$. We use the notations of Definition \ref{Dcyl}. 
\begin{enumerate}

\item We set
\[
\begin{split}
\eps_\tau(u) & =(1-\tau)\eps(u)+\tau\eps_P \\
\bar \eps_\tau(u) & =(1-\tau)\bar \eps(u)+\tau\bar\eps_P
\end{split}
\]
Here,  $P$ is the abstract partition of the frame of the geometric partition of $u$. 
\item We define a subset $\Delta'_{u,\tau}\subset \RR^{6p}$ as follows.
\begin{itemize}
\item $\Delta'_{u,\tau}=\emptyset$ if $u$ is degenerate at $\pm\infty$.
\item Otherwise, it is the set of $(x,v)$ 
such that 
$|x_\alpha-x_\beta|\leq \max\{ c_{\alpha\beta}^v-\eps_\tau(u),\ \bar \eps_\tau(u)\}$  for $ (\alpha,\beta)\in u'$ with $\alpha\not=\beta$, and $ |v_\alpha|\leq \bar\eps_\tau(u)$ for $(\alpha,\alpha)\in u'$, where $u'$ is the frame of $u$.

\end{itemize}
\item
Let $P$ be the abstract partition of the frame of the geometric partition of $u$.
Let $
\nu_{u,\tau}'\subset \RR^{6n}
$ be the subset consisting of 
$(y,w)$ satisfying the following conditions. 
\begin{itemize}
\item The distance between $(y,w)$ and $e_{P}(\RR^{6p})$ is $<\eps_\tau(u)$.
\item $|y_i|<  1,\ |w_i|< 1$ for each $1\leq i \leq n$.
\item Putting $ \pi_{P}(y,w)=( x, v)$, for any $\alpha\in P^\circ$,
\[
-1+ c_{\leq \alpha}^{ v}+c^v_{\alpha}-\eps_\tau(u)<_1 x_\alpha+c_\alpha v_\alpha, \qquad x_\alpha-c_\alpha v_\alpha<_1 1- (c_{\geq \alpha}^{ v}+c^v_\alpha)+\eps_\tau(u).
\]
\end{itemize}

\item We define a subset
\[
\begin{split}
 \tilde U_7 & \subset \RR^{6n}\times|\GP_n|\times [0,1] \quad \text{by}\quad \\
(y,w;u,\tau)& \in \tilde U_7 \iff (y,w; u)\in \tilde U(\Delta'_{u,\tau},e_{u}',\nu_{u,\tau}' ),
\end{split}
\]
where $e_u'=e_P$ with $P$ being the partition of the frame of $u$.
We set
\[
U_7=(\tilde U_7)^*/\sim, 
\]
where for $(y,w,u=k_0J_0+\cdots +k_mJ_m,\tau)$ with $k_0\not=0$, we declare that $(y,w,u,\tau)\sim *$ $\iff$ $J_0$ has a non-empty edge set.
\item We let $\phi_6:U_6\to U_7$ be the inclusion to $\tau=0$. We define a map  $\phi_7:U_{\TT}\to U_7$ by 
\[
\phi_7(x,v;u)=i_1(\delta'_{PQ}(x,v) ; u),
\]
where $Q$ and $P$ are the underlying abstract partitions of $u$ and of the frame of the geometric partition of $u$, respectively, $\delta'_{PQ}$ is the map given in Lemma \ref{Ldiagonal_bound}, and $i_1$ is the inclusion to $\tau=1$. 

\end{enumerate}

\end{defi}
Since  $\cup_{|A|\leq l}N(A)$ is closed in $\PPP_{n,l}$ which is open in $\PPP_n$, the map $\phi_7$ is continuous. The following claim is proved completely similarly to Proposition \ref{Phomology_U_1}.
\begin{prop}\label{Pphi_6}
Let  $P_0\in \PP_n$.  When we consider the rough cardinality filtration on $U_6[P_0]$ and $U_7[P_0]$ and the cardinality filtration on $U_\TT(P_0)$, the following maps given by the restrictions of $\phi_6$, $\phi_7$  preserve the filtrations and induce isomorphisms between homology of the filtered quotients:
\[
U_6[P_0]\stackrel{\phi_6}{\longrightarrow}U_7[P_0]\stackrel{\phi_7}{\longleftarrow} U_\TT (P_0).
\]
See the paragraph after Proposition \ref{LThomSinha} for $U_\TT(P_0)$.
\end{prop}

\section{Proof of the results in Introduction and miscellaneous topics}\label{Sproof}
 Putting Propositions \ref{Lisom_sigma}, \ref{LThomSinha}, \ref{Phomology_U_1}, \ref{PU_2U_6} and \ref{Pphi_6} into together, we obtain the following theorem. 
\begin{thm}\label{TunstableVS}
Let $l_0$ be an integer with $l_0<\frac{3n}{5}$. The homology spectral sequence associated to $\Sigma_{\leq l_0}=F_{l_0}(\Sigma)$ with the cardinality filtration is isomorphic to the $l_0$-truncated Sinha sequence from the $E_2$-page under the degree shift $(l,d)\leftrightarrow (-l, 6n-d)$. \hfill \qedsymbol
\end{thm}

\begin{proof}[Proof of Theorem \ref{TE_infinity}]
We give a detailed proof since the Vassiliev sequence does not directly come from a filtered complex. Let $\FFF_k$ and $F_l$ be the complexity and cardinality filtrations on $\Sigma$, respectively. 
For a subspace $X$ of $\Sigma$, and filtration $F'_i=\FFF_k\ or\ F_l$ on $\Sigma$, we denote by $ ^{F'}E_r(X)$ the spectral sequence associated to $(\bar C_*(X), \{\bar C_*(X\cap F'_i)\})$.

Note that the homogeneous elements of the $E_1$-page of the Vassiliev sequence and the $E_2$-page of the Sinha sequence are equipped with a cardinality $l$ and complexity $k$. The corresponding bidegrees are $(-k,-l+3k)$ and $(-l, 2k)$, respectively, so   these labels are  inherited to the homogeneous elements of $E_\infty$-pages. Let 
\[
^VE_\infty(k,l)\subset {}^VE_\infty,\qquad {}^SE_\infty(k,l)\subset {}^SE_\infty
\] be the vector subspaces generated by elements labeled with cardinality $\leq l$ and complexity $\leq k$. We consider the similar subspaces $^\mathscr{F}E_\infty(\mathscr{F}_{k_1})(k,l)$,  $^FE_\infty({F}_{l_1})(k,l)$ of $^\mathscr{F}E_\infty(\mathscr{F}_{k_1})$ and $^FE_\infty({F}_{l_1})$ for integers $k_1,\ l_1$. These are finite dimensional. If we take $k_1$, $l_1$ and $n$, the integer used to define the space $\Gamma_n$ of polynomial maps, so that $k,l\ll k_1, l_1\ll n$, we have 
\[
\dim {}^\mathscr{F}E_\infty(\mathscr{F}_{k_1})(k,l)=\dim {}^VE_\infty(k,l),\qquad \dim {}^FE_\infty({F}_{l_1})(k,l) =\dim {}^SE_\infty(k,l).
\]
The first equality follows from the definition of  Vassiliev sequence and the second from Theorem \ref{TunstableVS}.
Set  $X_{kl}=\mathscr{F}_k\cap F_l\subset \Sigma$. The space ${}^\mathscr{F}E_\infty(\mathscr{F}_{k_1})(k,l)$ is the associated graded of the image of the map $\bar H_*(X_{kl})\to \bar H_*(\mathscr{F}_{k_1})$ induced by the inclusion, and  the space ${}^FE_\infty(F_{l_1})(k,l)$ is that of the image of the map $\bar H_*(X_{kl})\to \bar H_*(F_{l_1})$. If we take $k_1, l_1$ so that $2k_1<l_1$ (in addition to the above inequalities), we have $\dim {}^\mathscr{F}E_\infty(\mathscr{F}_{k_1})(k,l)\leq \dim {}^FE_\infty({F}_{l_1})(k,l)$ as $\FFF_{k_1}\subset F_{l_1}$, so $\dim {}^VE_\infty(k,l)\leq \dim {}^SE_\infty(k,l)$. If $l_1<2k_1$, we have the opposite inequality. Thus,  we have  $\dim {}^VE_\infty(k,l)= \dim {}^SE_\infty(k,l)$.
\end{proof}

\begin{proof}[Proof of Theorem \ref{Cdegree0}]
The diagonal part of the Vassiliev or  Sinha sequence is a submodule of a free $\kk$-module, which is also free since $\kk$ is a principal ideal domain. The rest of the proof is similar to  Theorem \ref{TE_infinity}.
\end{proof}
\begin{rem}
Kontsevich's configuration space integral which was used in the proof of  degeneracy of the diagonal part has been generalized extensively to give interesting construction of non-trivial cocycles of embedding spaces beyond knot spaces, see e.g. Cattaneo--Cotta-Ramusino--Longoni \cite{CCL}, Cattaneo-Rossi \cite{CR}, and Sakai-Watanabe \cite{SW}.
\end{rem}
Recall the space of long knots modulo immersions $\bar K_3$ from Introduction.  
\begin{thm}[Boavida de Brito-Horel\cite{BH}]
Let $p$ be a prime number and $\kk=\FF_p$ or $\ZZ_{(p)}$. The only possibly non-trivial differentials in the cohomological  Sinha spectral sequence $\bar E_r$ for $\bar K_3$ are $d_{1+2k(p-1)}$ for $k\geq 0$.
\end{thm}
\begin{proof}
This is essentially proved in \cite{BH}. We give an outline of proof for $\kk=\FF_p$. In \cite{BH}, a $p$-completed version of Sinha cosimplicial space $L_pK_3^\bullet$ for the space of long knots modulo immersions is constructed. It admits a map from the original cosimplicial space $K_3^\bullet$ (for knots modulo immersions) in the homotopy category of cosimplicial spaces such that the  map $K_3^l\to L_pK_3^l$ at each cosimplicial degree $l$ is the $p$-completion. Therefore,  with $\FF_p$ coefficients,  Bousfield-Kan cohomology spectral sequences for $L_pK_3^\bullet$ and $K_3^\bullet$ are isomorphic. The space $K_3^l$ is homotopy equivalent to the ordered configuration space of $l$ points in $\RR^3$ so its cohomology group is free. An action of a version of Grothendieck-Teichm\"uller group $GT_p$ on $L_pK_3^\bullet$ is constructed in \cite{BH1, BH} and the action induces an action  on $\bar E_r$. In \cite{BH}, the action on $H^k(L_pK_3^l,\ZZ_p)$ is proved to be cyclotomic of weight $2k$, i.e. it factors through the action of $(\ZZ_p)^\times$ by multiplication of the $2k$-th power, where $\ZZ_p$ is the $p$-adic integers. Since the natural isomorphism $H^k(L_pK_3^l,\ZZ_p)\otimes \FF_p\to H^k(L_pK_3^l,\FF_p)$ is compatible with the $GT_p$-action, the action on $H^k(L_pK_3^l,\FF_p)$ is also cyclotomic of weight $2k$ by Proposition 2.1 of \cite{BH}. This implies the action on $\bar E_r^{*,2k}$ is also cyclotomic of weight $2k$. Since maps between modules with cyclotomic actions of different weights must be zero, this implies the claim.
\end{proof}

\begin{proof}[Proof of Corollary \ref{Cinvariants}]
 The natural inclusion from the cosimplicial model of $\bar K_3$ to that of $K_3$  induces a map between spectral sequences ${}^SE_r\to \bar E_r$. In \cite{turchin}, this map is proved to be a monomorphism at the $E_2$-page. 
The part of total degree 1 and complexity $\leq 2$ is zero. The part of complexity $k$ is $^SE_2^{*,2k}$ and the differential $d_{1+2(p-1)}$ decreases the degree on the right-hand side by $2(p-1)$. Therefore,   a sufficient condition for $E^{-2k,2k}_2$ to be stationary after $E_2$ is $2k-2(p-1)\leq 4$. This implies $k\leq p+1$.
\end{proof}
We shall compare Corollary \ref{Cinvariants}  with a result of Boavida de Brito-Horel.
For $f,g\in \pi_0(K_3)$, write $f\sim_n g$ if the values of any type $n$ invariants at $f$ and $g$ coincide. The relation $\sim_n$ is an equivalence relation, and the set 
$\KKK_n=\pi_0(K_3)/\sim_n$ is a finitely generated abelian group under the connected sum (see \cite{gusarov}). A monoid map from $\pi_0(K_3)$ to an abelian group $A$ which factors through $\KKK_n$ is called an {\em  additive (finite) type $n$ invariant}. Boavida de Brito-Horel \cite{BH} proved that $\KKK_n\otimes \ZZ_{(p)}$ is isomorphic to the space of indecomposable Feynman diagrams $\oplus_{s\leq n}\mathcal{A}^I_s$, a combinatorial object similar to the weight system  if $n\leq p+1$. While the additive finite type invariants are included in the (not-neccesarily additive) finite type invariants, the corollary does not imply the result of Boavida de Brito-Horel (and vice versa). One reason for this is that $\KKK_n$ may have higher order torsions. (Any torsion of $\KKK_n$ is not known at present.) The other is that the relation between the additive  and non-additive invariants is not enough clear for coefficients other  than $\QQ$.  \\
\subsection{A conjecture on naturality of the isomorphism}
 Recall the subspace $U_{\TT}(P_0)\subset U_{\TT}$ from the paragraph after Proposition \ref{LThomSinha}. Since the total complex of the normalization of a functor $\Delta_n^{op}\to \CH_\kk$ is naturally quasi-isomorphic to the homotopy colimit of the functor, using the natural transformations given in section \ref{Stranslation} we have a zigzag of quasi-isomorphisms 
$\bar C_*(U_\TT)\simeq \mathrm{hocolim}_{\PP_n} C^*(\PK)$. Composing this with the canonical maps $\mathrm{hocolim}_{\PP_n} C^*(\PK)\to C^*(\mathrm{holim}_{\PP_n}\PK)\to C^*(K_3)$ and the map $C_*(U_\TT(P_0))\to C_*(U_\TT)$ induced by the inclusion, we have a map 
$\iota_0:\bar H_*(U_\TT(P_0))\to H^*(K_3)$. Composing the duality map $\bar H_*(\Sigma)\cong H_*( \Gamma_n^*, Sing^*)\cong H^*(\Gamma_n-Sing)$ and the map $H^*(K_3)\to H^*(\Gamma_n-Sing)$ induced by the inclusion, we obtain a map $\iota_1:H^*(K_3)\to \bar H_*(\Sigma)$ (see section \ref{Spreliminary}). The maps $\psi,\phi_1,\dots, \phi_7$ given in previous sections induce an isomorphism $\bar H_*(U_{\TT}(P_0))\cong \bar H_*(\bar \Sigma(P_0))$ if $|P_0^\circ|\leq 3n/5$. Composing 
\begin{spacing}{1.2}
\noindent
 this with the map induced by $\phi_0:\bar \Sigma\to \Sigma$, we have a map $\tilde \phi :\bar H_*(U_\TT(P_0))\to \bar H_*(\Sigma)$.
\end{spacing}
\vspace{-0.5\baselineskip}
\begin{conj}\label{Conj_compati}
Let $P_0$ be a partition with $|P_0^\circ|\leq 3n/5$. Under the above notations, the following diagram would be commutative
\[
\xymatrix{ H^*(K_3)\ar[rd]^{\iota_1} & \\
\bar H_*(U_{\TT}(P_0))\ar[u]^{\iota_0}\ar[r]^{\tilde \phi} & \bar H_*(\Sigma)\ .}
\]
\end{conj}
This conjecture is related to a conjecture by Budney-Conant-Koytcheff-Sinha which claims  equivalence between the embedding calculus invariants and additive finite type invariants. Precisely speaking, the canonical map $\pi_0(K_3)\to \pi_0(T_{n+1}K_3)$ is known to factor through $\KKK_n$ and the induced map $\alpha_n:\KKK_n\to \pi_0(T_{n+1}K_3)$ forms a map between the two towers  $\{\KKK_n\}_n$ and $\{\pi_0(T_{n+1}K_3)\}_n$. 
Budney-Conant-Koytcheff-Sinha \cite{BCKS} conjectured that {\em the map $\alpha_n:\KKK_n\to \pi_0(T_{n+1}K_3)$ would be bijective for any $n$}. This conjecture has been attracting much attention recently. Kosanovi\'c \cite{kosanovic} proved that the map is surjective and Boavida de Brito-Horel \cite{BH} proved  that the map is bijective after tensored with $\ZZ_{(p)}$ if $n\leq p+1$. If Conjecture \ref{Conj_compati} is true, it follows that the map induces a bijection between the {\em inverse limit} of the towers {\em without any Grothendieck-Teichm\"uller action or operad formality}. This is because we can use the conjecture to take a map $\beta_n:\pi_0(T_{2n}K)\to \KKK_{n}$ such that the following diagram is commutative:
\[
\xymatrix{\KKK_{2n}\ar[r] \ar[d]& \pi_0(T_{2n}(K_3)) \ar[ld]^{\beta_n}\ar[d]\\
\KKK_n\ar[r] & \pi_0(T_n(K_3)), }
\]
where the vertical maps are the structure maps of the towers and the horizontal maps are $\alpha_{2n}$ and $\alpha_n$ followed by the structure maps.

\subsection{The unstable spectral sequence}\label{SSunstable}
We use the notations in Definition \ref{Dunreduced_Vassiliev} and the proof of Theorem \ref{TunstableVS}. Let $\mathscr{F}_k$ (resp. $F_l$) be the complexity (resp. cardinality) filtration of $\Sigma$. We denote by $F_{l,r}^E$ the image of the map $^{\mathscr{F}}E_r(F_l)\to E_r(n)$ induced by the inclusion. The maximum of complexity of generating graphs of cardinality $n$ is $2n-1$ so $F_n$ intersects with  the unstable range of complexity $>6n/5$. Vassiliev conjectured that {\em for all $p,q$, there would exist a number $n_0$ such that  for any $n\geq n_0$ and $r$, we have $E^{p,q}_r(n)\cong {}^V\!E^{p,q}_r$}. He verified this for $p+q=0$ (see \cite{vassiliev}). This conjecture is related to whether the Vassiliev $E_\infty$-page gives non-trivial elements of the genuine cohomology of $K_3$, and is true if all the differentials coming from the unstable range $p< -6n/5$ and  going into the stable range are zero. The difficulty in the analysis of the unstable range is that the corresponding subspace of $\mathscr{F}_k/\mathscr{F}_{k-1}$ does not have a natural cell structure. As a byproduct of the construction of the isomorphism of the  two spectral sequences, we have the following partial computation.
\begin{prop}\label{Punstable_diff}
Suppose $\kk=\QQ$. Under the above notations, for $x\in F_{n,r}^E$ and $r\geq 1$, if $d_r(x)$ belongs to the stable range $\{(p,q)\mid p\geq -6n/5\}$, we have $d_r(x)=0$.
\end{prop} 
\begin{proof}
We consider the complexity filtration $\mathscr{F}_k=\FFF_k(X)$ on $X=\bar\Sigma$ or $F_n(\Sigma)$, the $n$-th stage of the cardinality filtration on $\Sigma$, and the cardinality filtration $\bar F_l=\bar F_l(X)$ on $\FFF_k/\FFF_{k-1}$. The space $\bar F_l'=\bar F_l'(X)=\bar F_l(X)/\bar F_{l-1}(X)$ is a wedge sum of components labeled by the equivalence classes of classed configurations (resp. partitional configurations) of complexity $k$ and cardinality $l$ for $X=F_n(\Sigma)$ (resp. $\bar \Sigma$), see Definition \ref{Dpart_config}. This is the case even if $k$ belongs to the unstable range.   Consider the case $X=F_n(\Sigma)$. Let $d_1:\bar H_*(\bar F'_l)\to \bar H_*(\bar F'_{l-1})$ be the first differential of the sequence associated to $\bar F_l$. Let $S$ be a component (wedge summand) of $\bar F'_l$ labeled by a classed configuration $H$. Let $\partial_iH$ be the classed configuration obtained by identifying the $i$-th and $i+1$-th geometric points of $H$ (in the order on $\RR$). Let $\partial '_{i,S}:\bar F'_{l-1}\to \bar F'_{l-1}$ be the map collapsing all  components but the one labeled by $\partial_i H$. We define a map $\tilde \partial_i:\bar H_*(\bar F'_l)\to \bar H_*(\bar F'_{l-1})$ by $\tilde \partial_i=(-1)^i(\partial'_{i,S})_*\circ d_1$ on $\bar H_*(S)$. We have 
$\tilde \partial_i\circ \tilde \partial_j=\tilde \partial_{j-1}\circ \tilde \partial_i$ for $i<j$. We use this map to define a functor $Y:\Delta_n^{op}\to \CH_{\kk}$. We set
\[
Y([l])=\bigoplus_{f:[l]\to [k]}(f, \bar H_*(\bar F'_k(F_n(\Sigma))))
\]
where $f$ runs through the order-preserving surjections and $(f, \bar H_*(\bar F'_k(F_n(\Sigma))))$ is a copy of $\bar H_*(\bar F'_k(F_n(\Sigma)))$ regarded as a complex with zero differential. Let $s^i:[l+1]\to [l]$ be the map repeating $i$, and $\partial^i:[l-1]\to [l]$ the map skipping $i$. We define a degeneracy map $s_i:Y([l])\to Y([l+1])$ by $s_i(f,x)=(f\circ s^i,x)$ and a face map $\partial_i:Y([l])\to Y([l-1])$ by
\[
\partial_i(f,x)=
\left\{
\begin{array}{ll}
(f\circ \partial^i, x) & \text{if $f\circ\partial^i$ is surjective,} \\
(f', \tilde \partial_{f(i)}x) & \text{otherwise},
\end{array}
\right.
\]
where $f':[l-1]\to [k-1]$ is the map such that $f\circ \partial^i=\partial^{f(i)}\circ f'$. Verification of the simplicial identity for $s_i,\partial_i$ is routine and the functor $Y$ is well-defined (this is a left Kan extension). Let $S_1$ be the component  of $\bar F_l'(\bar \Sigma)$ labeled with a partitional configuration $(A_0,H_0)$. The map $\phi_0:\bar \Sigma\to \Sigma$ induces a homeomorphism from $S_1$ to a component of $\bar F_l'(F_n(\Sigma))$ or the constant map from $S_1$, according to whether $(A_0,H_0)$ is regular or not. By this observation, the map between $E_1$-pages of the spectral sequence associated to the filtration $\bar F_l$, induced by $\phi_0:\bar \Sigma_n\to F_n(\Sigma)$, is identified with $N\CECHF \mathcal{F}^*Y\to NY$. Therefore,  by Lemma \ref{Lcofinal}, $\phi_0$ induces an isomorphism ${}^\FFF\! E_r(\bar \Sigma)\to {}^\FFF\! E_r(F_n(\Sigma))$ for $r\geq 1$.  The composition of this map and the  defining map ${}^\FFF E_r(F_n(\Sigma))\to F^E_{n,r}$ is a surjection ${}^{\FFF}E_r(\bar \Sigma)\to F^E_{n,r}$ from the $E_1$-page. On the other hand, the maps $\psi, \phi_1,\dots, \phi_7$ induce a map $\bar \phi: {}^\FFF E_r(\bar \Sigma) \to {}^\FFF E_r(U_\TT)$ which is an isomorphism for the part $p\geq -6n/5$ of the first page $E_1^{p,*}$, which can be proved similarly to the proof of Proposition \ref{Phomology_U_1} swapping the complexity and cardinality filtrations (see Remark \ref{Rfiltration}). The $E_1$ -page of ${}^\FFF E_r(U_\TT)$ is isomorphic to the $E_2$-page of the  $n$-truncated  Sinha sequence up to the standard degree shift (including the unstable range) since the sequence for  the cardinality filtration of  $U_\TT$ is isomorphic to Sinha's one by the results in section \ref{Stranslation}. The truncated  Sinha sequence degenerates at $E_2$ for $\kk=\QQ$ so ${}^\FFF E_r(U_\TT)$ degenerates at $E_1$. These facts imply that the differentials $d_r$ of ${}^\FFF E_r(\bar \Sigma)$  going into the stable range are zero for $r\geq 1$ so are the differentials of $F^E_{n,r}$.
\end{proof}


\begin{thebibliography}{99}

\bibitem{bar-natan} D. Bar-Natan, {\em On the Vassiliev knot invariants}, Topology 
{\bf 34} (1995) no.2, 423–472.

\bibitem{BL} J. S. Birma and X.-S. Lin, 
{\em Knot polynomials and Vassiliev's invariants}, 
Invent. Math. {\bf 111} (1993)  no. 2, 225–270.
\bibitem{BH1} P. Boavida de Brito and G. Horel, {\em
On the formality of the little disks operad in positive characteristic}, 
J. Lond. Math. Soc. (2) {\bf 104} (2021) no. 2, 634–667.
\bibitem{BH} P. Boavida de Brito and G. Horel, 
{\em Galois symmetries of knot spaces}, Compos. Math. {\bf 157} (2021) no.5, 997–1021.
\bibitem{browder} W. Browder, {\em Surgery on simply-connected manifolds}, Ergebnisse der Mathematik und ihrer Grenzgebiete, Band 65. Springer-Verlag, New York-Heidelberg, (1972) ix+132 pp. 
\bibitem{BCSS} R. Budney, J. Conant,  K. P.  Scannell, and D. Sinha, {\em New perspectives on self-linking},
 Adv. Math. {\bf 191}  (2005)  no.1, 78–113. 
\bibitem{BCKS} R. Budney,  J. Conant, R. Koytcheff, and D. Sinha, 
{\em Embedding calculus knot invariants are of finite type},
Algebr. Geom. Topol. {\bf 17} (2017) no. 3, 1701–1742.
\bibitem{CS} M. Ching and P. Salvatore, 
{\em Koszul duality for topological En-operads},
Proc. Lond. Math. Soc. (3) {\bf 125} (2022) no. 1, 1–60.
\bibitem{cohen} R. L. Cohen, {\em Multiplicative properties of Atiyah duality}, Homology Homotopy Appl.  {\bf 6}  (2004)  no. 1, 269-281. 
\bibitem{conant} J. Conant, {\em Homotopy approximations to the space of knots, Feynman diagrams, and a
 conjecture of Scannell and Sinha},  Amer. J.  Math. {\bf 130}  (2008) no.2, 341–357.

\bibitem{CCL} A. S. Cattaneo, P. Cotta-Ramusino and R. Longoni, {\em
Configuration spaces and Vassiliev classes in any dimension},
Algebr. Geom. Topol. {\bf 2} (2002) 949–1000.
\bibitem{CR} A. S. Cattaneo and C. A. Rossi, {\em
Wilson surfaces and higher dimensional knot invariants},
Comm. Math. Phys. {\bf 256} (2005) no. 3, 513–537.
\bibitem{FT}
Y. Felix and J.-C. Thomas, {\em Configuration spaces and Massey products}, Int. Math. Res. Not.  (2004)  no. 33, 1685-1702.
\bibitem{goodwillie} T. G. Goodwillie, {\em
A multiple disjunction lemma for smooth concordance embeddings},
Mem. Amer. Math. Soc. {\bf 86} (1990) no. 431, viii+317 pp.

\bibitem{GK} T. G. Goodwillie and J. Klein, {\em Multiple disjunction for space of smooth embeddings}, J. of
Topology {\bf 8} (2015) 675-690.

\bibitem{GW} T. G. Goodwillie and M. S. Weiss, {\em Embeddings from the point of view of immersion theory:
Part II}, Geometry \& Topology {\bf 3} (1999) 103-118.
\bibitem{gusarov} M. N. Gusarov, 
 {\em On n-equivalence of knots and invariants of finite degree}, In: Topology of manifolds and varieties.  Adv. Soviet Math. {\bf 18}, Amer. Math. Soc., Providence, RI, (1994)
 173–192.


\bibitem{hirschhorn} P. S. Hirschhorn, {\em Model categories and their localizations}, Mathematical Survey \& Monographs, {\bf 99}, Amer. Math. Soc. Providence, RI, (2003) xvi+457 pp.


\bibitem{kontsevich}

M. Kontsevich, {\em Operads and motives in deformation quantization}, Mosh\'e Flato (1937-1998). Lett. Math. Phys. {\bf 48} (1999) no. 1, 35-72. 

\bibitem{kontsevich1}
M. Kontsevich, {\em Vassiliev's knot invariants}, I. M. Gel'fand Seminar,
Adv. Soviet Math. {\bf 16}, Part 2, Amer. Math. Soc., Providence, RI, (1993) 137-150.



\bibitem{kosanovic} 
D. Kosanović, 
{\em Embedding calculus and grope cobordism of knots}, 
Adv. Math. {\bf 451} (2024) Paper No. 109779, 118 pp.
\bibitem{LTV} P. Lambrechts, V. Turchin, and I. Voli\'c, {\em  The rational homology of spaces of long knots in codimension $>$2},  Geom. Topol.  {\bf 14}  (2010)  no. 4, 2151-2187. 
\bibitem{livernet} M. Livernet, {Non-formality of the Swiss-cheese operad,} J. Topol. {\bf 8} (2015) no.4, 1156-1166.
\bibitem{LV} P. Lambrechts and I. Voli\'c,  {\em Formality of the little $N$-disks operad},  Mem. Amer. Math. Soc. {\bf 230}  (2014) no. 1079, viii+116 pp. 


\bibitem{malin} 
C. Malin, 
{\em The stable embedding tower and operadic structures on configuration spaces},
Homology Homotopy Appl. {\bf 26} (2024) no. 1, 229–258.
\bibitem{marino} A. Marino, {\em A Fox-Neuwirth Basis for the Sinha Spectral Sequence}, preprint, arXiv:2505.22958 (2025).
\bibitem{MS}  A. Marino and P. Salvatore, {\em Non collapse of the  Sinha spectral sequence for knots in $\RR^3$}, arXiv:2504.16785 (2025).
\bibitem{moriya} S. Moriya, {\em Multiplicative formality of operads and Sinha's spectral sequence for long knots},  Kyoto J. Math. {\bf 55} (2015) no. 1, 17–27.

\bibitem{moriya1}S. Moriya, {\em Non-formality of the odd dimensional framed little balls operads}, Int. Math. Res. Not.   (2019) no. 2, 625–639.
\bibitem{moriya2}S. Moriya, {\em Models for knot spaces and Atiyah duality}, 
Algebr. Geom. Topol. {\bf 24} (2024) no. 1, 183–250.

\bibitem{moriya3} S. Moriya, {\em
Sinha's spectral sequence for long knots in codimension one and non-formality of the little 2-disks operad},
Q. J. Math. {\bf 75} (2024) no. 3, 1073–1121.

\bibitem{sakai1} K. Sakai, 
{\em An integral expression of the first nontrivial one-cocycle of the space of long knots in $\RR^3$},
Pacific J. Math. {\bf 250} (2011) no. 2, 407–419.
\bibitem{SW} K. Sakai and  T. Watanabe, {\em
1-loop graphs and configuration space integral for embedding spaces},
Math. Proc. Cambridge Philos. Soc. {\bf 152} (2012) no. 3, 497–533.

\bibitem{SS} K. P. Scannell and D. P. Sinha, {\em A one-dimensional embedding complex}, J. Pure Appl.
 Algebra {\bf 170} (2002) no.1, 93–107.
\bibitem{sinha2} D. P. Sinha, {\em Manifold-theoretic compactifications of configuration spaces}, 
Selecta Math. (N.S.)  {\bf 10}  (2004)  no. 3, 391–428.
\bibitem{sinha} D. P.  Sinha, 
{\em The topology of spaces of knots: cosimplicial models}, 
Amer. J. Math.  {\bf 131}  (2009)  no. 4, 945-980. 
\bibitem{sinha1} D. P. Sinha, {\em Operads and knot spaces}, J.  Amer. Math. Soc. {\bf 19} (2006) no.2, 461-486.
 
\bibitem{tsopmene}
P. A. Songhafouo Tsopm\'en\'e, 
{\em Formality of Sinha's cosimplicial model for long knots spaces and the Gerstenhaber algebra structure of homology}, Algebr. Geom. Topol. {\bf 13} (2013) no.4, 2193–2205.


\bibitem{turchin} V. Turchin (Tourtchine), {\em On the other side of the bialgebra of chord diagrams}, J. Knot Theory Ramifications {\bf 16} (2007) no. 5, 575–629.
\bibitem{vassiliev}
V. A. Vassiliev, 
{\em Cohomology of knot spaces,} Theory of singularities and its applications, 23–69.
Adv. Soviet Math., 1
American Mathematical Society, Providence, RI, (1990).
\bibitem{vassiliev1} V. A. Vassiliev, {\em Complements of discriminants of smooth maps: topology and applications},
Translated from the Russian by B. Goldfarb
Transl. Math. Monogr. {\bf 98}
American Mathematical Society, Providence, RI, 1992. vi+208 pp.
\bibitem{vassiliev2} V. A. Vassiliev, 
{\em Combinatorial formulas for cohomology of knot spaces},
Mosc. Math. J. {\bf 1} (2001) no. 1, 91–123.
\bibitem{volic}

I. Volić, 
{\em Finite type knot invariants and the calculus of functors},
Compos. Math. {\bf 142} (2006) no. 1, 222–250.
\bibitem{weiss}  M. Weiss, {\em Embeddings from the point of view of immersion theory: Part I}, Geometry \& Topology {\bf 3} (1999) no.1, 67-101.


\end{thebibliography}
\end{document}